\pdfoutput=1 
\documentclass[11pt]{amsart}
\usepackage[utf8]{inputenc}
\usepackage{amssymb,amsmath,amsthm,enumerate,enumitem,colonequals,mlmodern,tikz-cd,microtype,dirtytalk}
\usepackage[mathcal]{eucal}
\usepackage[top=3.5cm, bottom=3cm, left=3cm, right=3cm]{geometry}

\makeatletter
\@namedef{subjclassname@2020}{\textup{2020} Mathematics Subject Classification}
\makeatother

\usepackage{hyperref}
\hypersetup{
    colorlinks=true,
    linkcolor= blue,
    urlcolor= blue,
    citecolor= green,
}

\def\eqref#1{\textcolor{darkblue}{(\ref{#1})}}

\PassOptionsToPackage{hyphens}{url}
\usepackage{hyperref}
\hypersetup{bookmarksdepth=2}

\usepackage[nameinlink]{cleveref}
\Crefformat{section}{#2\S#1#3} 
\Crefmultiformat{section}{#2\S\S#1#3}{ and~#2#1#3}{, #2#1#3}{, and~#2#1#3}
\Crefname{hypothesis}{hypothesis}{hypotheses}
\Crefname{hypothesis}{Hypothesis}{Hypotheses}

\crefname{equation}{}{}
\Crefname{equation}{}{}
\crefname{item}{}{}
\Crefname{item}{}{}

\numberwithin{equation}{subsection}

\newcommand{\hooklongrightarrow}{\lhook\joinrel\longrightarrow}

\newcommand{\Hom}[4][]{\operatorname{Hom}^{#1}_{#2}(#3,#4)}
\newcommand{\HomT}[3][]{\operatorname{Hom}^{#1}_{\mathsf{T}}(#2,#3)}

\newcommand{\Yon}[2]{\mathcal{Y}(#1)|_{\mathcal{G}_{#2}}}

\newcommand{\hocolim}{\operatornamewithlimits{\underset{\boldsymbol{\xrightarrow{\hspace{0.85cm}}}}{{Hocolim}}}}
\makeatletter
\newcommand{\colim@}[2]{%
  \vtop{\m@th\ialign{##\cr
    \hfil$#1\operator@font colim$\hfil\cr
    \noalign{\nointerlineskip\kern1.5\ex@}#2\cr
    \noalign{\nointerlineskip\kern-\ex@}\cr}}%
}
\newcommand{\colim}{%
  \mathop{\mathpalette\colim@{\rightarrowfill@\textstyle}}\nmlimits@
}
\makeatother

\usepackage[pagewise]{lineno}
\let\oldequation\equation
\let\oldendequation\endequation
\renewenvironment{equation}{\linenomathNonumbers\oldequation}{\oldendequation\endlinenomath}
\expandafter\let\expandafter\oldequationstar\csname equation*\endcsname
\expandafter\let\expandafter\oldendequationstar\csname endequation*\endcsname
\renewenvironment{equation*}{\linenomathNonumbers\oldequationstar}{\oldendequationstar\endlinenomath}
\let\oldalign\align
\let\oldendalign\endalign
\renewenvironment{align}{\linenomathNonumbers\oldalign}{\oldendalign\endlinenomath}
\expandafter\let\expandafter\oldalignstar\csname align*\endcsname
\expandafter\let\expandafter\oldendalignstar\csname endalign*\endcsname
\renewenvironment{align*}{\linenomathNonumbers\oldalignstar}{\oldendalignstar\endlinenomath}


\newcounter{HypCounter}

\newtheorem{theoremintro}{Theorem}

\theoremstyle{definition}
\newtheorem{definitionintro}{Definition}

\newtheorem{exampleintro}{Example}

\newtheorem{ideaintro}{Idea}

\newtheorem*{remarkintro}{Remark}

\theoremstyle{plain}
\newtheorem{theorem}{Theorem}[section]
\newtheorem*{theorem*}{Theorem}
\newtheorem{lemma}[theorem]{Lemma}
\newtheorem{corollary}[theorem]{Corollary}
\newtheorem{proposition}[theorem]{Proposition}

\theoremstyle{definition}

\newtheorem{convention}[theorem]{Convention}
\newtheorem{definition}[theorem]{Definition}
\newtheorem*{definition*}{Definition}
\newtheorem{example}[theorem]{Example}

\newtheorem{hypothesis}[HypCounter]{Hypothesis}
\newtheorem*{hypothesis*}{Hypothesis}
\newtheorem{notation}[theorem]{Notation}

\newtheorem{remark}[theorem]{Remark}
\newtheorem{setup}[theorem]{Setup}

\numberwithin{equation}{section}
\numberwithin{theorem}{section}

\title{Representability theorems via metric techniques}

\author[K. ~Manali Rahul]{Kabeer Manali Rahul}
\address{K. ~Manali Rahul,
Center for Mathematics and its Applications, 
Mathematical Science Institute, Building 145, 
The Australian National University, 
Canberra, ACT 2601, Australia}
\email{kabeer.manalirahul@anu.edu.au}

\date{\today}


\begin{document}
    
\begin{abstract}
    We prove new Brown representability theorems for triangulated categories using metric techniques as introduced in the work of Neeman. In the setting of algebraic geometry, this gives us new representability theorems for homological and cohomological functors on the bounded derived category of coherent sheaves. 
    
    To prove this result, we introduce the notion of a $\mathcal{G}$-approximable triangulated category, which generalises the notion of an approximable triangulated category. 
\end{abstract}

\maketitle

\section{Introduction}
The importance and the ubiquity of triangulated categories in modern mathematics cannot be overstated. Neeman has recently initiated the use of metrics, and approximations via metrics, to study triangulated categories. These have been used to prove many important results, including three conjectures in algebraic geometry, see \cite{Neeman:2021a,Neeman:2024}. The notion of a metric allows us to utilize analytic ideas, such as that of completions of metric spaces, to study triangulated categories.

The idea of a metric, which goes back to \cite{Lawvere:1973,Lawvere:2002}, is that it assigns lengths to morphisms, and the length function satisfies the triangle inequality. In the world of triangulated categories, the metrics we are interested in are translation invariant, that is, the length of a morphism is the same as the length of the unique map from the zero object to the cone of the morphism, see \cite[Heuristic 9]{Neeman:2020}. 

Even though the theory is still in its infancy, it has already been used by Neeman and others to prove many remarkable results, see for example \cite{Neeman:2021a,Neeman:2018,Neeman:2024,Canonaco/Neeman/Stellari:2024,Neeman:2021b,Biswas/Chen/ManaliRahul/Parker/Zheng:2024,DeDeyn/Lank/ManaliRahul:2024a,DeDeyn/Lank/ManaliRahul:2024b}.

In this work, we add to this growing theory in the following ways:
\begin{enumerate}
    \item We introduce the notion of a (weakly) $\mathcal{G}$-approximable triangulated category, which is a generalisation of a (weakly) approximable triangulated category. Further, we show that most of the basic results on approximable triangulated categories have corresponding $\mathcal{G}$-approximable analogues.
    \item Using the notion of $\mathcal{G}$-approximability, we prove new representability results. As an application, we get new representability theorems in algebraic geometry.
\end{enumerate}

We go over each of these aspects in more detail in the following subsections.
\subsection[$\mathcal{G}$-approximability]{$\mathcal{G}$-approximable triangulated categories}

Neeman formally introduced the notion of (weakly) approximable triangulated categories in \cite{Neeman:2021b}, although the idea of approximability already appears in \cite{Neeman:2021a}. We state the definition first, and then discuss a simple example which should help in understanding the definition. The definition itself involves some technical notation, which we do not define precisely in the introduction, see \Cref{Notation from Neeman's paper}.
\begin{definition*}\label{Definition of approximability introduction}\cite[Definition 0.21]{Neeman:2021b}
    A triangulated category $\mathsf{T}$ is said to be \emph{weakly approximable} if there exists a compact generator $G$, a positive integer $A$, and a t-structure $(\mathsf{T}^{\leq 0},\mathsf{T}^{\geq 0})$ on $\mathsf{T}$ such that,
    \begin{enumerate}
        \item $\Sigma^A G \in \mathsf{T}^{\leq 0}$ and $\HomT{\Sigma^{-i}G}{\mathsf{T}^{\leq 0}} = 0$ for all $i \geq A$.
        \item For all $F \in \mathsf{T}^{\leq 0}$, there exists a triangle $E \to F \to D \to \Sigma E$ with $E \in \overline{\langle G \rangle}^{[-A,A]}$ and $D \in \mathsf{T}^{\leq -1}$.
    \end{enumerate}
    We say that $\mathsf{T}$ is \emph{approximable}, if further,
    \begin{enumerate}
        \item[(2)$'$] For all $F \in \mathsf{T}^{\leq 0}$, there exists a triangle $E \to F \to D \to \Sigma E$ with $E \in \overline{\langle G \rangle}^{[-A,A]}_{A}$ and $D \in \mathsf{T}^{\leq -1}$.
    \end{enumerate}
\end{definition*}
Here $\overline{\langle G \rangle}^{[-A,A]}$ denotes the full subcategory of $\mathsf{T}$ generated from the object $G$ using coproducts, summands, extensions, and at most $A$ suspensions and desuspensions. Further, $\overline{\langle G \rangle}^{[-A,A]}_{A}$ denotes the full subcategory of $\mathsf{T}$ generated from the object $G$ using coproducts, summands, at most $A$ extensions, and at most $A$ suspensions and desuspensions. The example of the derived category of a ring below should give some intuition on these subcategories, see \Cref{Example to explain approximability}. 

Important examples of approximable triangulated categories to keep in mind are the derived category of a ring, and the derived category of quasicoherent sheaves over a noetherian, separated scheme, see \cite[Example 3.5]{Neeman:2021b}.

The first condition in the definition of approximability is a technical assumption on the compatibility of the compact generator and the t-structure, but it is the other condition which is the heart of this definition. The idea is that we can \say{approximate} any objects in the aisle of the t-structure $\mathsf{T}^{\leq 0}$ by \say{simple objects}, that is, objects built in \say{finitely many steps} from the compact generator, up to an error term which is \say{smaller} than the object we started with. To understand this condition better, we look at the example of the derived category of a ring.

\begin{exampleintro}\label{Example to explain approximability}
    Let $R$ be a ring. We denote by $\mathbf{D}(R)$ the derived category of the abelian category of (right) $R$-modules. It has a compact generator $R$, and the standard t-structure $(\mathbf{D}(R)^{\leq 0},\mathbf{D}(R)^{\geq 0})$. We now discuss the second condition of the definition of approximability in this particular case. Consider a complex $F$ in $\mathbf{D}(R)^{\leq 0}$, that is, its positive cohomologies vanish. Then, up to quasi-isomorphism, we can write the complex as,
    \[\cdots \to P^{-n-2} \to P^{-n-1}\to P^{-n}\to \cdots \to P^{-1} \to P^0 \to 0 \to \cdots\]
    with $P^{-i}$ projective modules for all $i \geq 0$. We get the following commutative diagram for any $n \geq 0$,
    \[\begin{tikzcd}[column sep=18]
        \cdots & 0 & 0 & P^{-n} & \cdots & P^{-1} & P^{0} & 0 & \cdots \\
        \cdots & P^{-n-2} & P^{-n-1} & P^{-n} & \cdots & P^{-1} & P^{0} & 0 & \cdots \\
        \cdots & P^{-n-2} & P^{-n-1} & 0 & \cdots & 0 & 0 & 0 & \cdots
	   \arrow[from=1-1, to=1-2]
	   \arrow[from=1-2, to=2-2]
	   \arrow[from=1-2, to=1-3]
	   \arrow[from=1-3, to=1-4]
	   \arrow[from=1-3, to=2-3]
	   \arrow[from=1-4, to=1-5]
	   \arrow[from=1-4, to=2-4]
	   \arrow[from=1-5, to=1-6]
	   \arrow[from=1-6, to=1-7]
	   \arrow[from=1-6, to=2-6]
	   \arrow[from=1-7, to=1-8]
	   \arrow[from=1-7, to=2-7]
	   \arrow[from=1-8, to=1-9]
	   \arrow[from=1-8, to=2-8]
	   \arrow[from=2-1, to=2-2]
	   \arrow[from=2-2, to=2-3]
	   \arrow[from=2-2, to=3-2]
	   \arrow[from=2-3, to=2-4]
	   \arrow[from=2-3, to=3-3]
	   \arrow[from=2-4, to=2-5]
	   \arrow[from=2-4, to=3-4]
	   \arrow[from=2-5, to=2-6]
	   \arrow[from=2-6, to=2-7]
	   \arrow[from=2-6, to=3-6]
	   \arrow[from=2-7, to=2-8]
	   \arrow[from=2-7, to=3-7]
	   \arrow[from=2-8, to=2-9]
	   \arrow[from=2-8, to=3-8]
	   \arrow[from=3-1, to=3-2]
	   \arrow[from=3-2, to=3-3]
	   \arrow[from=3-3, to=3-4]
	   \arrow[from=3-4, to=3-5]
	   \arrow[from=3-5, to=3-6]
	   \arrow[from=3-6, to=3-7]
	   \arrow[from=3-7, to=3-8]
	   \arrow[from=3-8, to=3-9]
    \end{tikzcd}\]
    This gives us maps of complexes, $E_n \to F \to D_n$, with $E_n$ a bounded complex of projectives, and $D_n$ with vanishing cohomology for all degrees greater than or equal to $-n$. As the maps form a degree-wise split exact sequence, $E_n \to F \to D_n$ extends to a distinguished triangle, $E_n \to F \to D_n \to \Sigma E_n$. The complexes $D_n$ represent \say{error terms} of sorts. They get smaller and smaller, in the sense that $D_{n+1} \in \mathbf{D}(R)^{\leq -n-2} \subseteq \mathbf{D}(R)^{\leq -n-1}$.

    What remains to show is that the complexes $E_n$ are built from the compact generator $R$ in finitely many steps. The object $E_n$ is a complex of projective modules in degrees $-n$ to $0$. Let $P$ be any projective module. Then, it is a summand of a free module, that is, it is built from $R$ by direct sums and summands. Now, for any complex $E_{n+1}$ of projective modules in degrees $-n-1$ to $0$, we have the following commutative diagram,
    \[\begin{tikzcd}
	   \cdots & 0 & 0 & {P^n} & \cdots & {P^0} & 0 & \cdots \\
	   \cdots & 0 & {P^{-n-1}} & {P^n} & \cdots & {P^0} & 0 & \cdots \\
        \cdots & 0 & {P^{-n-1}} & 0 & \cdots & 0 & 0 & \cdots
	   \arrow[from=1-1, to=1-2]
	   \arrow[from=1-2, to=1-3]
	   \arrow[from=1-2, to=2-2]
	   \arrow[from=1-3, to=1-4]
	   \arrow[from=1-3, to=2-3]
	   \arrow[from=1-4, to=1-5]
	   \arrow[from=1-5, to=1-6]
	   \arrow[from=1-4, to=2-4]
          \arrow[from=1-6, to=1-7]
          \arrow[from=1-7, to=1-8]
          \arrow[from=1-7, to=2-7]
          \arrow[from=1-6, to=2-6]
	   \arrow[from=2-1, to=2-2]
	   \arrow[from=2-2, to=2-3]
	   \arrow[from=2-2, to=3-2]
	   \arrow[from=2-3, to=2-4]
	   \arrow[from=2-3, to=3-3]
	   \arrow[from=2-4, to=2-5]
	   \arrow[from=2-5, to=2-6]
          \arrow[from=2-6, to=2-7]
          \arrow[from=2-6, to=3-6]
          \arrow[from=2-7, to=2-8]
          \arrow[from=2-7, to=3-7]
	   \arrow[from=2-4, to=3-4]
	   \arrow[from=3-1, to=3-2]
	   \arrow[from=3-2, to=3-3]
	   \arrow[from=3-3, to=3-4]
	   \arrow[from=3-4, to=3-5]
	   \arrow[from=3-5, to=3-6]
          \arrow[from=3-6, to=3-7]
          \arrow[from=3-7, to=3-8]
    \end{tikzcd}\]
    Again, this is degree-wise split, so it extends to a triangle $E_{n} \to E_{n+1} \to \Sigma^{n+1}P^{-n-1} \to \Sigma E_{n}$, where $E_n$ is a complex of projective modules in degrees $-n$ to $0$. So, inductively we get that $E_n$ can be constructed from $R$ using direct sums, summands and at most $n$ shifts and $n$ extensions. This is what we mean by saying that an object is \say{built in finitely many steps from the compact generator $R$}. Note that the argument above for $n=1$ gives us the second condition of approximability for $\mathbf{D}(R)$.
\end{exampleintro}

\begin{ideaintro}\label{Intuition for approximability}
    We see that the notion of approximability requires the following : 
    \begin{enumerate}
        \item A nice set of compacts which generates, which is given by the compact generator $G$.
        \item A notion of \say{smallness}, that is, the notion of a metric, which will allow us to make sense of the \say{error term} being smaller than the object we started with. We get this from the aisle of the t-structure in the case of approximability.
    \end{enumerate}
    We generalise this by 
    \begin{enumerate}
        \item Replacing a single compact generator, by a \emph{generating sequence} $\mathcal{G}$, see \Cref{Definition of generating sequence}
        \item Replacing the t-structure metric, by the more general notion of an \emph{orthogonal metric}, see \Cref{Good Metric}.
    \end{enumerate}
    This helps us define the notion of (weak) $\mathcal{G}$-approximability, see \Cref{Definition of G-approximability}. 
\end{ideaintro}

\begin{remarkintro}\label{Remark introduction on G-quasiapproximability}
    We will in fact define a notion slightly more general than that of $\mathcal{G}$-approximability, namely that of $\mathcal{G}$-quasiapproximability, see \Cref{Definition of G-quasiapproximability}. It turns out that the notion of $\mathcal{G}$-quasiapproximability suffices to prove the representability theorems we want, but for the sake of simplicity, we will only discuss $\mathcal{G}$-approximability in the introduction. Note that the practical reason for considering $\mathcal{G}$-quasiapproximability is that there are examples of $\mathcal{G}$-quasiapproximable triangulated categories, which the author has not been able to prove are $\mathcal{G}$-approximable, see \Cref{Theorem Kmproj co-approx}. 
\end{remarkintro}
To further understand this notion, we now state the definition of a \emph{(weakly) co-approximable} triangulated category. Co-approximability is a special case of $\mathcal{G}$-approximability, which utilises a co-t-structure to define the metric instead of t-structures as in the notion of approximability.
\begin{definitionintro}\label{Definition co-approx intro}[\Cref{Definition of co-approx}]
    A triangulated category $\mathsf{T}$ is said to be \emph{weakly co-approximable} if there exists a compact generator $G$, a positive integer $A$, and a co-t-structure $(\mathsf{T}^{\geq 0},\mathsf{T}^{\leq 0})$ on $T$ such that,
    \begin{enumerate}
        \item $\Sigma^{-A} G \in \mathsf{T}^{\geq 0}$ and $\HomT{\Sigma^{i}G}{\mathsf{T}^{\geq 0}} = 0$ for all $i \geq A$.
        \item For all $F \in \mathsf{T}^{\geq 0}$, there exists a triangle $E \to F \to D \to \Sigma E$ with $E \in \overline{\langle G \rangle}^{[-A,A]}$ and $D \in \mathsf{T}^{\geq 1}$.
    \end{enumerate}
    We say that $\mathsf{T}$ is \emph{co-approximable}, if further,
    \begin{enumerate}
        \item[(2)$'$] For all $F \in \mathsf{T}^{\geq 0}$, there exists a triangle $E \to F \to D \to \Sigma E$ with $E \in \overline{\langle G \rangle}^{[-A,A]}_{A}$ and $D \in \mathsf{T}^{\geq 1}$.
    \end{enumerate}
\end{definitionintro}
\begin{remarkintro}
    One of the motivations to define $\mathcal{G}$-approximability is that certain naturally occurring triangulated categories are not (weakly) approximable. One such important example is the homotopy category of injectives $\mathbf{K}(\operatorname{Inj}\text{-}X)$ for a noetherian scheme $X$, which we can easily show cannot be weakly approximable if $X$ is not regular. This category agrees with the category of \emph{ind-coherent sheaves} $\mathbf{IndCoh}(X)$, see \cite[Proposition A.1]{Krause:2005}. It turns out that it is more natural to work with co-t-structures in this situation.
\end{remarkintro}
The following result shows that $\mathbf{K}(\operatorname{Inj}\text{-}X)$ is (weakly) co-approximable for a large class of schemes.
\begin{theoremintro}\label{Theorem Intro Examples of co-approximability}
    Let $X$ be a noetherian finite-dimensional scheme. Then,
    \begin{itemize}
        \item If $X$ is a J-2 scheme, then $\mathbf{K}(\operatorname{Inj}\text{-}X)$ is weakly co-approximable, see \cite[Theorem 4.13]{ManaliRahul:2025}. In fact, it is enough to assume that the regular loci of all integral closed subschemes of $X$ are open.
        \item $\mathbf{K}(\operatorname{Inj}\text{-}X)$ is co-approximable and $\mathbf{K}_m(\operatorname{Proj}\text{-}X)$ is co-quasiapproximable for any of the following cases,
        \begin{enumerate}
            \item $X$ is regular.
            \item $X$ is separated and quasiexcellent.
        \end{enumerate}
        see \Cref{Theorem Kinj co-approx}, \Cref{Theorem Kmproj co-approx}, and \Cref{Theorem Examples of co-approximability}. 
    \end{itemize}
\end{theoremintro}

Similar result holds for the mock homotopy category of projectives, see \Cref{Theorem Kmproj co-approx} and \Cref{Theorem Examples of co-approximability}. Further, noncommutative and stacky versions of the result also hold, see \Cref{Remark Kinj co-approx for stacks and algebras}, \Cref{Theorem Examples of co-approximability}, and \Cref{Theorem Examples of co-approximability for stacks}.
\subsection{Representability results}
One of the most important applications of the theory of approximable triangulated categories are the representability theorems proved in \cite{Neeman:2021b}. We first give an idea of what we mean by representability. Let $\mathsf{T}$ be a $R$-linear triangulated category for a commutative ring $R$. Let $\mathsf{S}$ be a triangulated subcategory. A cohomological functor $H : \mathsf{S}^{\operatorname{op}} \to \operatorname{Mod}(R)$ is \emph{representable} if there exists an object $t \in \mathsf{T}$ such that $H(-) \cong \HomT{-}{t}|_{\mathsf{S}}$. There is a long history of representability theorems, that is, results which prove that $H$ is representable under conditions on $\mathsf{T}$, $\mathsf{S}$, and $H$. We recall a few of them for the readers convenience, although this is by no means an exhaustive list. One of the first results of this flavour was by Brown, in the context of algebraic topology, see \cite[Theorem I]{Brown:1962}. This was vastly generalised to the context of (co)product preserving functors on compactly generated triangulated categories by Neeman, see \cite[Theorem 3.1]{Neeman:1996}. This was even further generalised to the context of well generated triangulated categories, see \cite[Theorem 1.17]{Neeman:2001}. Note that these result requires the existence of arbitrary coproducts in the triangulated category. When the triangulated category does not have all coproducts, typically more stringent conditions are required on the category. Two important results in this situation are the theorem by Bondal and Van den Bergh, see \cite[Theorem 1.3]{Bondal/VanDenBergh:2003}, and the improvement on their result by Rouquier, see \cite[Theorem 4.16]{Rouquier:2008}. Both of these results require that the category $\mathsf{S}$ be strongly generated. 

With this background, we now come to the result by Neeman in \cite{Neeman:2021b}, and the result proven in this work.
To state them, we first need to discuss the notion of the \say{closure of compacts}.
\begin{ideaintro}
    As we have discussed, approximability (and $\mathcal{G}$-approximability) involves a \emph{metric} on a category, which is used to make sense of \say{approximations}. Given the notion of a  metric, it makes sense to talk about the \emph{closure} of the compact objects with respect to a metric. The closure of the compacts will be denoted by $\overline{\mathsf{T}^c}$ when in the general context, borrowing the notation for the closure from point set topology. The notation will be $\mathsf{T}^-_c$ in the context of approximability (and in general when the metric arises from a t-structure), and $\mathsf{T}^+_c$ when the metric comes from a co-t-structure. In all the cases we are interested in, there is a notion of bounded objects, and the full subcategory of all bounded objects in the closure of the compacts will be denoted by $\mathsf{T}^b_c$.

    In the example of the derived category of a ring we discussed before, these give rise to familiar categories. For $\mathsf{T} = \mathbf{D}(R)$ for a (right) noetherian ring $R$, $\mathsf{T}^-_c = \mathbf{D}^{-}(\textbf{mod}(R))$ and $\mathsf{T}^b_c = \mathbf{D}^{b}(\textbf{mod}(R))$ where $\mathbf{mod}(R)$ denotes the abelian category of finitely generated (right) $R$-modules, see \cite[Example 3.1]{Neeman:2021b}.
    In the context of algebraic geometry, the following example is also useful to keep in mind. For $\mathsf{T} = \mathbf{D}_{\operatorname{Qcoh}}(X)$ for a noetherian scheme $X$, $\mathsf{T}^-_c = \mathbf{D}^{-}_{\operatorname{coh}}(X)$ and $\mathsf{T}^b_c = \mathbf{D}^{b}_{\operatorname{coh}}(X)$, see \cite[Example 3.4]{Neeman:2021b}. 
\end{ideaintro}
With the notion of the \say{closure of compacts}, we are now in position to explain the main representability result of \cite{Neeman:2021b}. For the rest of this section $R$ denotes a commutative ring. The following is the result in question, after omitting some of the technicalities for the ease of understanding.
\begin{theorem*}\cite[Theorem 0.3]{Neeman:2021b}
    For a compactly generated $R$-linear triangulated category $\mathsf{T}$, we define the restricted Yoneda functor $\mathcal{Y} : \mathsf{T} \to \Hom{}{[\mathsf{T}^c]^{\operatorname{op}}}{\operatorname{Mod}(R)}$ by sending an object $F$ to the functor $\HomT{-}{F}|_{\mathsf{T}^c}$. Then, if $\mathsf{T}$ is a \say{nice} approximable triangulated category, we have that, 
    \begin{itemize}
        \item The restriction of $\mathcal{Y}$ to the closure of the compacts $\mathsf{T}^-_c$ is full, and we have an explicit description of its essential image.
        \item The restriction of $\mathcal{Y}$ to the bounded objects in the closure of the compacts $\mathsf{T}^b_c$ is fully faithful, and we have an explicit description of its essential image.
    \end{itemize}
    Note that as the explicit description of the essential image is known, if we have a cohomological functor on $\mathsf{T}^c$ satisfying those conditions, we know that it is representable. 
\end{theorem*}
In this work, we prove the obvious generalisation of this theorem to the setting of $\mathcal{G}$-approximable triangulated categories, see \Cref{Main Theorem on Tc-} and \Cref{Main Theorem on Tbc}. We state the results below, again omitting some of the technicalities.
\begin{theoremintro}
    Let $\mathsf{T}$ be a \say{nice} $\mathcal{G}$-approximable $R$-linear triangulated category. Then,
    \begin{itemize}
        \item The restriction of $\mathcal{Y}$ to the closure of the compacts $\overline{\mathsf{T}^c}$ is full, and we have an explicit description of its essential image.
        \item The restriction of $\mathcal{Y}$ to the bounded objects in the closure of the compacts $\mathsf{T}^b_c$ is fully faithful, and we have an explicit description of its essential image.
    \end{itemize}
    Note that the closure of the compacts here is computed with respect to the metric in the definition of $\mathcal{G}$-approximability.
\end{theoremintro}
As concrete applications of this abstract result, we have the following two results.
\begin{theoremintro}[\Cref{Theorem on representability on DbcohX}]\label{Introduction Theorem on representability on DbcohX}
    Let $X$ be a quasiexcellent finite-dimensional scheme which is proper over a noetherian ring $R$. Let $\mathcal{Y}:\mathbf{D}^{+}_{\operatorname{coh}}(X) \to \operatorname{Hom}([\mathbf{D}^{b}_{\operatorname{coh}}(X)]^{\operatorname{op}},\operatorname{Mod}(R))$ be the restricted Yoneda functor.
    \begin{itemize}
        \item The restricted Yoneda functor $\mathcal{Y}$ is full on $\mathbf{D}^{+}_{\operatorname{coh}}(X)$, and the essential image is all the cohomological functors $H : \mathbf{D}^{b}_{\operatorname{coh}}(X)^{\operatorname{op}} \to \operatorname{Mod}(R)$ such that for any $F \in \mathbf{D}^{b}_{\operatorname{coh}}(X)$, $H(\Sigma^n F)$ is a finitely generated $R$-module for all $n \in \mathbb{Z}$ and vanishes for $n >> 0$.
        \item The functor $\mathcal{Y}$ gives an equivalence between $\mathbf{D}^{b}(\operatorname{Inj}\text{-}X) \colonequals \mathbf{D}^{b}_{\operatorname{coh}}(X) \cap \mathbf{K}^{b}(\operatorname{Inj}\text{-}X)$ and all the cohomological functors $H : \mathbf{D}^{b}_{\operatorname{coh}}(X)^{\operatorname{op}} \to \operatorname{Mod}(R)$ such that for any $F \in \mathbf{D}^{b}_{\operatorname{coh}}(X)$, $H(\Sigma^n F)$ is a finitely generated $R$-module for all $n \in \mathbb{Z}$ and vanishes for $|n| >> 0$.
    \end{itemize}

\end{theoremintro}
Noncommutative and stacky analogues to \Cref{Introduction Theorem on representability on DbcohX} also exist, see \Cref{Remark on representability on DbcohX for stacks} and \Cref{Remark on representability on DbcohA for algebra}.

\begin{theoremintro}[\Cref{Theorem on representability on DbcohXop}]
    Let $X$ be a quasiexcellent finite-dimensional scheme which is proper over a noetherian ring $R$. Let $\mathcal{Y}:\mathbf{D}^{-}_{\operatorname{coh}}(X)^{\operatorname{op}} \to \operatorname{Hom}(\mathbf{D}^{b}_{\operatorname{coh}}(X),\operatorname{Mod}(R))$ be the restricted Yoneda functor. Then,
    \begin{itemize}
        \item The restricted Yoneda functor $\mathcal{Y}$ is full on $\mathbf{D}^{-}_{\operatorname{coh}}(X)^{\operatorname{op}}$, and the essential image is all the homological functors $H : \mathbf{D}^{b}_{\operatorname{coh}}(X) \to \operatorname{Mod}(R)$ such that for any $F \in \mathbf{D}^{b}_{\operatorname{coh}}(X)$, $H(\Sigma^n F)$ is a finitely generated $R$-module for all $n \in \mathbb{Z}$ and vanishes for $n >> 0$.
        \item The functor $\mathcal{Y}$ gives an equivalence between $\mathbf{D}^{\operatorname{perf}}(X)$ and all the homological functors $H : \mathbf{D}^{b}_{\operatorname{coh}}(X) \to \operatorname{Mod}(R)$ such that for any $F \in \mathbf{D}^{b}_{\operatorname{coh}}(X)$, $H(\Sigma^n F)$ is a finitely generated $R$-module for all $n \in \mathbb{Z}$ and vanishes for $|n| >> 0$.
    \end{itemize}

\end{theoremintro}

\subsection*{Acknowledgements} This work is part of the author's PhD work at the Australian National University under the supervision of Amnon Neeman. The author would like to thank him for
generously sharing his knowledge and his help throughout the PhD. The author would also like to thank Janina Letz and Greg Stevenson for some useful discussions regarding this work. Further, the author thanks BIREP group at Universit\"{a}t Bielefeld and the Algebraic Geometry group at Universit\`{a} degli Studi di Milano for their hospitality during the author's stay there.

Finally, the author was supported by an Australian Government Research Training Program scholarship, and was supported under the ERC Advanced
Grant 101095900-TriCatApp, the Deutsche Forschungsgemeinschaft (SFB-TRR 358/1 2023 - 491392403), and the Australian Research Council Grant DP200102537 at different points during the PhD.

\section{Background and notation}
Throughout this work, we will use $\Sigma$ to denote the autoequivalence in the definition of triangulated categories. For any functor $F : \mathsf{S} \to \mathsf{T}$, $F(\mathsf{S})$ will denote the essential image of $\mathsf{S}$ under $F$. Further, all the gradings involved in this work will be cohomological.

We begin by setting up some basic notation.

\begin{notation}\label{Notation star operation}
    Let $\mathsf{T}$ be a triangulated category, with full subcategories $\mathsf{A}$ and $\mathsf{C}$. Then, we define the following full subcategories of $\mathsf{T}$,
    \begin{itemize}
        \item $\mathsf{A}^{\perp} \colonequals \{c \in \mathsf{T} : \HomT{a}{c} = 0 \text{ for all } a \in \mathsf{A}\}$. 
        \item Similarly, $^{\perp}\mathsf{A} \colonequals \{c \in \mathsf{T} : \HomT{c}{a} = 0 \text{ for all } a \in \mathsf{A}\}$
        \item If $\mathsf{C} \subseteq \mathsf{A}^{\perp}$, we write $\HomT{\mathsf{A}}{\mathsf{C}}=0$.
        \item $\mathsf{A} \star \mathsf{C}$ is the full subcategories of all extensions of $\mathsf{A}$ by $\mathsf{C}$, that is, it is given by objects $B$ such that there exists a triangle  $A \to B \to C \to \Sigma A \text{ with } A \in \mathsf{A} \text{ and } C \in \mathsf{C}$
        \item We define $\mathsf{A}^{\star n}$ inductively. $\mathsf{A}^{\star 1} \colonequals \mathsf{A}$ and $\mathsf{A}^{\star (n + 1)} \colonequals \mathsf{A}^{\star n} \star \mathsf{A}$.
    \end{itemize}
    It follows from the octahedral axiom that the $\star$ operation is associative, that is, for any full subcategories $\mathsf{A}$, $\mathsf{B}$, and $\mathsf{C}$ of $\mathsf{T}$, we have $\mathsf{A} \star (\mathsf{B} \star \mathsf{C}) = (\mathsf{A} \star \mathsf{B})\star \mathsf{C}$. 
\end{notation}

We now recall some notation from \cite{Bondal/VanDenBergh:2003,Neeman:2021b} on generating full subcategories from a given set of objects.

\begin{notation}\label{Notation from Neeman's paper}\cite[Reminder 0.8]{Neeman:2021b}\label{Notation Subcategories of a triangulated category}
    Let $\mathsf{T}$ be a triangulated category, and $\mathsf{U}$ a set of objects in $\mathsf{T}$. Then, we define the following full subcategories of $\mathsf{T}$,
    \begin{itemize}
        \item $\operatorname{smd}(\mathsf{U})$ is the closure of $\mathsf{U}$ under direct summands.
        \item $\operatorname{add}(\mathsf{U})$ is the closure of $\mathsf{U}$ under finite direct sums.
        \item For integers $A \leq B$, $\mathsf{U}[A,B]\colonequals \{\Sigma^{-i}U : U \in \mathsf{U} \text{ and } A \leq i \leq B \}$. We extend this definition to $A = -\infty$ and $B = \infty$ in the natural way. For example, we have that $\mathsf{U}(-\infty,B]\colonequals \{\Sigma^{-i}U : U \in \mathsf{U} \text{ and } i \leq B \}$.
        \item  $\operatorname{coprod}_{n}(\mathsf{U}[A,B])$ is defined inductively. For $n=1$, $\operatorname{coprod}_1(\mathsf{U}[A,B]) \colonequals \operatorname{add}(\mathsf{U}[A,B])$. Then, we inductively define $\operatorname{coprod}_{n + 1}(\mathsf{U}[A,B]) \colonequals \operatorname{coprod}_{1}(\mathsf{U}[A,B]) \star \operatorname{coprod}_{n}(\mathsf{U}[A,B])$.
        \item  $\operatorname{coprod}(\mathsf{U}[A,B])$ is the closure of $\mathsf{U}[A,B]$ under finite direct sums and extensions. Note that $\operatorname{coprod}(\mathsf{U}[A,B]) = \bigcup_{n \geq 1} \operatorname{coprod}_{n}(\mathsf{U}[A,B])$.
        \item For any $n \geq 0$, let $\langle\mathsf{U}\rangle^{[A,B]}_{n} \colonequals \operatorname{smd}(\operatorname{coprod}_n(\mathsf{U}[A,B]))$.
        \item $\langle\mathsf{U}\rangle^{[A,B]} \colonequals \operatorname{smd}(\operatorname{coprod}(\mathsf{U}[A,B])$. 
        \item When $A=-\infty$ and $B=\infty$ in the previous definitions, we drop the superscript. That is, $\langle\mathsf{U}\rangle_n \colonequals \operatorname{smd}(\operatorname{coprod}_n(\mathsf{U}(-\infty,\infty)))$ and $\langle\mathsf{U}\rangle \colonequals \operatorname{smd}(\operatorname{coprod}(\mathsf{U}(-\infty,\infty)))$. Note that $\langle\mathsf{U}\rangle$ is the thick triangulated category generated by $\mathsf{U}$.
    \end{itemize}
    If $\mathsf{T}$ has coproducts, we can define the corresponding \say{big} categories too.
    \begin{itemize}
        \item $\operatorname{Add}(\mathsf{U})$ is the closure of $\mathsf{U}$ under arbitrary coproducts.
        \item We first define $\operatorname{Coprod}_1(\mathsf{U}[A,B]) \colonequals \operatorname{Add}(\mathsf{U}[A,B])$. Then, we inductively define $\operatorname{Coprod}_{n + 1}(\mathsf{U}[A,B]) \colonequals \operatorname{Coprod}_{1}(\mathsf{U}[A,B]) \star \operatorname{Coprod}_{n}(\mathsf{U}[A,B])$.
        \item  $\operatorname{Coprod}(\mathsf{U}[A,B])$ is the closure of $\mathsf{U}[A,B]$ under arbitrary coproducts and extensions. 
        \item For any $n \geq 0$, $\overline{\langle \mathsf{U}\rangle}_n^{[A,B]} \colonequals \operatorname{smd}(\operatorname{Coprod}_n(\mathsf{U}[A,B]))$. 
        \item $\overline{\langle U\rangle}^{[A,B]} \colonequals \operatorname{smd}(\operatorname{Coprod}(\mathsf{U}[A,B])$. 
        \item When $A=-\infty$ and $B=\infty$ in the previous definition, we drop the superscript. That is, $\overline{\langle U \rangle}_n \colonequals \operatorname{smd}(\operatorname{Coprod}_n(\mathsf{U}(-\infty,\infty)))$ and $\overline{\langle U \rangle} \colonequals \operatorname{smd}(\operatorname{Coprod}(\mathsf{U}(-\infty,\infty)))$.
    \end{itemize}
\end{notation}
\begin{remark}
    In \Cref{Notation from Neeman's paper}, if $\mathsf{U}=\{G\}$ for an object $G \in \mathsf{T}$, then we denote $\langle U \rangle$ by $\langle G \rangle$, and similarly for the other categories defined in \Cref{Notation from Neeman's paper}. Further,
    \begin{itemize}
        \item An object $G \in \mathsf{T}$ is said to be a \emph{classical generator} for $\mathsf{T}$ if $\mathsf{T} = \langle G \rangle$.
        \item An object $G \in \mathsf{T}$ is said to be a \emph{strong generator} for $\mathsf{T}$ if $\mathsf{T} = \langle G \rangle_n$ for some $n \geq 1$.
        \item The \emph{Rouquier dimension} of a triangulated category $\mathsf{T}$ is the infimum over $n$ such that there exists an object $G \in \mathsf{T}$ with $\mathsf{T} = \langle G \rangle_{n+1}$, see \cite[Definition 3.2]{Rouquier:2008}. 
    \end{itemize}
\end{remark}

We recall the notions of t-structures and co-t-structures now, mostly to fix the notation.

\begin{definition}\cite[D\'{e}finition 1.3.1]{Beilinson/Bernstein/Deligne:1982} Let $\mathsf{T}$ be a triangulated category. A t-structure on $\mathsf{T}$ is a pair of strictly full subcategories $(\mathsf{T}^{\leq 0},\mathsf{T}^{\geq 0})$ such that,
    \begin{enumerate}
        \item $\Sigma \mathsf{T}^{\leq 0} \subseteq \mathsf{T}^{\leq 0}$ and $\mathsf{T}^{\geq 0} \subseteq \Sigma \mathsf{T}^{\geq 0}$. 
        \item $\HomT{\mathsf{T}^{\leq 0}}{\Sigma^{-1}\mathsf{T}^{\geq 0}} = 0$.
        \item For all $t \in \mathsf{T}$, there exists a triangle $u \to t \to v \to \Sigma u$ with $u \in \mathsf{T}^{\leq 0}$ and $v \in \Sigma^{-1}\mathsf{T}^{\geq 0}$.
    \end{enumerate}
    By \cite[Proposition 1.3.3]{Beilinson/Bernstein/Deligne:1982}, the triangle in (3) is unique.
    
    We define $\mathsf{T}^{\leq n} \colonequals \Sigma^{-n}\mathsf{T}^{\leq 0}$ and $\mathsf{T}^{\geq n} \colonequals \Sigma^{-n}\mathsf{T}^{\geq 0}$.
        
    We shall sometimes also denote a t-structure by the pair $(\mathsf{U}, \mathsf{V})$, with $\mathsf{U}=\mathsf{T}^{\leq 0}$ and $\mathsf{V}=\mathsf{T}^{\geq 0}$. Further, we define the \emph{heart of the t-structure} to be $\mathsf{T}^{\heartsuit} \colonequals \mathsf{T}^{\leq 0} \cap \mathsf{T}^{\geq 0}$. Note that the heart of any t-structure is an abelian category, see \cite[Th\'eor\`eme 1.3.6]{Beilinson/Bernstein/Deligne:1982}.
\end{definition}

The following is an important method of constructing t-structures.
\begin{theorem}\label{Theorem Compactly generated t-structure}
    \cite[Theorem A.1]{AlonsoTarrio/JeremiasLopez/SoutoSalorio:2003}Let $\mathsf{C}$ be a set of compact objects in a triangulated category $\mathsf{T}$ with coproducts. Then, $(\operatorname{Coprod}(\cup_{i\geq 0}\Sigma^i\mathsf{C}),(\cup_{i\geq 1}\Sigma^i\mathsf{C})^{\perp})$ is a t-structure on $\mathsf{T}$, see \Cref{Notation from Neeman's paper}. We call this t-structure the \emph{t-structure generated by $\mathsf{C}$}, and such t-structures are called \emph{compactly generated t-structures}.
\end{theorem}
We now define a co-t-structure or a weight structure. This notion was defined independently by Paukztello and Bondarko around the same time.
\begin{definition}\cite[Definition 1.1.1]{Bondarko:2010} and \cite[Definition 1.4]{Pauksztello:2008}
     Let $\mathsf{T}$ be a triangulated category. A co-t-structure on $\mathsf{T}$ is a pair of strictly full subcategories $(\mathsf{T}^{\geq 0},\mathsf{T}^{\leq 0})$ closed under direct summands such that,
    \begin{enumerate}
        \item $\Sigma \mathsf{T}^{\leq 0} \subseteq \mathsf{T}^{\leq 0}$ and $\mathsf{T}^{\geq 0} \subseteq \Sigma \mathsf{T}^{\geq 0}$.
        \item $\HomT{\mathsf{T}^{\geq 0}}{\Sigma\mathsf{T}^{\leq 0}} = 0$.
        \item For all $t \in \mathsf{T}$, there exists a triangle $u \to t \to v \to \Sigma u$ with $u \in \mathsf{T}^{\geq 0}$ and $v \in \Sigma\mathsf{T}^{\leq 0}$.
    \end{enumerate}
     We define $\mathsf{T}^{\leq n} \colonequals \Sigma^{-n}\mathsf{T}^{\leq 0}$ and $\mathsf{T}^{\geq n} \colonequals \Sigma^{-n}\mathsf{T}^{\geq 0}$. Further, we define the \emph{heart of the co-t-structure} to be $\mathsf{T}^{\heartsuit} \colonequals \mathsf{T}^{\leq 0} \cap \mathsf{T}^{\geq 0}$. 
     
    We shall sometimes also denote a co-t-structure by the pair $(\mathsf{U}, \mathsf{V})$, with $\mathsf{U}=\mathsf{T}^{\geq 0}$ and $\mathsf{V}=\mathsf{T}^{\leq 0}$.
\end{definition}

\begin{theorem}\label{Theorem Compactly generated co-t-structures}
    \cite[Theorem 5]{Pauksztello:2012}Let $\mathsf{C}$ be a set of compact objects in a triangulated category $\mathsf{T}$ with coproducts, such that $\Sigma^{-1}\mathsf{C}\subseteq \mathsf{C}$. Let $\mathsf{V} \colonequals \Sigma^{-1} \mathsf{C}^{\perp}$. Then, $(^{\perp}(\Sigma \mathsf{V}),\mathsf{V})$ is a co-t-structure on $\mathsf{T}$. Such co-t-structures are called \emph{compactly generated co-t-structures}.
\end{theorem}

Recall that a collection of maps $\mathcal{I}$ is called an ideal if the sum of two maps in $\mathcal{I}$ is also in $\mathcal{I}$, and for any triple of composable maps $g$,$f$, and $h$, the map $hfg$ is in $\mathcal{I}$ if $f$ is in $\mathcal{I}$.
We now recall some notions and results from \cite{Christensen:1998}, as we will need them in \Cref{Section Representability results}.

\begin{definition}\label{Definition of a projective class}\cite[Definition 2.5]{Christensen:1998}
    Let $\mathsf{T}$ be a triangulated category.
    Now, let $\mathcal{P}$ be a collection of objects, and $\mathcal{J}$ a collection of maps in $\mathsf{T}$. Then, we define, 
    \begin{enumerate}
        \item \textbf{null}$(\mathcal{P}) = \{f \in \HomT{X}{Y}  : X,Y \in \mathsf{T} \text{ and } fg=0 \text{ for all }g \in \HomT{\mathcal{P}}{X} \}$
        \item \textbf{proj}$(\mathcal{J}) = \{P \in \mathsf{T} : (\HomT{P}{X} \xrightarrow{\HomT{-}{f}}\HomT{P}{Y})=0 \text{ for all } f \in \mathcal{J}\}$
        \item $(\mathcal{P},\mathcal{J})$ is a \emph{projective class} if $\mathcal{J}=\textbf{null}(\mathcal{P})$ and $\mathcal{P} = \textbf{proj}(\mathcal{J})$ and every object $X \in \mathsf{T}$ admits a triangle $P \to X \xrightarrow{f} Y \to \Sigma P$ with $P \in \mathcal{P}$ and $f \in \mathcal{J}$.
    \end{enumerate}
\end{definition}
\begin{definition}\label{Definition of a Phantom map}\cite[Section 5]{Christensen:1998}
    Let $\mathsf{T}$ be a triangulated category with coproducts. A map $f:X \to Y$ is called a \emph{phantom map} if for any compact object $K$, and any map $K \to X$, the composite $K \to X \xrightarrow{f} Y$ vanishes. Note that the collection of phantom maps forms an ideal $\mathcal{I}$. Then, if $\big(\overline{\langle\mathsf{T}^c \rangle}_1,\mathcal{I}\big)$ is a projective class, it is called the \emph{phantom projective class}, see \Cref{Notation from Neeman's paper}. This is always the case if $\mathsf{T}$ has all coproducts, and $\mathsf{T}^c$ is essentially small. As we noted before, this is in particular the case if $\mathsf{T}$ is compactly generated.
\end{definition}

\begin{theorem}\label{Christensen 1.1}\cite[Theorem 1.1]{Christensen:1998}
    Let $(\mathcal{P},\mathcal{J})$ be a projective class. Define $\mathcal{P}_n$ for all $n \geq 1$ inductively by $\mathcal{P}_1 = \mathcal{P}$ and $\mathcal{P}_{n+1} \colonequals \operatorname{smd}(\mathcal{P}_{n}\star \mathcal{P}_1)$. Then, $(\mathcal{P}_n,\mathcal{J}^n)$ is a projective class for all $n \geq 1$.
\end{theorem}

\subsection{Metrics and completions for triangulated categories}\label{Section Metrics and Completions}
 In this subsection we will go over the definitions and results we need related to metrics and completions of triangulated categories à la Neeman. See \cite{Neeman:2020} for a 
 discussion on how the notion of good metrics relates to the notion of metrics on a category as introduced in \cite[Page 139-140]{Lawvere:1973}. All of the results in this subsection are from \cite{Neeman:2018} and \cite{Neeman:2020}.
 \begin{definition}\label{Definition Good Metric}\cite[Definition 10]{Neeman:2020}
     Let $\mathsf{T}$ be a triangulated category. A \emph{good metric} is a sequence $\{\mathcal{M}_n\}_{n \geq 1}$ of strictly full subcategories containing 0 such that,
     \begin{enumerate}
         \item $\mathcal{M}_n \star \mathcal{M}_n = \mathcal{M}_n$ for all $n \geq 1$.
         \item $\Sigma^{-1} \mathcal{M}_{n+1} \cup \mathcal{M}_{n+1} \cup \Sigma \mathcal{M}_{n+1} \subseteq \mathcal{M}_{n}$ for all $n \geq 1$.
     \end{enumerate}
 \end{definition}
 \begin{remark}\label{Remark on intution of metric}
     In \Cref{Definition Good Metric}, the intuition to keep in mind is that $\mathcal{M}_n$ represents the ball of radius $2^{-n}$ around $0$. To be more precise, given a good metric $\{\mathcal{M}_n\}_{n \geq 1}$, one can define the length of a morphism $f$ to be $2^{-N_f}$ where we define $N_f \colonequals \sup\{n : \operatorname{Cone}(f) \in \mathcal{M}_n\} $. In fact, this does give a metric in the sense of Lawvere, see \cite[Heuristic 9]{Neeman:2020}. 
 \end{remark}

\begin{definition}\label{Definition equivalence of metrics}
    Let $\mathsf{T}$ be a triangulated category. A good metric $\{\mathcal{M}_n\}_{n \geq 1}$ on $\mathsf{T}$ is said to be \emph{finer} than a good metric $\{\mathcal{N}_n\}_{n \geq 1}$ if for every $i \geq 1$, there exists $j \geq 1$ such that $\mathcal{M}_j \subseteq \mathcal{N}_i$. We say the metrics  $\{\mathcal{M}_n\}_{n \geq 1}$ and  $\{\mathcal{N}_n\}_{n \geq 1}$ are \emph{equivalent} if $\{\mathcal{M}_n\}_{n \geq 1}$ is finer than $\{\mathcal{N}_n\}_{n \geq 1}$ and $\{\mathcal{N}_n\}_{n \geq 1}$ is also finer than $\{\mathcal{M}_n\}_{n \geq 1}$
\end{definition}

\begin{definition}\label{Definition Cauchy sequence}
    Let $\mathsf{T}$ be a triangulated category with a good metric $\{\mathcal{M}_n\}_{n \geq 1}$. A \emph{Cauchy sequence} is a sequence $E_1 \to E_2 \to E_3 \to E_4 \to \cdots$ in $\mathsf{T}$ such that for each $i\geq 1$, there exists $n \geq 1$ such that $\operatorname{Cone}(E_j \to E_{j+1}) \in \mathcal{M}_i$ for all $j \geq n$. Note that given a Cauchy sequence of the form above, there exists a subsequence $E'_1 \to E'_2 \to E'_3 \to E'_4 \to \cdots$ such that $\operatorname{Cone}(E'_i \to E'_{i+1}) \in \mathcal{M}_{i+1}$ for each $i \geq 1$.
\end{definition}

\begin{definition}\label{Definition completions}\cite[Definition 1.10]{Neeman:2018}
    Let $\mathsf{T}$ be an essentially small triangulated category with a good metric $\{\mathcal{M}_n\}_{n \geq 1}$. We denote the Yoneda embedding by $\mathcal{Y}:\mathsf{T} \to \operatorname{Mod}(\mathsf{T})$, where $\operatorname{Mod}(\mathsf{T})$ denote the functor category $\Hom{}{\mathsf{T}^{\operatorname{op}}}{\operatorname{Mod}(\mathbb{Z})}$. Then, we can define the following full subcategories of $\operatorname{Mod}(\mathsf{T})$,
    \begin{enumerate}
        \item $\mathfrak{L}(\mathsf{T}) \colonequals \{\colim \mathcal{Y}(E_i) : E_1 \to E_2 \to \cdots \text{ is a Cauchy sequence in } \mathsf{T}\}$.
        \item $\mathfrak{C}(\mathsf{T}) \colonequals \bigcup_{n \geq 1}\mathcal{Y}(\mathcal{M}_n)^{\perp}$.
        \item $\mathfrak{S}(\mathsf{T}) \colonequals \mathfrak{L}(\mathsf{T}) \cap \mathfrak{C}(\mathsf{T})$. This is a triangulated category by \cite[Theorem 2.11]{Neeman:2018}
    \end{enumerate}
    We note here that equivalent metrics (\Cref{Definition equivalence of metrics}) give rise to the same $\mathfrak{L}(\mathsf{T})$, $\mathfrak{C}(\mathsf{T})$, and $\mathfrak{S}(\mathsf{T})$.
\end{definition}
The computation of the $\mathfrak{S}(\mathsf{T})$ is difficult in general. The following result, which was proven in a more general setting in \cite{Neeman:2018}, helps us in the computation in a special situation. 
\begin{theorem}\label{Theorem Tc to T is a good extension}\cite[Theorem 3.15]{Neeman:2018}
    Let $\mathsf{T}$ be triangulated category with coproducts and let $\mathsf{T}^c$ be the full subcategory of compact objects. Assume $\mathsf{T}^c$ has a good metric $\{\mathcal{M}_n\}_{n \geq 1}$. Let $\mathcal{Y} : \mathsf{T} \to \Hom{}{[\mathsf{T}^c]^{\operatorname{op}}}{\operatorname{Mod}(\mathbb{Z})}$ be the restricted Yoneda functor. Then, we can define the following full subcategories of $\mathsf{T}$, 
    \begin{enumerate}
        \item $\mathfrak{L}'(\mathsf{T}^c) \colonequals \{F \in \mathsf{T} : F \cong \hocolim E_n \text{ for a Cauchy sequence } E_1 \to E_2 \to \cdots \text{ in } \mathsf{T}^c \} $.
        \item $\mathcal{Y}^{-1}(\mathfrak{C}(\mathsf{T}^c)) = \bigcup_{n \geq 1} \mathcal{M}_n^{\perp}$, where the orthogonal is taken in $\mathsf{T}$, see \Cref{Definition completions}(2).
        \item $\mathfrak{S}'(\mathsf{T}^c) \colonequals \mathfrak{L}'(\mathsf{T}^c) \cap \mathcal{Y}^{-1}(\mathfrak{C}(\mathsf{T}^c)) $
    \end{enumerate}
    Then, the obvious functor from $\mathfrak{L}'(\mathsf{T}^c) \to \mathfrak{L}(\mathsf{T}^c)$ is an essential surjection which restricts to an equivalence $\mathfrak{S}'(\mathsf{T}^c) \to \mathfrak{S}(\mathsf{T}^c)$, see \Cref{Definition completions}. 
\end{theorem}

\subsection{Approximable triangulated categories}
In this subsection, we recall the definition of an approximable triangulated category, and review some related notions and results. All of the definitions and results are from \cite{Neeman:2021b}. We begin by recalling an equivalence relation on the collection of t-structures on a triangulated category.
\begin{definition}\label{Definition equivalence of t-structures}\cite[Definition 0.10]{Neeman:2021b}
    Let $\mathsf{T}$ be a triangulated category. Then, two t-structures $(\mathsf{U},\mathsf{V})$ and $(\mathsf{U}',\mathsf{V}')$ are said to be equivalent if there exists an integer $n \geq 0$ such that,
    \[\Sigma^n \mathsf{U} \subseteq \mathsf{U}' \subseteq \Sigma^{-n}\mathsf{U}\]
    This is an equivalence relation on the collection of t-structures on $\mathsf{T}$.
\end{definition}

\begin{definition}\label{Definition preferred equivalence class of t-structures}\cite[Definition 0.14]{Neeman:2021b}
    Let $\mathsf{T}$ be a triangulated category with a single compact generator $G$. Then, we define,
    \begin{enumerate}
        \item The \emph{t-structure generated by $G$} to be $(\mathsf{T}_{G}^{\leq 0},\mathsf{T}_{G}^{\geq 0})$, where $\mathsf{T}_{G}^{\leq 0} = \operatorname{Coprod}\big(\{\Sigma^i G\}_{i \geq 0}\big)$ and $\mathsf{T}_G^{\geq 0} = \big(\Sigma \mathsf{T}_{G}^{\leq 0}\big)^{\perp}$, see \Cref{Notation from Neeman's paper}. This is a t-structure by \Cref{Theorem Compactly generated t-structure}.
        \item The \emph{preferred equivalence class of t-structures} is the equivalence class of t-structures which contains the t-structure $(\mathsf{T}_{G}^{\leq 0},\mathsf{T}_{G}^{\geq 0})$, under the equivalence relation of \Cref{Definition equivalence of t-structures}.
    \end{enumerate}
   
\end{definition}
\begin{definition}\label{Definition T^-c and T^b_c}\cite[Definition 0.16]{Neeman:2021b}
    Let $\mathsf{T}$ be a triangulated category with coproducts and let $(\mathsf{U},\mathsf{V})$ be a t-structure on $\mathsf{T}$. Then, we define,
    \begin{enumerate}
        \item $\mathsf{T}^- = \bigcup_{n \in \mathbb{Z}} \mathsf{U}$, $\mathsf{T}^+ = \bigcup_{n \in \mathbb{Z}} \mathsf{V}$ and $\mathsf{T}^b = \mathsf{T}^- \cap \mathsf{T}^+$
        \item $\mathsf{T}^-_c = \bigcap_{n \in \mathbb{Z}}\big( \mathsf{T}^c \star \Sigma^n \mathsf{U} \big)$
        \item $\mathsf{T}^b_c = \mathsf{T}^-_c \cap \mathsf{T}^b$
    \end{enumerate}
    Note that all of these categories remain the same if we replace the t-structure by an equivalent one. If $\mathsf{T}$ has a single compact generator, we will always compute these categories with respect to a t-structure in the preferred equivalence class unless explicitly stated otherwise.
\end{definition}
We now recall the definition of approximability.
\begin{definition}\label{Definition approximability}\cite[Definition 0.21]{Neeman:2021b}
    Let $\mathsf{T}$ be a triangulated category with a single compact generator $G$. Then, we say that $\mathsf{T}$ is \emph{weakly approximable} if there exists a positive integer $A$ and a t-structure $(\mathsf{U},\mathsf{V})$ such that,
    \begin{enumerate}
        \item $\Sigma^A G \in \mathsf{U}$.
        \item $\HomT{\Sigma^{-i}G}{\mathsf{U}} = 0$ for all $i \geq A$.
        \item For all $F \in \mathsf{U}$, there exists a triangle $E \to F \to D \to \Sigma E$ with $E \in \overline{\langle G \rangle}^{[-A,A]}$ and $D \in \Sigma \mathsf{U}$, see \Cref{Notation from Neeman's paper}.
    \end{enumerate}
    We say that $\mathsf{T}$ is \emph{approximable}, if further,
    \begin{enumerate}
        \item[(3)$'$] For all $F \in \mathsf{U}$, there exists a triangle $E \to F \to D \to \Sigma E$ with $E \in \overline{\langle G \rangle}^{[-A,A]}_{A}$ and $D \in \Sigma \mathsf{U}$, see \Cref{Notation from Neeman's paper}.
    \end{enumerate}
    It is easy to verify that if $\mathsf{T}$ is weakly approximable (resp. approximable) with respect to one compact generator, then it is weakly approximable (resp. approximable) with respect to any other compact generator up to increasing the integer $A$ if necessary, see \cite[Proposition 2.6]{Neeman:2021b}.
\end{definition}

\subsection{On categories coming from algebraic geometry}\label{Section on notation for algebraic geometry}
We start by fixing the notation for some abelian categories arising from  algebra and algebraic geometry.
\begin{notation}
    Let $R$ be a ring. Then,
    \begin{enumerate}
        \item $\operatorname{Mod}(R)$ denotes the category of (right) $R$-modules.
        \item If $R$ is a coherent ring, $\mathbf{mod}(R)$ denotes the category of finitely presented modules. In particular, if the ring is noetherian, these coincide with the category of finitely generated modules.
    \end{enumerate}
\end{notation}
\begin{notation}
    Let $X$ be a quasicompact and quasiseparated scheme, and $Z \subseteq X$ a closed subset with a quasicompact complement. Then,
    \begin{enumerate}
        \item $\operatorname{Qcoh}(X)$ is the abelian category of all quasicoherent $\mathcal{O}_X$-modules.
        \item $\operatorname{Qcoh}_Z(X)$ is the abelian category of all quasicoherent $\mathcal{O}_X$-modules with support contained in the closed subset $Z$.
    \end{enumerate}
    If $X$ is further a noetherian scheme, we also define the following categories,
    \begin{enumerate}
    \setcounter{enumi}{2}
        \item $\operatorname{coh}(X)$ is the abelian category of all coherent $\mathcal{O}_X$-modules.
        \item $\operatorname{coh}_{Z}(X)$ is the abelian category of all coherent $\mathcal{O}_X$-modules with support contained in the closed subset $Z$.
    \end{enumerate}
\end{notation}

The primary examples and applications in this text will be to triangulated categories arising from algebraic geometry. Below we list some of the important ones, both as a reminder, and also to fix the notation. 
\begin{notation}\label{Notation triangulated categories for schemes}
    Let $X$ be a quasicompact and quasiseparated scheme, and $Z \subseteq X$ a closed subset with a quasicompact complement. Then,
    \begin{enumerate}
        \item $\mathbf{D}_{\operatorname{Qcoh}}(X)$ denotes the derived category of sheaves of $\mathcal{O}_X$-modules with quasicoherent cohomology. It is a compactly generated triangulated category, with the subcategory of compact objects being exactly $\mathbf{D}^{\operatorname{perf}}(X)$ defined below, see \cite[Theorem 3.2(i)]{Neeman:2024}.
        \item $\mathbf{D}_{\operatorname{Qcoh},Z}(X) \subseteq \mathbf{D}_{\operatorname{Qcoh}}(X)$ denotes the derived category of sheaves of $\mathcal{O}_X$-modules with quasicoherent cohomology which are supported on $Z$. That is, the restriction to $\mathbf{D}_{\operatorname{Qcoh}}(U)$ is isomorphic to 0, where $U$ is the complement of $Z$ in $X$. It is also a compactly generated triangulated category, with the subcategory of compact objects being exactly $\mathbf{D}^{\operatorname{perf}}_{Z}(X)$ defined below, see \cite[Theorem 3.2(i)]{Neeman:2024}.
        \item $\mathbf{D}^{\operatorname{perf}}(X) \subseteq \mathbf{D}_{\operatorname{Qcoh}}(X)$ denotes the strictly full subcategory of perfect complexes. If $X = \operatorname{Spec}(R)$ is an affine scheme, then there is an equivalence $\mathbf{D}^{\operatorname{perf}}(X) \cong \mathbf{K}^{b}(\operatorname{proj}\text{-}R)$, that is, the bounded homotopy category of finitely generated projective modules. Throughout this document, we will confuse $\mathbf{K}^{b}(\operatorname{proj}\text{-}R)$ with its essential image in $\mathbf{D}(R)$.
        \item $\mathbf{D}^{\operatorname{perf}}_{Z}(X)\subseteq \mathbf{D}_{\operatorname{Qcoh},Z}(X)$ denotes the strictly full subcategory of perfect complexes with  cohomologies supported on $Z$.
        \item $\mathbf{D}^{p}_{\operatorname{Qcoh},Z}(X)\subseteq \mathbf{D}_{\operatorname{Qcoh}}(X)$ denotes the strictly full subcategory of pseudocoherent complexes supported on $Z$. A complex is pseudocoherent if its restriction to the derived category of any affine open subscheme is isomorphic to a bounded above complex of finitely generated projective modules. 
        \item $\mathbf{D}^{p,b}_{\operatorname{Qcoh},Z}(X) \subseteq \mathbf{D}^{p}_{\operatorname{Qcoh},Z}(X)$ denotes the strictly full subcategory of cohomologically bounded pesudeocoherent complexes supported on $Z$. 
        \item $\mathbf{D}(\operatorname{Qcoh}(X))$ denotes the derived category of quasicoherent sheaves on $X$. If $X$ is separated, then $\mathbf{D}(\operatorname{Qcoh}(X))$ is equivalent to $\mathbf{D}_{\operatorname{Qcoh}}(X)$.
    \end{enumerate}
    If $X$ is further a noetherian scheme, then,
    \begin{enumerate}
     \setcounter{enumi}{7}
        \item $\mathbf{D}^{b}_{\operatorname{coh}}(X) \subseteq \mathbf{D}_{\operatorname{Qcoh}}(X)$ denotes the strictly full subcategory of cohomologically bounded complexes of sheaves $\mathcal{O}_X$-modules with coherent cohomology. Note that if $X$ is separated, $\mathbf{D}^{b}_{\operatorname{coh}}(X)$ is equivalent to $ \mathbf{D^b}(\operatorname{coh}(X)) \subseteq \mathbf{D}(\operatorname{Qcoh}(X))$, which is the full subcategory of cohomologically bounded complexes of coherent sheaves. 
        \item Similarly, $\mathbf{D}^{b}_{\operatorname{coh},Z}(X) \subseteq \mathbf{D}_{\operatorname{Qcoh},Z}(X)$ denotes the strictly full subcategory of cohomologically bounded complexes of sheaves $\mathcal{O}_X$-modules with coherent cohomology. Note that it agrees with $\mathbf{D}^{p,b}_{\operatorname{Qcoh},Z}(X)$.
        \item $\mathbf{D}^{+}_{\operatorname{coh}}(X)$ denotes the derived category of sheaves of $\mathcal{O}_X$-modules with bounded below and coherent cohomology.
        \item $\mathbf{K}(\operatorname{Inj}\text{-}X)$ denotes the homotopy category of injective quasicoherent sheaves. It is a compactly generated triangulated category. The subcategory of compact objects is equivalent to $\mathbf{D^{b}}(\operatorname{coh}(X))$ via the functor sending a complex in $\mathbf{D^{b}}(\operatorname{coh}(X))$ to its injective resolution, see \cite[Proposition 2.3]{Krause:2005}. 
        \item If $X$ is separated, $\mathbf{K}_{m}(\operatorname{Proj}\text{-}X)$ denotes the mock homotopy category of projectives, see \cite{Murfet:2008} for details.
    \end{enumerate}

\end{notation}

\begin{theorem}\label{Theorem Tc- Tbc for schemes Neeman}\cite[Example 3.4]{Neeman:2021b}
    Let $X$ be a quasicompact and quasiseparated scheme, and $Z\subseteq X$ a closed subset with quasicompact complement. Then, $\mathbf{D}_{\operatorname{Qcoh},Z}(X)$ is weakly approximable, see \Cref{Definition approximability}. Further,
    \[\big(\mathbf{D}_{\operatorname{Qcoh},Z}(X)\big)^-_c = \mathbf{D}^{p}_{\operatorname{Qcoh},Z}(X) \text{ and } \big(\mathbf{D}_{\operatorname{Qcoh},Z}(X)\big)^b_c = \mathbf{D}^{p,b}_{\operatorname{Qcoh},Z}(X)\]
    see \Cref{Definition T^-c and T^b_c}.
    Finally, from \cite[Proposition 0.15(i)]{Neeman:2018}, we get that,
    \[\mathfrak{S}(\mathbf{D}^{\operatorname{perf}}_{Z}(X)) = \mathbf{D}^{p,b}_{\operatorname{Qcoh},Z}(X)\]
    with respect to the metric on $\mathbf{D}^{\operatorname{perf}}_{Z}(X)$ given by $\mathbf{D}^{\operatorname{perf}}_{Z}(X) \cap \mathbf{D}_{\operatorname{Qcoh}}(X)^{\leq -n}\}_{n \geq 1}$. See \cite[Remark 4.7]{Neeman:2018} for other equivalent metrics.
    We recall here that $\mathbf{D}^{p,b}_{\operatorname{Qcoh},Z}(X) = \mathbf{D}^{b}_{\operatorname{coh}}(X)$ for any noetherian scheme $X$.
\end{theorem}

\section{Generating sequence and orthogonal metrics}\label{Section Generating sequence and orthogonal metrics}

Recall that a good metric is $\mathbb{N}$-graded, see \Cref{Definition Good Metric}. In this work, we will require metrics to be $\mathbb{Z}$-graded at times, instead of $\mathbb{N}$-graded. We give the definition below.
\begin{definition}\label{Good Metric}
    Let $\mathsf{S}$ be a triangulated category. An \emph{extended good metric} on $\mathsf{S}$ is a sequence of strictly full subcategories $\mathcal{M} = \{\mathcal{M}_n\}_{n \in \mathbb{Z}}$ each containing 0, such that, 
    \begin{enumerate}
        \item $\mathcal{M}_n \star \mathcal{M}_n \subseteq \mathcal{M}_n $ for all $n \in \mathbb{Z}$.
        \item $ \Sigma \mathcal{M}_{n + 1} \cup \mathcal{M}_{n + 1} \cup \Sigma^{-1}\mathcal{M}_{n + 1} \subseteq \mathcal{M}_n$ for all $n \in \mathbb{Z}$.
    \end{enumerate}
    We say an extended good metric is an \emph{orthogonal metric} if $^{\perp}(\mathcal{M}_n^{\perp}) = \mathcal{M}_n$ for all $n \in \mathbb{Z}$. Note that if $\mathsf{S}$ has coproducts and $\mathcal{M}$ is an orthogonal metric, then $\mathcal{M}_n$ is closed under coproducts for all $n \in \mathbb{Z}$.
\end{definition}
\begin{remark}
    Note that every good metric $\{\mathcal{M}_n\}_{n \geq 1 }$ can be made into an extended good metric by setting $\mathcal{M}_{-n}= \mathsf{S}$ for all $n \geq 0$. Conversely, if $\{\mathcal{M}_n\}_{n \in \mathbb{Z} }$ is an extended good metric, then $\{\mathcal{M}_n\}_{n \geq 1 }$ is a good metric. So, from now on, when we write metric, we would mean an extended good metric, or a good metric, depending on the context. Observe that some notions only depend on the $\mathcal{M}_n$ for $n > > 0$, for example the notion of a Cauchy sequence, see \Cref{Definition Cauchy sequence}. 
\end{remark}

We extend the definition of equivalence of metrics to extended good metrics in the obvious manner, see \cite[Definition 2]{Neeman:2020}. Further, we introduce the notion of $\mathbb{N}$-equivalence of extended good metrics.

\begin{definition}\label{Definition equivalence relation on extended good metrics}
    Let $\mathsf{S}$ be a triangulated category. Then, 
    \begin{enumerate}
        \item We say an extended good metric $\mathcal{M}$ is \emph{finer} than an extended good metric $\mathcal{N}$ if for all $i \in \mathbb{Z}$, there exists $j \in \mathbb{Z}$ such that $\mathcal{M}_j \subseteq \mathcal{N}_i$. We denote this by $\mathcal{M} \preceq \mathcal{N}$.
        \item We say an extended good metric $\mathcal{M}$ is $\mathbb{N}$-\emph{finer} than an extended good metric $\mathcal{N}$ if there exists $n \geq 0$ such that $\mathcal{M}_{i + n} \subseteq \mathcal{N}_{i}$ for all $i \in \mathbb{Z} $. We denote this by $\mathcal{M} \preceq_{\mathbb{N}} \mathcal{N}$.
    \end{enumerate}
    see \Cref{Good Metric} for the definition of an extended good metric.
    We say that $\mathcal{M}$ and $\mathcal{N}$ are equivalent if $\mathcal{M} \preceq \mathcal{N} \preceq \mathcal{M}$. Similarly, we say that they are $\mathbb{N}$-equivalent if $\mathcal{M} \preceq_{\mathbb{N}} \mathcal{N} \preceq_{\mathbb{N}} \mathcal{M}$.
\end{definition}
The two notions coincide in some cases, for example when the metrics arise from t-structures or co-t-structures. The following remark states this in a slightly more general setting. 
\begin{remark}
    Let $\mathsf{T}$ be a triangulated category, with extended good metrics $\mathcal{M}$ and $\mathcal{N}$ (\Cref{Good Metric}) such that one of the following two conditions hold,
    \begin{enumerate}
        \item There exist two strictly full subcategories $\mathsf{U}_1$ and $\mathsf{U}_2$ such that $\mathcal{M}_n = \Sigma^n\mathsf{U}_1$, and $\mathcal{N}_n = \Sigma^n \mathsf{U}_2$ for all $n \in \mathbb{Z}$.
        \item There exist two strictly full subcategories $\mathsf{U}_1$ and $\mathsf{U}_2$ such that $\mathcal{M}_n = \Sigma^{-n}\mathsf{U}_1$, and $\mathcal{N}_n = \Sigma^{-n} \mathsf{U}_2$ for all $n \in \mathbb{Z}$.
    \end{enumerate}
     Then, $\mathcal{M}\preceq \mathcal{N}$ if and only if $\mathcal{M}\preceq_{\mathbb{N}} \mathcal{N}$, see \Cref{Definition equivalence relation on extended good metrics}.
\end{remark}

We now introduce the notion of a pre-generating sequence.

\begin{definition}\label{Definition pre-generating sequence}
    Let $\mathsf{T}$ be a compactly generated triangulated category. A \emph{pre-generating sequence} is a sequence of full subcategories $\{\mathcal{G}^i\}_{i \in \mathbb{Z}}$ of $\mathsf{T}^c$ such that the smallest triangulated subcategory of $\mathsf{T}$ containing the union of $\mathcal{G}^i$ which is closed under finite direct sums, extensions, and summands is exactly $\mathsf{T}^c$. In other words, $\operatorname{smd}(\operatorname{coprod}(\cup_{i \in \mathbb{Z}} \mathcal{G}^i )) = \mathsf{T}^c$, see \Cref{Notation from Neeman's paper}. Note that this implies that $\big(\bigcup_{i \in \mathbb{Z}}\mathcal{G}^i \big)^{\perp} = 0$ as $\mathsf{T}$ is compactly generated.
\end{definition}

Given a pre-generating sequence, we define some full subcategories of $\mathsf{T}$ in a similar spirit to \Cref{Notation Subcategories of a triangulated category}.

\begin{notation}\label{Notation Subcategories for (pre)-generating sequences}
    Let $\mathsf{T}$ be a compactly generated triangulated category with a pre-generating sequence $\mathcal{G} \colonequals \{\mathcal{G}^i\}_{i \in \mathbb{Z}}$, see \Cref{Definition pre-generating sequence}. We define the following full subcategories of $\mathsf{T}$ for any integer $i \leq j$,
    \begin{itemize}
        \item $\mathcal{G}{[i,j]} \colonequals \bigcup_{i \leq n \leq j}\mathcal{G}^n$
        \item $\mathcal{G}{(-\infty,j]} \colonequals \bigcup_{ n \leq j}\mathcal{G}^n$
        \item $\mathcal{G}{[i, \infty)} \colonequals \bigcup_{ i \leq n}\mathcal{G}^n$
        \item $\mathcal{G}{(-\infty,\infty)} \colonequals \bigcup_{ n \in \mathbb{Z } }\mathcal{G}^n$
    \end{itemize}
    And the following full subcategories of $\mathsf{T}$ for any integer $i \leq j$ and $n \geq 1$,
    \begin{itemize}
        \item $\mathcal{G}^{[i,j]}_n \colonequals \operatorname{smd}\left(\operatorname{coprod}_n\left( \mathcal{G}{[i,j]} \right) \right) $ 
        \item $\mathcal{G}^{(-\infty,j]}_n \colonequals \operatorname{smd}\left(\operatorname{coprod}_n\left( \mathcal{G}{(-\infty,j]}  \right) \right)$
        \item $\mathcal{G}^{[i, \infty)}_n \colonequals \operatorname{smd}\left(\operatorname{coprod}_n\left(   \mathcal{G}{[i, \infty)}  \right) \right)$
        \item $\mathcal{G}^{[i,j]}_{\infty } = \mathcal{G}^{[i,j]} \colonequals \operatorname{smd}\left(\operatorname{coprod}\left( \mathcal{G}{[i,j]} \right) \right) $ 
        \item $\mathcal{G}^{(-\infty,j]}_{\infty} = \mathcal{G}^{(-\infty,j]} \colonequals \operatorname{smd}\left(\operatorname{coprod}\left( \mathcal{G}{(-\infty,j]}  \right) \right)$
        \item $\mathcal{G}^{[i, \infty)}_{\infty} = \mathcal{G}^{[i, \infty)} \colonequals \operatorname{smd}\left(\operatorname{coprod}\left(   \mathcal{G}{[i, \infty)}  \right) \right)$
        \item $ \mathcal{G}_n = \mathcal{G}^{(-\infty, \infty)}_{n} \colonequals \operatorname{smd}\left(\operatorname{coprod}\left(\mathcal{G}(-\infty, \infty) \right) \right)$
    \end{itemize}
    see \Cref{Notation from Neeman's paper}.
    We also define the analogous big categories. For any integer $i \leq j$ and $n \geq 1$,
    \begin{itemize}
        \item $\overline{\mathcal{G}}^{[i,j]}_n \colonequals \operatorname{smd}\left(\operatorname{Coprod}_n\left( \mathcal{G}{[i,j]} \right) \right) $ 
        \item $\overline{\mathcal{G}}^{(-\infty,j]}_n \colonequals \operatorname{smd}\left(\operatorname{Coprod}_n\left( \mathcal{G}{(-\infty,j]}  \right) \right)$
        \item $\overline{\mathcal{G}}^{[i, \infty)}_n \colonequals \operatorname{smd}\left(\operatorname{Coprod}_n\left(   \mathcal{G}{[i, \infty)}  \right) \right)$
        \item $\overline{\mathcal{G}}^{[i,j]} \colonequals \operatorname{smd}\left(\operatorname{Coprod}\left( \mathcal{G}{[i,j]} \right) \right) $
        \item $\overline{\mathcal{G}}^{(-\infty,j]} \colonequals \operatorname{smd}\left(\operatorname{Coprod}\left( \mathcal{G}{(-\infty,j]}  \right) \right)$
        \item $\overline{\mathcal{G}}^{[i, \infty)} \colonequals \operatorname{smd}\left(\operatorname{Coprod}\left(   \mathcal{G}{[i, \infty)}  \right) \right)$
    \end{itemize}
\end{notation}
\begin{remark}\label{Remark big categories intersection compacts gives small categories}
    We note here that for a pre-generating sequence (\Cref{Definition pre-generating sequence}), we recover the small categories from the big ones on taking the intersection with the compacts by \cite[Proposition 1.9]{Neeman:2021a}. That is, $\overline{\mathcal{G}}^{?}_{!} \cap \mathsf{T}^c = \mathcal{G}^{?}_{!} $ for any choice of $?$ and $!$ as in \Cref{Notation Subcategories for (pre)-generating sequences}.
\end{remark}
We now give the definition of a generating sequence.

\begin{definition}\label{Definition of generating sequence}
    Let $\mathsf{T}$ be a compactly generated triangulated category. A \emph{generating sequence} is a sequence of full subcategories $\mathcal{G} = \{\mathcal{G}^i\}_{i \in \mathbb{Z}}$ of $\mathsf{T}^c$ such that 
    \begin{enumerate}
        \item $\mathcal{G}$ is a pre-generating sequence (\Cref{Definition pre-generating sequence}), that is, $\operatorname{smd}(\operatorname{coprod}(\cup_{i \in \mathbb{Z}} \mathcal{G}^i )) = \mathsf{T}^c$.
        \item $\Sigma^{-1}\mathcal{G}^n \cup \mathcal{G}^n \cup \Sigma\mathcal{G}^n \subseteq \mathcal{G}^{[n-1,n+1]}$ for all $n \in \mathbb{Z}$, see \Cref{Notation Subcategories for (pre)-generating sequences}. In particular, we get an extended good metric (see \Cref{Good Metric}) $\mathcal{M}^{\mathcal{G}}$ on $\mathsf{T}^c$ with $\mathcal{M}^{\mathcal{G}}_n$ defined as follows,
        \[\mathcal{M}^{\mathcal{G}}_n \colonequals \mathcal{G}^{(-\infty,- n]} \]
    \end{enumerate}
    
    Finally, we call a generating sequence a \emph{finite generating sequence} if,
    \begin{enumerate}
        \setcounter{enumi}{2}
        \item Each of the subcategories $\mathcal{G}^i$ have only finitely many objects.
    \end{enumerate}

    Given any generating sequence $\mathcal{G}$, we define the orthogonal to $\mathcal{G}$ to be the full subcategory of $\mathsf{T}$ given by $\mathcal{G}^{\perp} = \bigcup_{n \in \mathbb{Z}}(\mathcal{M}^{\mathcal{G}}_n)^{\perp}$. We define the extended good metric $\mathcal{R}^{\mathcal{G}}$ on $\mathsf{T}$ by,
    \[\mathcal{R}^{\mathcal{G}}_n \colonequals{}^\perp\big[\big({\mathcal{M}_n^{\mathcal{G}}}\big)^{\perp}\big] \text{ for all } n \in \mathbb{Z} \]
    where the orthogonals are taken in $\mathsf{T}$. We show below that $\mathcal{R}^{\mathcal{G}}$ is an orthogonal metric on $\mathsf{T}$, see \Cref{Good Metric}.
\end{definition}
\begin{lemma}\label{Lemma R^G is an orthgonal good metric}
    Let $\mathsf{T}$ be a triangulated category.
    Then,
    \begin{enumerate}
        \item Let $\mathcal{M} \subseteq \mathsf{T}$ be a full subcategory. We define $\mathcal{R} = {}^\perp\big(\mathcal{M}^{\perp}\big)$. Then, $^\perp\big(\mathcal{R}^{\perp}\big) = \mathcal{R}$.
        \item If $\mathsf{T}$ has a generating sequence $\mathcal{G}$, then, $\mathcal{R}^{\mathcal{G}}$ is an orthogonal good metric on $\mathsf{T}$, see \Cref{Definition of generating sequence} for the definition of $\mathcal{R}^{\mathcal{G}}$ and \Cref{Good Metric} for the definition of an orthogonal metric.
    \end{enumerate}
\end{lemma}
\begin{proof}
   
    We just need to show (1), as it implies (2). For any full subcategory $\mathcal{N} \subseteq \mathsf{T}$, we have the following inclusions,
    \[\begin{tikzcd}
	{({}^{\perp}\mathcal{N})^{\perp}} & & {\mathcal{N}} & & {{}^{\perp}(\mathcal{N}^{\perp})}
	\arrow["{i(\mathcal{N})}"', hook', from=1-3, to=1-1]
	\arrow["{j(\mathcal{N})}", hook, from=1-3, to=1-5]
    \end{tikzcd}\]
    Note that taking the orthogonal of $i(\mathcal{N})$ gives us the inclusion $l(\mathcal{N}) :  {}^{\perp}\big([{}^{\perp}\mathcal{N}]^{\perp}\big) \hooklongrightarrow {}^{\perp}\mathcal{N}$.  
    This immediately gives us,
    \[\begin{tikzcd}
	{^{\perp}(\mathcal{R}^{\perp})={}^{\perp}\big([\mathcal{{}^{\perp}(\mathcal{M}^{\perp})]}^{\perp}\big)} & & {\mathcal{R}={}^{\perp}(\mathcal{M}^{\perp})} & & {{}^{\perp}\big([\mathcal{{}^{\perp}(\mathcal{M}^{\perp})]}^{\perp}\big) = {}^{\perp}(\mathcal{R}^{\perp}) }
	\arrow["l(\mathcal{M}^{\perp})", hook, from=1-1, to=1-3]
	\arrow["{j({}^{\perp}[\mathcal{M}^{\perp}])}", hook,from=1-3, to=1-5]
\end{tikzcd}\]
\end{proof}

The trivial example of a generating sequence would of course be to let each $\mathcal{G}^i = \mathsf{T}^c$. But non-trivial generating sequences also frequently occur naturally as the following example shows.
\begin{example}
    Let $\mathsf{T}$ be a triangulated category with a single compact generator $G$, then, the following are finite generating sequences (\Cref{Definition of generating sequence}),
    \begin{enumerate}
        \item $\mathcal{G}^i = \{\Sigma^i{G}\}$ for each $i \in \mathbb{Z}$.
        \item $\mathcal{G}^i = \{\Sigma^{-i}{G}\}$ for each $i \in \mathbb{Z}$.
    \end{enumerate}
    If there is a countable generating set $\{G_n\}_{n  \in \mathbb{N}}$, then we have the following finite generating sequence,
    \begin{enumerate}
    \setcounter{enumi}{2}
        \item $\mathcal{G}^i \colonequals \{\Sigma^{i-j} G_j : j \text{ lies between 0 and }i \}$
    \end{enumerate}
\end{example}

We record the following easy lemma for later use.

\begin{lemma}\label{Lemma for lemma on implicatiopn of strong generating sequence}
    Let $\mathsf{T}$ be a triangulated category with a generating sequence $\mathcal{G}$, see \Cref{Definition of generating sequence}. Suppose there exists an integer $A >0$ such that for each integer $i$, we have that, $\mathcal{M}^{\mathcal{G}}_{i} \subseteq \bigcup_{|n|\leq A}\Sigma^n\mathcal{M}^{\mathcal{G}}_{i+1}$, see \Cref{Definition of generating sequence}(2) for the definition of $\mathcal{M}^{\mathcal{G}}$. Then,
    \begin{enumerate}
        \item For any integers $i$ and positive integer $j$, we have that $\mathcal{M}^{\mathcal{G}}_{i} \subseteq \bigcup_{|n| \leq jA}\Sigma^n\mathcal{M}^{\mathcal{G}}_{i + j}$.
        \item For any integer $i$, $\mathsf{T}^c = \bigcup_{n \in \mathbb{Z}}\Sigma^n \mathcal{M}^{\mathcal{G}}_i$.
    \end{enumerate}
\end{lemma}
\begin{proof}
    We prove (1) first. Fix an integer $i$. We will prove the claim by induction on $j \geq 1$. The base case $j=1$ is true from the hypothesis. Suppose we know the result up to some positive integer $j$. So, we have that $\mathcal{M}^{\mathcal{G}}_{i} \subseteq \bigcup_{|n| \leq jA}\Sigma^n\mathcal{M}^{\mathcal{G}}_{i + j}$. Again, by the hypothesis, we have that $\mathcal{M}^{\mathcal{G}}_{i+j} \subseteq \bigcup_{|n|\leq A}\Sigma^n\mathcal{M}^{\mathcal{G}}_{i+j+1}$. So, we get that $\mathcal{M}^{\mathcal{G}}_{i} \subseteq \bigcup_{|n| \leq (j+1)A}\Sigma^n\mathcal{M}^{\mathcal{G}}_{i + j + 1}$, which is what we needed for induction. 
    
    Now we prove (2). Fix an integer $i$. By \Cref{Definition of generating sequence}(2), $\mathcal{M}^{\mathcal{G}}_i \supseteq \bigcup_{j \geq i} \mathcal{M}^{\mathcal{G}}_j$. Further, by (1), we have that $\mathcal{M}^{\mathcal{G}}_{j} \subseteq \bigcup_{|n| \leq (i-j)A}\Sigma^n\mathcal{M}^{\mathcal{G}}_{i}$ for any $j \leq i$. And so, 
    \[\mathsf{T}^c = \mathcal{G}^{(-\infty,\infty)} = \bigcup_{j \in \mathbb{Z}}\mathcal{G}^{(-\infty,-j]} = \bigcup_{j \in \mathbb{Z}}\mathcal{M}^{\mathcal{G}}_j \subseteq \bigcup_{n \in \mathbb{Z}}\Sigma^n \mathcal{M}^{\mathcal{G}}_i\]
    with notation as in \Cref{Notation Subcategories for (pre)-generating sequences}.
\end{proof}

We now define a preorder, and a corresponding equivalence relation on the collection of generating sequences, which restricts to a preorder and an equivalence relation on the subcollection of finite generating sequences too, see \Cref{Definition of generating sequence}.
\begin{definition}\label{Definition equivalence class of generating sequences}
    Let $\mathsf{T}$ be a compactly generated triangulated category. We say that a generating sequence $\mathcal{G}_1$ is $\mathbb{N}$-\emph{finer} than $\mathcal{G}_2$ if there exists a integer $N \geq 0$ such that for all $n \in \mathbb{Z}$, $\mathcal{G}_1^n \subseteq \mathcal{G}_2^{[n-N, n+N]}$, see \Cref{Notation Subcategories for (pre)-generating sequences}. We denote this by $\mathcal{G}_1 \preceq_{\mathbb{N}} \mathcal{G}_2$. We say two generating sequences $\mathcal{G}_1$ and $\mathcal{G}_2$ are $\mathbb{N}$-equivalent if $\mathcal{G}_1 \preceq_{\mathbb{N}} \mathcal{G}_2 \preceq_{\mathbb{N}} \mathcal{G}_1$. 
\end{definition}

\begin{example}\label{Examples of N-equivalent generating sequences}
    Let $\mathsf{T}$ be a triangulated category with a compact generator $G$. Suppose $H$ is another compact generator for $\mathsf{T}$. This implies that, 
    \[\langle G \rangle = \bigcup_{n \geq 0} \langle G \rangle^{[-n,n]}  = \mathsf{T}^c = \bigcup_{n \geq 0} \langle H \rangle^{[-n,n]} = \langle H \rangle\] 
    see \Cref{Notation from Neeman's paper}. And hence $H \in \langle G \rangle^{[-n, n]}$ and $G \in \langle H \rangle^{[-n, n]}$ for some integer $n \geq 0$. This immediately gives us that,
    \begin{enumerate}
        \item The generating sequence (\Cref{Definition of generating sequence}) $\mathcal{G}_G$ given by $\mathcal{G}_G^i = \{\Sigma^{-i}G\}$ is $\mathbb{N}$-equivalent (see \Cref{Definition equivalence class of generating sequences}) to the generating sequence $\mathcal{G}_H$ given by $\mathcal{G}_H^i = \{\Sigma^{-i}H\}$.
        \item The generating sequence (\Cref{Definition of generating sequence}) ${\widetilde{\mathcal{G}}}_G$ given by ${\widetilde{\mathcal{G}}_G}^i = \{\Sigma^{i} G\}$ is $\mathbb{N}$-equivalent to the generating sequence ${\widetilde{\mathcal{G}}}_H$ given by ${\widetilde{\mathcal{G}}}_H^{i} = \{\Sigma^{i} H\}$.
    \end{enumerate}
     
\end{example}

We now show that $\mathbb{N}$-equivalent generating sequences give rise to $\mathbb{N}$-equivalent metrics.

\begin{lemma}\label{Lemma eq. class}
    Let $\mathsf{T}$ be a compactly generated triangulated category with $\mathbb{N}$-equivalent generating sequences $\mathcal{G}$ and $\mathcal{G}'$, see \Cref{Definition equivalence class of generating sequences} for the definition of $\mathbb{N}$-equivalence of generating sequences. Then,
    \begin{enumerate}
        \item The corresponding extended good metrics $\mathcal{M}^{\mathcal{G}}$ and $\mathcal{M}^{\mathcal{G}'}$ on $\mathsf{T}^c$ are $\mathbb{N}$-equivalent, see \Cref{Definition of generating sequence}(2) for the definition of $\mathcal{M}^{\mathcal{G}}$.
        \item The corresponding orthogonal metrics $\mathcal{R}^{\mathcal{G}}$ and $\mathcal{R}^{\mathcal{G}'}$ on $\mathsf{T}$ are $\mathbb{N}$-equivalent, see \Cref{Definition of generating sequence} for the definition of $\mathcal{R}^{\mathcal{G}}$.
    \end{enumerate}
    See \Cref{Definition equivalence relation on extended good metrics} for the definition of $\mathbb{N}$-equivalence of extended good metrics
\end{lemma}
\begin{proof}
    By \Cref{Definition equivalence class of generating sequences}, there exists an integer $N \geq 0$ such that $\mathcal{G}'^n \subseteq \mathcal{G}^{[n-N,n+N]}$ for all $n \in \mathbb{Z}$, see \Cref{Notation Subcategories for (pre)-generating sequences}. And so, $\mathcal{G}'^{(-\infty, -n]} \subseteq \mathcal{G}^{(-\infty,-n+N]}$ for all $n \in \mathbb{Z}$. In particular, this shows that $\mathcal{M}^{\mathcal{G}'} \preceq_{\mathbb{N}} \mathcal{M}^{\mathcal{G}}$. By symmetry, we get that they are in fact $\mathbb{N}$-equivalent
    
    Further, this gives us that the right orthogonals have the reverse inclusion, that is, $(\mathcal{M}^{\mathcal{G}}_{n+N})^{\perp} \subseteq (\mathcal{M}^{\mathcal{G}'}_n)^{\perp}$. Taking left orthogonals, we get $\mathcal{R}^{\mathcal{G}'} \preceq_{\mathbb{N}} \mathcal{R}^{\mathcal{G}}$. Again, by symmetry, we get that they are in fact $\mathbb{N}$-equivalent
\end{proof}

Inspired by this lemma, we can define the preferred equivalence class of orthogonal metrics for a triangulated category with a given generating sequence.

\begin{definition}\label{Definition preferred equivalence class of orthogonal metrics}
    Let $\mathsf{T}$ be a compactly generated triangulated category with a generating sequence $\mathcal{G}$, see \Cref{Definition of generating sequence}. Then, 
    \begin{enumerate}
        \item An orthogonal metric $\mathcal{R}$ lies in the \emph{preferred equivalence class of orthogonal metrics} if it is equivalent to $\mathcal{R}^{\mathcal{G}}$. 
        \item An orthogonal metric $\mathcal{R}$ lies in the \emph{preferred $\mathbb{N}$-equivalence class of orthogonal metrics} if it is $\mathbb{N}$-equivalent to $\mathcal{R}^{\mathcal{G}}$.
    \end{enumerate} 
    See \Cref{Definition of generating sequence} for the definition of $\mathcal{R}^{\mathcal{G}}$, \Cref{Good Metric} for the definition of an orthogonal metric, and \Cref{Definition equivalence relation on extended good metrics} for the definition of the equivalence relations on metrics. Recall from \Cref{Lemma R^G is an orthgonal good metric} that $\mathcal{R}^{\mathcal{G}}$ is an orthogonal metric.
    
    Note that by \Cref{Lemma eq. class} the preferred equivalence class and the preferred $\mathbb{N}$-equivalence class of orthogonal metrics remain the same if $\mathcal{G}$ is replaced by a $\mathbb{N}$-equivalent generating sequence, see \Cref{Definition equivalence class of generating sequences}.
    
\end{definition}
 \begin{remark}\label{Remark on preferred equivalence class of orthogonal metrics}
 Let $\mathsf{T}$ be a compactly generated triangulated category with a generating sequence $\mathcal{G}$, see \Cref{Definition of generating sequence}. Then, 
     \begin{enumerate}
         \item If an orthogonal metric $\mathcal{R}$ lies in the preferred equivalence class of orthogonal metrics then $\mathcal{M}^{\mathcal{G}} \preceq \mathcal{R} \cap \mathsf{T}^c$.
         \item If an orthogonal metric $\mathcal{R}$ lies in the preferred $\mathbb{N}$-equivalence class of orthogonal metrics then $\mathcal{M}^{\mathcal{G}} \preceq_{\mathbb{N}} \mathcal{R} \cap \mathsf{T}^c$.
     \end{enumerate}
     See \Cref{Definition equivalence relation on extended good metrics} for the definition of the order relation on metrics, \Cref{Definition of generating sequence}(2) for the definition of $\mathcal{M}^{\mathcal{G}}$, and \Cref{Good Metric} for the definition of an orthogonal metric.
 \end{remark}
We now define the closure of the compacts with respect to a metric in the setup of a triangulated category with coproducts.

\begin{definition}\label{Definition closure of compacts}
    Let $\mathsf{T}$ be a triangulated category with coproducts, let $\mathcal{R}$ be an extended good metric (\Cref{Good Metric}) on $\mathsf{T}$, and $\mathcal{M}$ be an extended good metric on the subcategory of compacts, $\mathsf{T}^c$. Then, the closure of the compacts with respect to $\mathcal{R}$ is defined to be,
    \[ \overline{\mathsf{T}^c} \colonequals \bigcap_{n \in \mathbb{Z}} \mathsf{T}^c \star \mathcal{R}_n\]
    We define the orthogonal to the metric $\mathcal{M}$ to be,
    \[\mathcal{M}^{\perp} \colonequals \bigcup_{n \in \mathbb{Z}} \mathcal{M}_n^{\perp}\]
    where the orthogonal is taken in the category $\mathsf{T}$.
    Finally, we define the bounded objects of the closure of the compacts to be,
    \[\mathsf{T}^b_c  \colonequals \overline{\mathsf{T}^c} \cap \mathcal{M}^{\perp}\]
    We are mostly interested in the case when $\bigcup_{n \in \mathbb{Z}} \mathcal{M}_n^{\perp} = \bigcup_{n \in \mathbb{Z}} \mathcal{R}_n^{\perp}$.
    
     Note that all of these categories remain the same if we replace $\mathcal{R}$ and $\mathcal{M}$ with equivalent extended good metrics, see \Cref{Definition equivalence relation on extended good metrics}. Unless mentioned otherwise, if we are working with a compactly generated triangulated category $\mathsf{T}$ with a generating sequence $\mathcal{G}$, then we will compute the category $\overline{\mathsf{T}^c}$ with respect to the metric $\mathcal{R}^{\mathcal{G}}$, see \Cref{Definition of generating sequence}. Further, we will define $\mathsf{T}^b_c$ in this case by, $\mathsf{T}^b_c = \overline{\mathsf{T}^c} \cap \left(\mathcal{M}^{\mathcal{G}}\right)^{\perp} $. Note that $(\mathcal{M}^{\mathcal{G}})^{\perp} = (\mathcal{R}^{\mathcal{G}})^{\perp}$ with both orthogonals being taken in $\mathsf{T}$.
\end{definition}
\begin{remark}
    The reason for calling $\overline{\mathsf{T}^c}$ the closure of the compacts comes from metric spaces.  Let $\mathsf{T}$ be a triangulated category with coproducts, and let $\mathcal{M}$ be an extended good metric on the subcategory of compacts, $\mathsf{T}^c$. Then, any object $F$ in $\overline{\mathsf{T}^c}$ is an \say{arbitrarily small distance} from the category of compacts, as for any $n$, there exists a triangle $E \to F \to D \to \Sigma E$ with $E$ a compact object and $D$ in $\mathcal{R}_n$. We recall from \Cref{Remark on intution of metric} that $\mathcal{R}_n$ represents the ball of radius $2^{-n}$. Hence, the object $E$ is less than $2^{-n}$ distance away from $F$ for all $n$. This exactly mimics the closure of a subspace of a metric space. 
\end{remark}

We now define the notion of a $\mathcal{G}$-preapproximable triangulated category, which will be a prequel to the notion of $\mathcal{G}$-approximability, which is the main new definition of this work. We also show that the closure of the compacts is thick for a $\mathcal{G}$-preapproximable triangulated category.

\begin{definition}\label{Definition of G-preapproximability}
    Let $\mathsf{T}$ be a compactly generated triangulated category with a generating sequence $\mathcal{G}$ (see \Cref{Definition of generating sequence}),
   \emph{$\mathcal{G}$-preapproximable} if there exists an orthogonal metric $\mathcal{R}$ (see \Cref{Good Metric}), and an integer $A > 0$ such that,
    \begin{enumerate}
        \setcounter{enumi}{-1}
        \item For each integer $i$, we have that, $\mathcal{M}^{\mathcal{G}}_{i} \subseteq \bigcup_{|n|\leq A}\Sigma^n\mathcal{M}^{\mathcal{G}}_{i+1}$.
        \item $\mathcal{M}^{\mathcal{G}}_{i + A} \subseteq \mathcal{R}_{i}\cap \mathsf{T}^c$ for all $i \in \mathbb{Z}$.
        \item For each $n \in \mathbb{Z}$, $\HomT{\mathcal{G}^n}{\mathcal{R}_i} = 0$ for all $i >> 0 $.
    \end{enumerate}
\end{definition}

We begin with a technical lemma which will be useful later.

\begin{lemma}\label{Lemma Definition of G-approximability (1)}
    Let $\mathsf{T}$ be a compactly generated triangulated category with a generating sequence $\mathcal{G}$ (\Cref{Definition of generating sequence}), an orthogonal metric $\mathcal{R}$ (\Cref{Good Metric}), and an integer $A > 0$ satisfying the conditions of $\mathcal{G}$-preapproximability, see \Cref{Definition of G-preapproximability}. Then,
    \begin{enumerate}
        \item $\mathsf{T}^c \subseteq \bigcup_{n \in \mathbb{Z}}\Sigma^n \mathcal{R}_i $  for all $i \in \mathbb{Z}$.
        \item $\mathsf{T}^c \subseteq \overline{\mathsf{T}^c} \subseteq \bigcup_{n \in \mathbb{Z}}\mathcal{R}_n$, see \Cref{Definition closure of compacts} for the definition of $\overline{\mathsf{T}^c}$.
    \end{enumerate}
\end{lemma}
\begin{proof}
    By \Cref{Definition of G-preapproximability}(1), $\mathcal{R}_{i}\cap \mathsf{T}^c \supseteq \mathcal{M}^{\mathcal{G}}_{i + A}$ for all $i$. But, we have that 
    \[\bigcup_{n \in \mathbb{Z}}\Sigma^n \mathcal{R}_i \supseteq \bigcup_{n \in \mathbb{Z}}\Sigma^n \mathcal{M}^{\mathcal{G}}_{i + A} = \mathsf{T}^c\]
    where the equality follows from \Cref{Lemma for lemma on implicatiopn of strong generating sequence}(2) using \Cref{Definition of G-preapproximability}(0), which proves (1).

    For (2), note that,
    \[\mathsf{T}^c = \bigcup_{i \in \mathbb{Z}} \mathcal{G}^{(-\infty,i]} = \bigcup_{i \in \mathbb{Z}} \mathcal{M}^{\mathcal{G}}_i \subseteq \bigcup_{i \in \mathbb{Z}} \mathcal{R}_{i-A}\]
    see \Cref{Notation Subcategories for (pre)-generating sequences} and \Cref{Definition of generating sequence}(2) for the definition of $\mathcal{M}^{\mathcal{G}}$.
    This immediately implies that $\overline{\mathsf{T}^c} = \bigcap_{n \in \mathbb{Z}} (\mathsf{T}^c \star \mathcal{R}_n) \subseteq \bigcup_{i \in \mathbb{Z}} \mathcal{R}_i$.
\end{proof}

We now prove an easy lemma which will be useful to show that $\overline{\mathsf{T}^c}$ is triangulated.

\begin{lemma}\label{Lemma Compacts orthogonal to the metric}
    Let $\mathsf{T}$ be a triangulated category with a pre-generating sequence $\mathcal{G}$ (\Cref{Definition pre-generating sequence}), and an orthogonal metric $\mathcal{R}$ (\Cref{Good Metric}) such that for all $n\in \mathbb{Z}$, $\HomT{\mathcal{G}^n}{\mathcal{R}_i} = 0$ for all $i >> 0 $. Then, $\HomT{H}{\mathcal{R}_i} = 0$ for all $i >>0$, for any compact object $H$.
\end{lemma}
\begin{proof}
    Note that $\mathsf{T}^c = \bigcup_{n \geq 0}\mathcal{G}^{[-n,n]}$ (see \Cref{Notation Subcategories for (pre)-generating sequences}) as $\mathcal{G}$ is a pre-generating sequence. And so, there exists $n > 0$ such that $H \in \mathcal{G}^{[-n,n]}$. From the hypothesis $\HomT{\mathcal{G}[-n,n]}{\mathcal{R}_i} = 0$ for all $i >>0 $, and hence, $\HomT{\mathcal{G}^{[-n,n]}}{\mathcal{R}_i} = 0$ for all $i >>0 $. This implies that $\HomT{H}{\mathcal{R}_i} = 0$ for all $i >>0$, and we are done.
\end{proof}

\begin{proposition}\label{Proposition Closure of compacts is triangulated}
    Let $\mathsf{T}$ be a triangulated category with a pre-generating sequence $\mathcal{G}$, and an orthogonal metric $\mathcal{R}$ such that for all $n\in \mathbb{Z}$, $\HomT{\mathcal{G}^n}{\mathcal{R}_i} = 0$ for all $i >> 0 $. Then, $\overline{\mathsf{T}^c}$ (\Cref{Definition closure of compacts}) is a triangulated subcategory of $\mathsf{T}$.
\end{proposition}
\begin{proof}
    It is clear from the definition (\Cref{Definition closure of compacts}) that $\overline{\mathsf{T}^c}$ is closed under suspensions and desuspensions. So, we just need to show that $\overline{\mathsf{T}^c}$ is closed under extensions. Let $F_1, F_2 \in \overline{\mathsf{T}^c}$ and $F \in \mathsf{T}$, such that there is a triangle $F_1 \to F \to F_2 \to \Sigma F_1$. Choose and fix any integer $n$. Then, by \Cref{Definition closure of compacts}, there exist $E_2 \in \mathsf{T}^c$, and $D_2 \in \mathcal{R}_{n}$ such that $F_2 \in E_2 \star D_2$. By \Cref{Lemma Compacts orthogonal to the metric}, there exists an integer $b$ such that, $\HomT{E_2}{\mathcal{R}_i \cap \overline{\mathsf{T}^c}} = 0$ for all $i \geq n + b$.  Then, again by \Cref{Definition closure of compacts}, there exist objects $E_1$ and $D_1$ such that $F_1 \in E_1 \star  D_1$ with $E_1 \in \mathsf{T}^c$ and $D_1 \in \mathcal{R}_{ n + b + 1 }$. And so, $F \in E_1 \star D_1 \star E_2 \star D_2 $. 
    
    Now, let $C \in D_1 \star E_2$, that is, there exists a triangle $\Sigma^{-1}E_2 \to D_1 \to C \to E_2$. Now, as $\HomT{\Sigma^{-1}E_2}{D_1} = 0$, we get that $C \cong D_1 \oplus E_2 \in E_2 \star D_1$. And so, $D_1 \star E_2 = E_2 \star D_1$. This gives us that,
    \[F \in E_1 \star E_2 \star D_1 \star D_2 \subseteq \mathsf{T}^c\star \mathcal{R}_{n}\]
    As this is true for any $n \in \mathbb{Z}$, $F \in \overline{\mathsf{T}^c}$, and we are done.
\end{proof}

We now show that $\overline{\mathsf{T}^c}$ is thick for a $\mathcal{G}$-preapproximable triangulated category.
\begin{proposition}\label{Proposition Closure of compacts is thick}
    Let $\mathsf{T}$ be a triangulated category with a generating sequence $\mathcal{G}$ (\Cref{Definition pre-generating sequence}), an orthogonal metric $\mathcal{R}$ (\Cref{Good Metric}), and an integer $A >0$ satisfying the definition of a $\mathcal{G}$-preapproximable triangulated category, see \Cref{Definition of G-preapproximability}. Then, $\overline{\mathsf{T}^c}$ is a thick subcategory of $\mathsf{T}$.
\end{proposition}

\begin{proof}
    We already know that $\overline{\mathsf{T}^c} $ is triangulated by \Cref{Proposition Closure of compacts is triangulated}. We just need to show it is closed under direct summands. Let $F \oplus F' \in \overline{\mathsf{T}^c} $. We want to show that $F \in \overline{\mathsf{T}^c}
    $. We first show that $F \oplus \Sigma^{2i + 1}F \in \overline{\mathsf{T}^c} $ for all $i \in \mathbb{Z}$. We give the proof for $i \geq 0$, the proof for $i < 0$ goes similarly. 
    
    As $\operatorname{Cone}{(F \oplus F' \xrightarrow{0 \oplus 1} F\oplus F')} \cong F \oplus \Sigma F$, we have that $F \oplus \Sigma F \in \overline{\mathsf{T}^c} $ as $\overline{\mathsf{T}^c}$ is triangulated by \Cref{Proposition Closure of compacts is triangulated}, which gives us the claim for $i = 0$. 
    
    Now suppose we know the claim up to some non-negative integer $i$, that is, $F \oplus \Sigma^{2i + 1} F$ lies in $\overline{\mathsf{T}^c}$. Further, $\Sigma^{2i + 1}(F \oplus \Sigma F) \cong \Sigma^{2i + 1}F \oplus \Sigma^{2i + 2}F \in \overline{\mathsf{T}^c}$. Then, by \Cref{Proposition Closure of compacts is triangulated}, we get that, $\operatorname{Cone}{(\Sigma^{2i + 2}F \oplus \Sigma^{2i + 1}F \xrightarrow{0 \oplus 1} F\oplus \Sigma^{2i + 1}F)} \cong F\oplus \Sigma^{2(i + 1) + 1}F \in \overline{\mathsf{T}^c}$, and hence we get the induction step. So, we have shown that $F \oplus \Sigma^{2i + 1}F \in \overline{\mathsf{T}^c} $ for all $n \in \mathbb{Z}$.
    
    Now, fix an integer $n$. As by assumption $\mathsf{T}^c \subseteq \bigcup_{i \in \mathbb{Z}}\Sigma^i \mathcal{R}_n $ for any $n \in \mathbb{Z}$, there exists an integer $i$ such that $\Sigma^{2i + 2}F \in \mathcal{R}_{n}$. We can show this as follows; there exists an integer $j$ such that $\Sigma^j F \in \mathcal{R}_{n + 1}$. But, as $\mathcal{R}$ is an orthogonal metric, it is in particular an extended good metric, and so by \Cref{Good Metric}(2), $\Sigma^{-1}\mathcal{R}_{n + 1} \cup \mathcal{R}_{n + 1} \cup \Sigma\mathcal{R}_{n + 1} \subseteq \mathcal{R}_n$. So, in particular, $\Sigma^j F, \Sigma^{j + 1}F \in \mathcal{R}_n$, and $j$ or $j+1$ is even, which gives us the required $i$.
    
    From the triangle, 
    \[ F \oplus \Sigma^{2i + 1}F \to F \xrightarrow{0} \Sigma^{2i + 2 }F \to \Sigma F \oplus \Sigma^{2i + 2}F\]
    we get that, $F \in (F \oplus \Sigma^{2i + 1}F )\star ( \Sigma^{2i + 2 }F ) \subseteq \overline{\mathsf{T}^c} \star \mathcal{R}_n \subseteq \mathsf{T}^c \star \mathcal{R}_n $ by \Cref{Good Metric}(1). As this is true for all $n \in \mathbb{Z}$, $F \in \overline{\mathsf{T}^c}$.
\end{proof}

\section{$\mathcal{G}$-Approximability}

We now give the main definitions of this work. We refer the reader to the introduction for some intuition for the following definition, see \Cref{Example to explain approximability} and \Cref{Intuition for approximability}. 

\begin{definition}\label{Definition of G-approximability}
    Let $\mathsf{T}$ be a compactly generated triangulated category with a generating sequence $\mathcal{G}$ (see \Cref{Definition of generating sequence}).
    Then, we say $\mathsf{T}$ is  \emph{weakly $\mathcal{G}$-approximable} if there exists an orthogonal metric $\mathcal{R}$ (see \Cref{Good Metric}), and an integer $A > 0$ such that,
    \begin{enumerate}
        \item The triangulated category $\mathsf{T}$, the orthogonal metric $\mathcal{R}$, and the integer $A > 0$ satisfy the definition of $\mathcal{G}$-preapproximability, see \Cref{Definition of G-preapproximability}
        \item $\HomT{\mathcal{G}^n}{\mathcal{R}_i} = 0$ for all $n,i \in \mathbb{Z}$ such that $n + i \geq A$.
        \item For all $i \in \mathbb{Z} $, and $F \in \mathcal{R}_i$, there exists a triangle $E \to F \to D \to \Sigma E$ with $D \in \mathcal{R}_{i + 1}$ and $E \in \overline{\mathcal{G}}^{[-A-i,A-i]}$ (see \Cref{Notation Subcategories for (pre)-generating sequences}).
    \end{enumerate}
    We say $\mathsf{T}$ is \emph{$\mathcal{G}$-approximable} if we further have,
    \begin{enumerate}[label={  (\arabic*)$'$}]
    \setcounter{enumi}{2}
        \item For all $i \in \mathbb{Z} $, and $F \in \mathcal{R}_i$, there exists a triangle $E \to F \to D \to \Sigma E$ with $D \in \mathcal{R}_{i + 1}$ and $E \in \overline{\mathcal{G}}^{[-A-i,A-i]}_{A}$ (see \Cref{Notation Subcategories for (pre)-generating sequences}).
    \end{enumerate}
    Note that condition (2) is similar too, but stronger than, \Cref{Definition of G-preapproximability}(2).
\end{definition}

We also define a weaker notion, which will be enough to show the representability theorems. The only difference in this definition is that we only ask for objects in the closure of the compacts to be approximable.

\begin{definition}\label{Definition of G-quasiapproximability}
    Let $\mathsf{T}$ be a compactly generated triangulated category with a generating sequence $\mathcal{G}$ (see \Cref{Definition of generating sequence})
    Then, we say $\mathsf{T}$ is  \emph{weakly $\mathcal{G}$-quasiapproximable} if there exists an orthogonal metric $\mathcal{R}$ (see \Cref{Good Metric}), and an integer $A > 0$ such that,
    \begin{enumerate}
        \item The triangulated category $\mathsf{T}$, the orthogonal metric $\mathcal{R}$, and the integer $A > 0$ satisfy the definition of $\mathcal{G}$-preapproximability, see \Cref{Definition of G-preapproximability}.
        \item $\HomT{\mathcal{G}^n}{\mathcal{R}_i} = 0$ for all $n,i \in \mathbb{Z}$ such that $n + i \geq A$.
        \item For all $i \in \mathbb{Z} $, and $F \in \mathcal{R}_i \cap \overline{\mathsf{T}^c}$, there exists a triangle $E \to F \to D \to \Sigma E$ with $D \in \mathcal{R}_{i + 1}\cap \overline{\mathsf{T}^c}$ and $E \in \overline{\mathcal{G}}^{[-A-i,A-i]}$.
    \end{enumerate}
    We say $\mathsf{T}$ is \emph{$\mathcal{G}$-quasiapproximable} if we further have,
    \begin{enumerate}[label={(\arabic*)$'$}]
    \setcounter{enumi}{2}
        \item For all $i \in \mathbb{Z} $, and $F \in \mathcal{R}_i\cap \overline{\mathsf{T}^c}$, there exists a triangle $E \to F \to D \to \Sigma E$ with $D \in \mathcal{R}_{i + 1}\cap \overline{\mathsf{T}^c}$ and $E \in \overline{\mathcal{G}}^{[-A-i,A-i]}_{A}$.
    \end{enumerate}
\end{definition}
It will follow from \Cref{Proposition 3 of section 5} that any (weakly) $\mathcal{G}$-approximable triangulated category is (weakly) $\mathcal{G}$-quasiapproximable.

\begin{remark}\label{Remark conditions of G-approximability}
    For the ease of the reader, we explicitly write down the conditions required in the definition of $\mathcal{G}$-approximability. Note that analogous statements hold for $\mathcal{G}$-quasiapproximability.
    
    Let $\mathsf{T}$ be a compactly generated triangulated category with a generating sequence $\mathcal{G}$ (see \Cref{Definition of generating sequence})
    Then, $\mathsf{T}$ is \emph{weakly $\mathcal{G}$-approximable} if there exists an orthogonal metric $\mathcal{R}$ (see \Cref{Good Metric}), and an integer $A > 0$ such that,
    \begin{enumerate}
    \setcounter{enumi}{-1}
        \item For each integer $i$, we have that, $\mathcal{M}^{\mathcal{G}}_{i} \subseteq \bigcup_{|n|\leq A}\Sigma^n\mathcal{M}^{\mathcal{G}}_{i+1}$ see \Cref{Definition of generating sequence}(2) for the definition of $\mathcal{M}^{\mathcal{G}}$.
        \item $\mathcal{M}^{\mathcal{G}}_{i + A} \subseteq \mathcal{R}_{i}\cap \mathsf{T}^c$ for all $i \in \mathbb{Z}$.
        \item $\HomT{\mathcal{G}^n}{\mathcal{R}_i} = 0$ for all $n,i \in \mathbb{Z}$ such that $n + i \geq A$. 
        \item For all $i \in \mathbb{Z} $, and $F \in \mathcal{R}_i$, there exists a triangle $E \to F \to D \to \Sigma E$ with $D \in \mathcal{R}_{i + 1}$ and $E \in \overline{\mathcal{G}}^{[-A-i,A-i]}$.
    \end{enumerate}
    We say $\mathsf{T}$ is \emph{$\mathcal{G}$-approximable} if we further have,
    \begin{enumerate}[label={(\arabic*)$'$}]
    \setcounter{enumi}{2}
        \item For all $i \in \mathbb{Z} $, and $F \in \mathcal{R}_i$, there exists a triangle $E \to F \to D \to \Sigma E$ with $D \in \mathcal{R}_{i + 1}$ and $E \in \overline{\mathcal{G}}^{[-A-i,A-i]}_{A}$.
    \end{enumerate}
\end{remark}

\begin{convention}\label{Convention on approximability}
    We will use the following conventions to make notation less cumbersome.

    \begin{enumerate}
        \item We write approx (resp. quasiapprox) instead of approximable/approximability (resp.\ quasiapproximable/quasiapproximability).
        
        \item Let $\mathsf{T}$ be a triangulated category. We say $\mathsf{T}$ is weakly $\mathcal{G}$-approx (resp. $\mathcal{G}$-approx) if there exists a generating sequence $\mathcal{G}$ (\Cref{Definition of generating sequence}) for $\mathsf{T}$, and $\mathsf{T}$ is weakly $\mathcal{G}$-approx (resp. $\mathcal{G}$-approx), see \Cref{Definition of G-approximability}. Similar convention holds for (weak) $\mathcal{G}$-quasiapprox, see \Cref{Definition of G-quasiapproximability}.
        \item Finally, we say $(\mathsf{T}, \mathcal{R}, A)$ is weakly $\mathcal{G}$-approx (resp.\ $\mathcal{G}$-approx) if $\mathsf{T}$ is a compactly generated triangulated category, $\mathcal{G}$ is a generating sequence for $\mathsf{T}$, $\mathcal{R}$ is an orthogonal metric on $\mathsf{T}$, and $A$ is a positive integer satisfying the conditions of $\mathcal{G}$-approx (resp. $\mathcal{G}$-approx). Similar convention holds for (weak) $\mathcal{G}$-quasiapprox.
    \end{enumerate}
    
\end{convention}

\begin{remark}\label{Remark on Definition of G-preapproximability(1)}
    From  \Cref{Definition of G-preapproximability}(1), we get that $\mathcal{M}^{\mathcal{G}} \preceq_{\mathbb{N}} \mathcal{R}\cap \mathsf{T}^c$ as extended good metrics on $\mathsf{T}^c$, see \Cref{Definition equivalence relation on extended good metrics}(2) for the definition of $\mathcal{M}^{\mathcal{G}}$. Note that this holds for any weakly $\mathcal{G}$-approx (or weakly $\mathcal{G}$-quasiapprox) triangulated category .
\end{remark}

We now prove some results related to the ideas introduced above. We begin by showing that we can get \say{arbitrarily good approximations} of objects in $\mathcal{R}_i$.

\begin{proposition}\label{Proposition Approximating Sequence}
    Let $\mathsf{T}$ be a compactly generated triangulated category with a generating sequence $\mathcal{G}$ (\Cref{Definition of generating sequence}), an orthogonal metric $\mathcal{R}$ (see \Cref{Good Metric}) and an integer $A > 0$ satisfying \Cref{Definition of G-approximability}(3). Then, with notation as in \Cref{Notation Subcategories for (pre)-generating sequences},
    \begin{enumerate}
        \item Let $i < j$ be any positive integers. For any $F \in \mathcal{R}_i$, there exists a triangle $E \to F \to D \to \Sigma E$ with $E \in \overline{\mathcal{G}}^{[-A-j+1,A-i]}$ and $D \in \mathcal{R}_j$. If we further assume that $\mathsf{T}$ satisfies \Cref{Definition of G-approximability}$(3')$, we can produce the triangle with $E \in \overline{\mathcal{G}}^{[-A-j+1,A-i]}_{(j-i)A}$.
        \item For any $F \in \mathcal{R}_i$, there exists a sequence $E_1 \to E_2 \to E_3 \to E_4 \to \cdots $ of objects in $\mathsf{T}$ with compatible maps to $F$, such that for all $n\geq 1$, in the triangle $E_n \to F \to D_n \to \Sigma E_n$, we have $E_n \in \overline{\mathcal{G}}^{[-A-n-i+1,A-i]}$ and $D_n \in \mathcal{R}_{i+n}$. If we further assume that $\mathsf{T}$ satisfies \Cref{Definition of G-approximability}$(3')$, we can get the sequence with $E_n \in \overline{\mathcal{G}}^{[-A-n-i+1,A-i]}_{nA}$ for all $n \geq 1$.
    \end{enumerate}
    Further, the analogous result for $F \in \mathcal{R}_i \cap \overline{\mathsf{T}^c}$ holds if we instead assume that $\mathsf{T}$ is a compactly generated triangulated category with a generating sequence $\mathcal{G}$, an orthogonal metric $\mathcal{R}$ and an integer $A > 0$ satisfying \Cref{Definition of G-quasiapproximability}(3) (resp.\ \Cref{Definition of G-quasiapproximability}(3$'$)) instead of \Cref{Definition of G-approximability}(3) (resp.\ \Cref{Definition of G-approximability}(3$'$)).
\end{proposition}
\begin{proof}
    Clearly $(2) \implies (1)$. So, it is enough to show $(2)$.
    We proceed by induction on $n \geq 1$. Let $F \in \mathcal{R}_i$. Then, we get the case $n=1$ by \Cref{Definition of G-approximability}(3) and (3$'$). Now, suppose we have got a sequence $E_1 \to \cdots \to E_n$ with the required properties. So, we have the triangle $E_n \to F \to D_n \to \Sigma E_n$ with $E_n \in \overline{\mathcal{G}}^{[-A-n-i+1,A-i]}$ and $D_n \in \mathcal{R}_{i + n}$. If $\mathsf{T}$ further satisfies \Cref{Definition of G-approximability}$(2')$, we can assume that $E\in \overline{\mathcal{G}}^{[-A-n-i+1,A-i]}_{nA}$.
    
    By \Cref{Definition of G-approximability}(3), we know that there exists a triangle $E' \to D_n \to D_{n + 1} \to \Sigma E'$ with $E' \in \overline{\mathcal{G}}^{[-A-n-i,A-n-i]}$ and $D_{n + 1} \in \mathcal{R}_{i + n + 1}$. If $\mathsf{T}$ further satisfies \Cref{Definition of G-approximability}$(3')$, we can arrange so that $E' \in \overline{\mathcal{G}}^{[-A-n-i,A-n-i]}_{A}$.
    
    Using the octahedral axiom on $F \to D_n \to D_{n + 1}$, we get,
    \[\begin{tikzcd}
	E_n & E_{n + 1} & E' \\
	E_n & F & D_n \\
	& D_{n + 1} & D_{n + 1}
	\arrow[from=1-1, to=1-2]
	\arrow[from=1-2, to=1-3]
	\arrow[from=1-1, to=2-1, equals]
	\arrow[from=2-1, to=2-2]
	\arrow[from=1-2, to=2-2]
	\arrow[from=2-2, to=3-2]
	\arrow[from=3-2, to=3-3, equals]
	\arrow[from=2-2, to=2-3]
	\arrow[from=2-3, to=3-3,]
	\arrow[from=1-3, to=2-3]
\end{tikzcd}\]
    Note that here 
    \[E_{n + 1} \in \overline{\mathcal{G}}^{[-A-n-i+1,A-i]} \star \overline{\mathcal{G}}^{[-A-n-i,A-n-i]} \subseteq \overline{\mathcal{G}}^{[-A-n-i,A-i]}\] 
    If $\mathsf{T}$ further satisfies \Cref{Definition of G-approximability}$(2')$, we get that
    \[E_{n + 1} \in \overline{\mathcal{G}}^{[-A-n-i+1,A-i]}_{nA} \star \overline{\mathcal{G}}^{[-A-n-i,A-n-i]}_{A} \subseteq \overline{\mathcal{G}}^{[-A-n-i,A-i]}_{(n + 1)A}\] 
    and so, we are done via induction.
\end{proof}

\begin{lemma}\label{Lemma Intersection of the metric is zero}
    Let $\mathsf{T}$ be a compactly generated triangulated category with a pre-generating sequence $\mathcal{G}$ (\Cref{Definition pre-generating sequence}) and an extended good metric $\mathcal{R}$ (\Cref{Good Metric}) such that, for all $n \in \mathbb{Z}$, $\HomT{\mathcal{G}^n}{\mathcal{R}_i} = 0$ for $i >> 0$.
    Then, 
    \begin{enumerate}
        \item $\bigcap_{i \in \mathbb{Z}}\mathcal{R}_i = \{0\}$
        \item Let $E_1 \to E_2 \to E_3 \to E_4 \to \cdots $ be a sequence with compatible maps to an object $F$, such that for each $n \in \mathbb{Z}$, $\operatorname{Cone}(E_i \to F) \in \mathcal{R}_{n}$ for all $i >> 0$. Then the non-canonical map $\hocolim E_n \to F$ is an isomorphism.
    \end{enumerate} 
\end{lemma}
\begin{proof}
    \begin{enumerate}
        \item Let $F \in \bigcap_{i \in \mathbb{Z} }\mathcal{R}_i$. Then, $\HomT{\mathcal{G}^n}{F}=0$ for any $n \in \mathbb{Z}$ by hypothesis. As $(\bigcup_{j \in \mathbb{Z}}\mathcal{G}^j)^{\perp} = 0$, see \Cref{Definition pre-generating sequence}, we get that $F = 0$.
        \item Let $j \in \mathbb{Z}$ be any integer, and let $G \in \mathcal{G}^j$. Note that $\Sigma^i \mathcal{R}_{j+|i|} \subseteq \mathcal{R}_{j}$ for any integers $i$ and $j$, see \Cref{Good Metric}(2). So, by the hypothesis, 
        there exists an integer $a>0$ such that $\HomT{\Sigma^{r}G}{\mathcal{R}_p} = 0$  for all $p \geq a$ and $-1 \leq r \leq 2$. 
        
        As $\operatorname{Cone}(E_q \to F) \in \mathcal{R}_a$ for all $q >> 0$, by the long exact sequence we get from applying the functor $\HomT{G}{-}$ to the triangle $\Sigma^{-1} \operatorname{Cone}(E_q \to F) \to E_q \to F \to \operatorname{Cone}(E_q \to F)$, we get that the natural map $\HomT{\Sigma^r G}{E_{q}} \to \HomT{\Sigma^r G}{F}$ is an isomorphism for all $q >> 0$ and $-1 \leq r \leq 1$. Let $E \colonequals \hocolim E_n$. As $\HomT{\Sigma^r G}{E} = \colim\HomT{\Sigma^r G}{E_n}$ for any $r \in \mathbb{Z}$ by \cite[Lemma 2.8]{Neeman:1996}, we get that the natural map gives an isomorphism $\HomT{\Sigma^r G}{E} \cong \HomT{\Sigma^r G}{F}$ for $-1 \leq r \leq 1$.  

        So, by the long exact sequence we get from applying the functor $\HomT{G}{-}$ to the triangle $E \to F \to \operatorname{Cone}(E \to F) \to E[1]$, we get that $\HomT{G}{\operatorname{Cone}(E \to F)} = 0$. As $G \in \bigcup_{j \in \mathbb{Z}}\mathcal{G}^j$ was arbitrary and as $\bigcup_{j \in \mathbb{Z}}\mathcal{G}^j$ generates $\mathsf{T}$ (see \Cref{Definition pre-generating sequence}), we get that $\operatorname{Cone}(E \to F) \cong 0$. Therefore, $E \to F$ is an isomorphism, and we are done.
    \end{enumerate}
\end{proof}
\begin{corollary}\label{Corollary approximating by generating sequence}
   Let $(\mathsf{T}, \mathcal{R}, A)$ be weakly $\mathcal{G}$-approx, see \Cref{Convention on approximability}. For any object $F \in  \mathcal{R}_i$ consider the sequence $E_1 \to E_2 \to E_3 \to E_4 \to \cdots $ with compatible maps to $F$ we get from \Cref{Proposition Approximating Sequence}(2). Then, $\hocolim E_n \cong F $. In particular, $F \in \overline{\mathcal{G}}^{(-\infty,A-i+1]}$, see \Cref{Notation Subcategories for (pre)-generating sequences}.

   The analogous result holds for $F \in \mathcal{R}_i \cap \overline{\mathsf{T}^c}$ if instead $(\mathsf{T}, \mathcal{R}, A)$ is weakly $\mathcal{G}$-quasiapprox.
\end{corollary}
\begin{proof}
    Let $F \in \mathcal{R}_i$ be any object. Consider the sequence $E_1 \to E_2 \to E_3 \to E_4 \to \cdots $  we get from \Cref{Proposition Approximating Sequence}(2). We have that $\operatorname{Cone}(E_n \to F) \in \mathcal{R}_{n + i}$ for all $n$ by \Cref{Proposition Approximating Sequence}. So, by \Cref{Lemma Intersection of the metric is zero}, $\hocolim E_i \cong F$. Note that $E_i \in \overline{\mathcal{G}}^{(-\infty,A-i]}$ for each $i \in \mathbb{Z}$ again by \Cref{Proposition Approximating Sequence}(2). From the triangle,
    \[\bigoplus_{i \in \mathbb{Z} }E_i \to \bigoplus_{i \in \mathbb{Z} }E_i \to \hocolim E_i \to \Sigma \bigoplus_{i \in \mathbb{Z} }E_i\]
    we get that $F \cong \hocolim E_i \in \overline{\mathcal{G}}^{(-\infty,A-i+1]}$, see \Cref{Notation Subcategories for (pre)-generating sequences}.
\end{proof}

We now show that the orthogonal metric in the definition of approximability lies in the preferred $\mathbb{N}$-equivalence class of orthogonal metrics (see \Cref{Definition preferred equivalence class of orthogonal metrics}). Further, we show that the extended good metric it induces on $\mathsf{T}^c$ is equivalent to the metric $\mathcal{M}^{\mathcal{G}}$. This is an important fact, as, for example, equivalent metrics give rise to equivalent categories via the completion machinery, as the Cauchy sequences remain the same, see \Cref{Definition Cauchy sequence} and \Cref{Definition completions}.

\begin{proposition}\label{Proposition approximable implies preferred equivalence class}
   Let $(\mathsf{T},\mathcal{R},A)$ be weakly $\mathcal{G}$-quasiapprox, see \Cref{Convention on approximability}. Then,
    \begin{enumerate}
        \item The metrics $\mathcal{M}^{\mathcal{G}}$ (\Cref{Definition of generating sequence}(2)) and $\mathcal{R} \cap \mathsf{T}^c$ are $\mathbb{N}$-equivalent, see \Cref{Definition equivalence relation on extended good metrics} for the definition of the equivalence relation.
        \item $\mathcal{R} \cap \overline{\mathsf{T}^c} $ is $\mathbb{N}$-equivalent to the metric $\mathcal{R}^{\mathcal{G}} \cap \overline{\mathsf{T}^c}$ (see \Cref{Definition of generating sequence})  as metrics on $\overline{\mathsf{T}^c}$. This implies that $\mathcal{M}^{\mathcal{G}}$ is $\mathbb{N}$-equivalent to the metric $\mathcal{R}^{\mathcal{G}} \cap \mathsf{T}^c$ on $\mathsf{T}^c$.
    \end{enumerate}
    Further, if $(\mathsf{T},\mathcal{R},A)$ is weakly $\mathcal{G}$-approx (see \Cref{Definition of G-quasiapproximability}), then,
    \begin{enumerate}
    \setcounter{enumi}{2}
        \item $\mathcal{R}$ lies in the preferred $\mathbb{N}$-equivalence class of orthogonal metrics on $\mathsf{T}$, that is, it is equivalent to the orthogonal metric $\mathcal{R}^{\mathcal{G}}$.
    \end{enumerate}
\end{proposition}
\begin{proof}
    \begin{enumerate}
        \item We already know $\mathcal{M}^{\mathcal{G}} \preceq_{\mathbb{N}} \mathcal{R} \cap \mathsf{T}^c$ by \Cref{Remark on Definition of G-preapproximability(1)}. For the other direction, we need to show that there exists a integer $B>0$ such that $\mathcal{R}_{i + B} \cap \mathsf{T}^c \subseteq \mathcal{M}^{\mathcal{G}}_i$ for all $i\geq 0$. We claim that the integer $B$ can be chosen to be $A+1$. Let $F \in \mathcal{R}_{i + A + 1} \cap \mathsf{T}^c$ for any $i \geq 0$. Then, by \Cref{Corollary approximating by generating sequence}, $F \in \overline{\mathcal{G}}^{(-\infty, -i]}$ (see \Cref{Notation Subcategories for (pre)-generating sequences}). And so, $F \in \overline{\mathcal{G}}^{(-\infty, -i]} \cap \mathsf{T}^c = \mathcal{G}^{(-\infty, -i]}$, see \Cref{Remark big categories intersection compacts gives small categories}. But, by definition, $\mathcal{M}^{\mathcal{G}}_{i} = \mathcal{G}^{(-\infty, -i]} $. And so, $\mathcal{R}_{i + A + 1} \cap \mathsf{T}^c \subseteq \mathcal{M}^{\mathcal{G}}_i$, and we are done.
        \item From \Cref{Definition of G-quasiapproximability}(1), we have that $\mathcal{M}^{\mathcal{G}}_{i + A} \subseteq \mathcal{R}_{i} \cap \mathsf{T}^c \subseteq \mathcal{R}_{i}$ for all $i\geq 0$. And so, for all $i \geq 0 $,
        \[\mathcal{R}^{\mathcal{G}}_{i + A} = \ ^{\perp}\left((\mathcal{M}^{\mathcal{G}}_{i + A})^{\perp}\right) \subseteq \ ^{\perp}\left(\mathcal{R}_{i}^{\perp}\right) = \mathcal{R}_{i} \] 
        Conversely, we have that,
        \begin{displaymath}
            \mathcal{R}_{i+A+1}\cap \overline{\mathsf{T}^c} \subseteq \overline{\mathcal{G}}^{(-\infty,-i]} \subseteq \ ^{\perp}\left(\big(\overline{\mathcal{G}}^{(-\infty, -i]}\big)^{\perp}\right) = \ ^{\perp}\left(\mathcal{G}(-\infty, -i]^{\perp}\right) = \mathcal{R}^{\mathcal{G}}_i 
        \end{displaymath}
        for all $i$ where we get the first inclusion by \Cref{Corollary approximating by generating sequence},
        \item From \Cref{Definition of G-approximability}(1), we have that $\mathcal{M}^{\mathcal{G}}_{i + A} \subseteq \mathcal{R}_{i} \cap \mathsf{T}^c \subseteq \mathcal{R}_{i}$ for all $i$. And so, for all $i$,
        \[\mathcal{R}^{\mathcal{G}}_{i + A} = \ ^{\perp}\left((\mathcal{M}^{\mathcal{G}}_{i + A})^{\perp}\right) \subseteq \ ^{\perp}\left(\mathcal{R}_{i}^{\perp}\right) = \mathcal{R}_{i} \] 
        Conversely, by \Cref{Corollary approximating by generating sequence}, $\mathcal{R}_{i+A+1} \subseteq \overline{\mathcal{G}}^{(-\infty, -i]}$ for all $i \geq 0$. And so,
        \begin{displaymath}
            \mathcal{R}_{i + A +1 } = \ ^{\perp}\left((\mathcal{R}_{i + A + 1 })^{\perp}\right) \subseteq \ ^{\perp}\left(\big(\overline{\mathcal{G}}^{(-\infty, -i]}\big)^{\perp}\right) = \ ^{\perp}\left(\mathcal{G}(-\infty, -i]^{\perp}\right) = \mathcal{R}^{\mathcal{G}}_i
        \end{displaymath}  
        for all $i \geq 0$.
    \end{enumerate}
    
\end{proof}

We now give a converse to \Cref{Proposition approximable implies preferred equivalence class}(3). That is, we show that we can check approximability with any orthogonal metric in the preferred equivalence class of orthogonal metrics.

\begin{proposition}\label{Proposition G-approximability from orthogonal metric in preferred equivalence class}
    Let $\mathsf{T}$ be a weakly $\mathcal{G}$-approx triangulated category (see \Cref{Convention on approximability}), and $\mathcal{R}'$ be a metric in the preferred $\mathbb{N}$-equivalence class of orthogonal metrics on $\mathsf{T}$ (see \Cref{Definition preferred equivalence class of orthogonal metrics}). Then, there exists an integer $A' > 0$ such that $(\mathsf{T},\mathcal{R}',A')$ is weakly $\mathcal{G}$-approx. If $\mathsf{T}$ is further $\mathcal{G}$-approx, then we can choose the integer $A'$ such that $(\mathsf{T},\mathcal{R}',A')$ is $\mathcal{G}$-approx.

    The analogous result holds for (weak) $\mathcal{G}$-approx.
\end{proposition}
\begin{proof}
    As $\mathsf{T}$ is weakly $\mathcal{G}$-approx, there exists an orthogonal metric $\mathcal{R}$ and an integer $A > 0$ satisfying \Cref{Remark conditions of G-approximability}(0)-(3). By \Cref{Proposition approximable implies preferred equivalence class}(2), $\mathcal{R}$ lies in the preferred $\mathbb{N}$-equivalence class of orthogonal metrics. As $\mathcal{R}'$ also lies in the preferred $\mathbb{N}$-equivalence class, we get that $\mathcal{R}$ and $\mathcal{R}'$ are $\mathbb{N}$-equivalent. So, there exists an integer $B>0$ such that $\mathcal{R}'_{i + B} \subseteq \mathcal{R}_{i}$ and $\mathcal{R}_{i + B} \subseteq \mathcal{R}'_{i}$ for all $i \geq 0 $. We now show that the conditions (0)-(3) of \Cref{Remark conditions of G-approximability} hold with the orthogonal metric $\mathcal{R}'$ and the integer $A' = A + B  > 0$. Note that \Cref{Remark conditions of G-approximability}(0) holds trivially as it does not involve the orthogonal metric.
    \begin{enumerate}
        \item $\mathcal{M}^{\mathcal{G}}_{i + A + B} \subseteq \mathcal{R}_{i + B}\cap \mathsf{T}^c \subseteq \mathcal{R}'_{i}\cap \mathsf{T}^c$ for all $i \in \mathbb{Z}$.
        \item For all $n \in \mathbb{Z}$, $\HomT{\mathcal{G}^n}{\mathcal{R}'_i} = 0$ for $n + i \geq A+B $, as $\mathcal{R}'_i \subseteq \mathcal{R}_{i - B}$ for all $i \in \mathbb{Z}$.
        \item Let $i \geq 0$, and $F \in \mathcal{R}'_{i} \subset \mathcal{R}_{i - B}$. By \Cref{Proposition Approximating Sequence}(1), there exists a triangle $E \to F \to D \to \Sigma E$ with $E \in \overline{\mathcal{G}}^{[-A-(i+B),A-i+B]}$ and $D \in \mathcal{R}_{i+B+1} \subseteq \mathcal{R}'_{i + 1}$. 
    \end{enumerate}
    
    If $\mathsf{T}$ is further $\mathcal{G}$-approx with the orthogonal metric $\mathcal{R}$ and integer $A > 0$ satisfying the conditions of $\mathcal{G}$-approx, then, 
    \begin{enumerate}[label={(\arabic*$'$)}]
        \setcounter{enumi}{2}
        \item Let $i \geq 0$, and $F \in \mathcal{R}'_{i} \subset \mathcal{R}_{i - B}$. By \Cref{Proposition Approximating Sequence}(1), there exists a triangle $E \to F \to D \to \Sigma E$ with $E \in \overline{\mathcal{G}}^{[-A-(i+B),A-i+B]}_{(2B+1)A}$ and $D \in \mathcal{R}_{i+B+1} \subseteq \mathcal{R}'_{i + 1}$.
    \end{enumerate}
    And so  we have shown $\mathcal{G}$-approx using the orthogonal metric $\mathcal{R}'$ and the integer $A' = 2AB + A  > 0$.
\end{proof}
The following can be proved analogously to \Cref{Proposition G-approximability from orthogonal metric in preferred equivalence class}. 
\begin{remark}
    Let $\mathsf{T}$ be a weakly $\mathcal{G}$-approx triangulated category (see \Cref{Definition of G-approximability}), and let $\mathcal{G}'$ be a generating sequence $\mathbb{N}$-equivalent to $\mathcal{G}$ (\Cref{Definition equivalence class of generating sequences}) and $\mathcal{R}'$ be a metric in the preferred $\mathbb{N}$-equivalence class of orthogonal metrics on $\mathsf{T}$ (see \Cref{Definition preferred equivalence class of orthogonal metrics}). Then, there exists an integer $A' > 0$ such that $(\mathsf{T},\mathcal{R}',A')$ is weakly $\mathcal{G}'$-approx. If $\mathsf{T}$ was $\mathcal{G}$-approx, we can further choose $A'$ such that $(\mathsf{T},\mathcal{R}',A')$ is $\mathcal{G}'$-approx. The analogous result also holds for (weak) $\mathcal{G}$-quasiapprox.
\end{remark}

We now record a lemma to be used later.

\begin{lemma}
    \label{Lemma on implication of strong generating sequence}
    Let $(\mathsf{T},\mathcal{R},A)$ be  weakly $\mathcal{G}$-quasiapprox. Then, there exists $b\geq 0$ such that for any integer $i$ and positive integer $j$, $\mathcal{R}_{i} \cap \mathsf{T}^c  \subseteq \bigcup_{|n| \leq (j + 2b)A}\Sigma^n \mathcal{R}_{i + j} \cap \mathsf{T}^c$.
\end{lemma}
\begin{proof}
    By \Cref{Proposition approximable implies preferred equivalence class}(1), there exists an integer $b \geq 0$ so that for all $i$, 
    \[\mathcal{M}^{\mathcal{G}}_{i + b} \subseteq \mathcal{R}_{i} \cap \mathsf{T}^c \subseteq \mathcal{M}^{\mathcal{G}}_{i - b}\]
    see \Cref{Definition of generating sequence}(2) for the definition of $\mathcal{M}^{\mathcal{G}}$. By \Cref{Lemma for lemma on implicatiopn of strong generating sequence}(1), 
    \[\mathcal{R}_{i} \cap \mathsf{T}^c \subseteq \mathcal{M}^{\mathcal{G}}_{i - b} \subseteq \bigcup_{|n|\leq (j+2b)A}\Sigma^{n}\mathcal{M}^{\mathcal{G}}_{i + b +j} \subseteq \bigcup_{|n|\leq (j+2b)A}\Sigma^{n}\mathcal{R}_{i+j} \cap \mathsf{T}^c\]
    which is what we needed to show.
\end{proof}
As might be clear from the name, $\mathcal{G}$-approximability is closely related to the notion of approximability introduced by Neeman, see \cite[Definition 0.21]{Neeman:2021b}. In fact, approximability is a special case of $\mathcal{G}$-approximability. The connection is explored below.

\begin{remark}\label{Remark on equivalence relation in approximability and G-approximability}
    Let $\mathsf{T}$ be a triangulated category with a single compact generator $G$. We define the finite generating sequence $\mathcal{G}$ by $\mathcal{G}^n \colonequals \{\Sigma^{-n}G\} $ for all $n \in \mathbb{Z}$ (see \Cref{Definition of generating sequence}). Note that by \Cref{Examples of N-equivalent generating sequences}, if we choose a different compact generator we get a $\mathbb{N}$-equivalent generating sequence (\Cref{Definition equivalence class of generating sequences}). Then, it is immediate from the definitions that,
    \begin{enumerate}
        \item There is a bijection between the preferred equivalence class of t-structures on $\mathsf{T}$, see \Cref{Definition preferred equivalence class of t-structures}, and those orthogonal metrics $\mathcal{R}$ in the preferred $\mathbb{N}$-equivalence class of orthogonal metrics (\Cref{Definition preferred equivalence class of orthogonal metrics}), such that $\mathcal{R}_n = \Sigma^{n}\mathsf{U}$ for a t-structure $(\mathsf{U},\mathsf{V})$. In particular, the t-structure generated by $G$ (see \Cref{Theorem Compactly generated t-structure}) is exactly $\big(\mathcal{R}_0^{\mathcal{G}},(\Sigma \mathcal{R}_0^{G})^{\perp}\big)$ (see \Cref{Definition of generating sequence} for the definition of $\mathcal{R}^{\mathcal{G}}$).
        \item The closure of the compacts $\overline{\mathsf{T}^c} = \mathsf{T}^-_c$. See \Cref{Definition T^-c and T^b_c}(2) for the definition of $\mathsf{T}^-_c$, and \Cref{Definition closure of compacts} for the definition of $\overline{\mathsf{T}^c}$.
    \end{enumerate}
    
\end{remark}

\begin{theorem}\label{Theorem Approximability and G-approximability}
    Let $\mathsf{T}$ be a triangulated category with a single compact generator $G$. We define the finite generating sequence $\mathcal{G}$ by $\mathcal{G}^n \colonequals \{\Sigma^{-n}G\} $ for all $n \in \mathbb{Z}$ (\Cref{Definition of generating sequence}). Then,  
    \begin{enumerate}[label={(\roman*)}]
        \item $\mathsf{T}$ is a weakly approximable triangulated category (\Cref{Definition approximability}) if and only if $\mathsf{T}$ is weakly $\mathcal{G}$-approximable (\Cref{Definition of G-approximability}).
        \item $\mathsf{T}$ is an approximable triangulated category (\Cref{Definition approximability}) if and only if $\mathsf{T}$ is $\mathcal{G}$-approximable (\Cref{Definition of G-approximability}). 
    \end{enumerate}
\end{theorem}
\begin{proof}
    Both (i) and (ii) follow the same way, so we only prove (i). Suppose $\mathsf{T}$ is weakly $\mathcal{G}$-approximable. Then, as mentioned in \Cref{Remark on equivalence relation in approximability and G-approximability}(1), the t-structure compactly generated by $G$ gives $\mathcal{R}^{\mathcal{G}}$, which is a metric in the preferred $\mathbb{N}$-equivalence class of orthogonal metrics, under the bijection given in \Cref{Remark on equivalence relation in approximability and G-approximability}(1). By \Cref{Proposition G-approximability from orthogonal metric in preferred equivalence class}, $\mathsf{T}$ is weakly $G$-approximable with respect to any orthogonal metric in the preferred equivalence class of orthogonal metrics, in particular with $\mathcal{R}^{\mathcal{G}}$. It is easy to see that this gives us that $\mathsf{T}$ is weakly approximable. \Cref{Remark conditions of G-approximability}(1), (2), and (3) imply \Cref{Definition approximability}(1), (2), and (3) respectively with the t-structure generated by $G$, see \Cref{Theorem Compactly generated t-structure}.

    Conversely, suppose $\mathsf{T}$ is weakly approximable. So, there exists a t-structure $(\mathsf{T}^{\leq 0},\mathsf{T}^{\geq 0})$, and an integer $A > 0$ such that the conditions (1)-(3) of \Cref{Definition approximability} are satisfied with the compact generator $G$. We will show that $\mathsf{T}$ is weakly $\mathcal{G}$-approximable by showing the conditions (0)-(3) of \Cref{Remark conditions of G-approximability} hold with the orthogonal good metric $\mathcal{R}_i = \Sigma^i \mathsf{T}^{\leq 0}$ for all $i \in \mathbb{Z}$ and the integer $A > 0$.
    \begin{enumerate}
    \setcounter{enumi}{-1}
        \item As $\mathcal{G}^n = \{\Sigma^{-n }G\}$, we get that $\mathcal{M}^{\mathcal{G}}_i = \Sigma^{-1} \mathcal{M}^{\mathcal{G}}_{i+1}$ for all $i \in \mathbb{Z}$, as $\mathcal{M}^{\mathcal{G}}_i = \mathcal{G}^{(-\infty,-i]} = \langle G \rangle^{(-\infty,-i]}$ by definition, see \Cref{Definition of generating sequence}(2) for the definition of $\mathcal{M}^{\mathcal{G}}$. Also see, \Cref{Notation Subcategories for (pre)-generating sequences} and \Cref{Notation from Neeman's paper}.
        \item As $\Sigma^A G \in \mathsf{T}^{\leq 0} = \mathcal{R}_0$, $\mathcal{M}^{\mathcal{G}}_{i+A} = \langle G \rangle^{(-\infty,-i-A]}\subseteq \mathcal{R}_i \cap \mathsf{T}^c$ for all $i \in \mathbb{Z}$, see \Cref{Notation from Neeman's paper}.
        \item As $\HomT{\Sigma^{-i} G}{\mathsf{T}^{\leq 0}} = 0$ for all $i \geq A$, we get that,
        \[\HomT{\mathcal{G}^n}{\mathcal{R}_{i}} = \HomT{\Sigma^{-n} G}{\mathsf{T}^{\leq -i}}=0\] for all $n + i \geq A$.
        \item For any $i \in \mathbb{Z}$ and $F \in \mathcal{R}_{i} = \mathsf{T}^{\leq -i}$, we get that $\Sigma^{-i}F \in \mathsf{T}^{\leq 0}$. So, by the \Cref{Definition approximability}(3), we get a triangle $E \to F \to D \to \Sigma E$ with $E \in \langle G \rangle^{[-A-i,A-i]}$ and $D \in \mathsf{T}^{\leq -i-1} = \mathcal{R}_{i+1}$, which is what we needed as $\langle G \rangle^{[-A-i,A-i]} = \mathcal{G}^{[-A-i,A-i]}$.
    \end{enumerate}
\end{proof}
\section{The closure of the compacts}

We now give some results related to the subcategory $\overline{\mathsf{T}^c}$, see \Cref{Definition closure of compacts}. Whenever there is an orthogonal metric $\mathcal{R}$ given on $\mathsf{T}$, $\overline{\mathsf{T}^c}$ will always mean the closure of the compacts with respect to the metric $\mathcal{R}$ unless explicitly stated otherwise. For a triangulated category with a generating sequence $\mathcal{G}$ (\Cref{Definition of generating sequence}), the closure of the compacts will be with respect to the metric $\mathcal{R}^{\mathcal{G}}$ unless explicitly stated otherwise

Now, we prove that under the assumption of (weak) $\mathcal{G}$-quasiapprox (\Cref{Definition of G-quasiapproximability}), we can approximate an object in the closure of compact by a \say{very nice} Cauchy sequence, see \Cref{Corollary 4 of section 5}. This will be used to prove the representability results in \Cref{Section Representability results}. We begin with a few technical lemmas.

\begin{lemma}\label{Lemma 1 of section 5}
    Let $(\mathsf{T},\mathcal{R},A)$ be weakly $\mathcal{G}$-quasiapprox or weakly $\mathcal{G}$-approx, see \Cref{Convention on approximability}. Then, 
    for any $i \in \mathbb{Z}$, 
    $m > 0$,
    and any $K \in \mathsf{T}^c \cap \mathcal{R}_i$, there exists an object $L$ and a triangle $E \to K\oplus L \to D \to \Sigma E$ with $E \in \mathcal{G}^{[-A-m+1-i,A-i]}$ and $D \in \mathcal{G}^{(-\infty, A-m-i]}$, see \Cref{Notation Subcategories for (pre)-generating sequences}. If $(\mathsf{T},\mathcal{R},A)$ is further $\mathcal{G}$-quasiapprox or $\mathcal{G}$-approx, then we can choose the triangle so that $ E \in \mathcal{G}^{[-A-m+1-i,A-i]}_{mA} $.
\end{lemma}
\begin{proof}
    Let $K \in \mathsf{T}^c \cap \mathcal{R}_i$ for some integer $i$. For every $m>0$, there exists a triangle 
    \begin{equation}\label{Diagram triangle in Lemma 1 of section 5}
        E_m \to K \to D_m \to \Sigma E_m \quad \text{with } E_m \in \overline{\mathcal{G}}^{[-A-m+1-i,A-i]} \text{ and } D_m \in \mathcal{R}_{i+m} \cap \overline{\mathsf{T}^c}
    \end{equation}
    by \Cref{Proposition Approximating Sequence}(1). If $\mathsf{T}$ is $\mathcal{G}-$quasiapprox, we can further get a triangle with $E_m \in \overline{\mathcal{G}}^{[-A-m+1-i,A-i]}_{mA}$.

    As $K\in\mathsf{T}^c$, there exists an integer $B > 0$, such that $\Hom{\mathsf{T}}{K}{\mathcal{R}_B} = 0$ by \Cref{Lemma Compacts orthogonal to the metric}. By increasing $B$ if necessary, we may assume that $B \geq m+2A+i$ so that we have $\Hom{\mathsf{T}}{\Sigma E_m}{\mathcal{R}_B} = 0$, as $E_m \in \overline{\mathcal{G}}^{[-A-m+1-i,A-i]}_{mA}$ and by \Cref{Definition of G-quasiapproximability}(2) we get that $\Hom{\mathsf{T}}{\overline{\mathcal{G}}^{[-A-m+1-i,A-i]}}{\mathcal{R}_{m+2A-1+i}} = 0$. Hence, we get that, $\Hom{\mathsf{T}}{D_m}{\mathcal{R}_B} = 0$ from the triangle \Cref{Diagram triangle in Lemma 1 of section 5}.

    By \Cref{Proposition Approximating Sequence}(1), we can construct a triangle for $D_m \in \mathcal{R}_{m + i} \cap \overline{\mathsf{T}^c}$ as follows,
    \begin{equation}\label{Diagram triangle 2 in Lemma 1 of section 5}
        E' \to D_m \to Q \to \Sigma E' \quad \text{with }E' \in \overline{\mathcal{G}}^{[-A-B+1,A-m-i]} \text{ and } Q \in \mathcal{R}_B \cap \overline{\mathsf{T}^c}\,.
    \end{equation}
    
    As $\Hom{\mathsf{T}}{D_m}{\mathcal{R}_B} = 0$, the map $D_m \to Q$ is zero, and so $D_m$ is a summand of $E'$ by \Cref{Diagram triangle 2 in Lemma 1 of section 5}. Hence $D_m \in \overline{\mathcal{G}}^{[-A-B+1,A-m-i]} \subseteq \overline{\mathcal{G}}^{(-\infty,A-m-i]}$. 

    So, $K \in \overline{\mathcal{G}}^{[-A-m+1-i,A-i]} \star \overline{\mathcal{G}}^{(-\infty,A-m-i]}$. If $\mathsf{T}$ is $\mathcal{G}$-quasiapprox, we further have that $K \in \overline{\mathcal{G}}^{[-A-m+1-i,A-i]}_{mA}\star \overline{\mathcal{G}}^{(-\infty,A-m-i]}$. 
    We define the following full subcategories,
    \begin{itemize}
        \item $\mathcal{X}_1 \colonequals \operatorname{Coprod}(\mathcal{G}[-A-m+1-i,A-i])$
        \item $\mathcal{X}_2 \colonequals \operatorname{Coprod}_{mA}(\mathcal{G}[-A-m+1-i,A-i])$
        \item $\mathcal{Z} \colonequals \operatorname{Coprod}( \mathcal{G}(-\infty,A-m-i])$
    \end{itemize}
    see \Cref{Notation from Neeman's paper}.
    We have that $K \in \operatorname{smd}(\mathcal{X}_i) \star \operatorname{smd}(\mathcal{Z}) \subseteq \operatorname{smd}(\mathcal{X}_i \star \mathcal{Z})$ with $i=1$ or $2$ depending on $\mathsf{T}$ being weakly $\mathcal{G}$-quasiapprox ($i=1$), or $\mathcal{G}$-quasiapprox ($i=2$). So, there exists an object $K' \in \mathcal{X}_i \star \mathcal{Z}$ with maps $K \xrightarrow{f} K' \to K$ composing to identity.
    
    We now define the corresponding small subcategories,
    \begin{itemize}
        \item $\mathcal{A}_1 \colonequals \operatorname{coprod}(\mathcal{G}[-A-m+1-i,A-i])$
        \item $\mathcal{A}_2 \colonequals \operatorname{coprod}_{mA}(\mathcal{G}[-A-m+1-i,A-i])$
        \item $\mathcal{C} \colonequals \operatorname{coprod}( \mathcal{G}(-\infty,A-m-i])$
    \end{itemize}
    By \cite[Lemma 1.8]{Neeman:2021a}, any map from compacts to the big categories $\mathcal{X}_1, \mathcal{X}_2$ and $\mathcal{Z}$ factors through the corresponding small category $\mathcal{A}_1, \mathcal{A}_2$ and $\mathcal{C}$. Let $\mathcal{S} = \mathsf{T}^c$. Observe that $\mathcal{C} \subseteq \mathcal{S} = \mathsf{T}^c$ and so the conditions of \cite[Remark 1.7]{Neeman:2021a} are satisfied. Hence, we can apply \cite[Lemma 1.6]{Neeman:2021a} to the map $K \xrightarrow{f}K'$, as it is a map from a compact object $K$ to $K' \in \mathcal{X}_i\star \mathcal{Z}$, to get that it must factor through $K'' \in \mathcal{A}_i \star \mathcal{C} $. The composite $K \to K'' \to K' \to K$ is thus the identity morphism, which gives us that $K$ is a summand of $K''$. And so, $K \in \operatorname{smd}(\mathcal{A}_i \star\mathcal{C} )$ which is what we needed to show.

\end{proof}

\begin{lemma}\label{Lemma 2 of section 5}
    Let $(\mathsf{T},\mathcal{R},A)$ be  weakly $\mathcal{G}$-quasiapprox or weakly $\mathcal{G}$-approx, see \Cref{Convention on approximability}. Then, there exists an integer $B > 0$, such that for any $i \in \mathbb{Z}$ and $K \in \mathsf{T}^c\cap \mathcal{R}_i$, there exists a triangle $E \to K \to D \to \Sigma E$ with $E \in \mathcal{G}^{[-B-i, B-i]}$ and $D \in \mathcal{R}_{i + 1}$.
    
    If $(\mathsf{T},\mathcal{R},A)$ is further $\mathcal{G}$-quasiapprox or $\mathcal{G}$-approx,  then we can further arrange $B$ so that we have $E \in \mathcal{G}^{[-B-i, B-i]}_{B} $.
\end{lemma}
\begin{proof}
    
    Let $K \in \mathsf{T}^c\cap \mathcal{R}_i$ for some $i \in \mathbb{Z}$. We apply \Cref{Lemma 1 of section 5} to $K$ with $m = nA + 1$ for any integer $n > 1$, to get triangles $E \to K \oplus L \to D \to \Sigma E$ with $E \in \mathcal{G}^{[-(n+1)A-i,A-i]}$ and $D \in \mathcal{G}^{(-\infty, -(n-1)A-1-i]}$ (see \Cref{Notation Subcategories for (pre)-generating sequences}). If the category is $\mathcal{G}$-quasiapprox, we can further assume that $E \in \mathcal{G}^{[-(n+1)A-i,A-i]}_{(nA + 1)A}$. 

     Note that $D \in \mathsf{T}^c \cap\mathcal{R}_{(n-2)A+1+i}$ as $\mathcal{M}^\mathcal{G}_{i+A} \subseteq \mathcal{R}_i$ by \Cref{Remark conditions of G-approximability}(1), see \Cref{Definition of generating sequence}(2) for the definition of $\mathcal{M}^{\mathcal{G}}$. Applying \Cref{Lemma 1 of section 5} to $D$ with $m = 6A$, we can get a triangle, $E' \to D\oplus M \to D' \to \Sigma E'$ with $E' \in \mathcal{G}^{[-(n + 5)A-i, -(n - 3)A-1-i]}$ and $D' \in \mathcal{G}^{(-\infty, -(n + 3)A - 1 - i ]}$.  We can further assume that $E' \in \mathcal{G}^{[-(n + 5)A-i, -(n - 3)A-1-i]}_{6A^2}$ if $\mathsf{T}$ is $\mathcal{G}$-quasiapprox.

     Now, we apply the octahedral axiom to $K\oplus L \oplus M \to D\oplus M \to D'$, where the map $K\oplus L \oplus M \to D\oplus M$ is the sum of the map $K\oplus L \to D$ and the identity map on $M$, to get,
     \[\begin{tikzcd}
        E & {E''} & {E'} \\
        E & {K\oplus L\oplus M} & {D\oplus M} \\
        & {D'} & {D'}
	  \arrow[from=1-1, to=2-1, equals]
	  \arrow[from=1-1, to=1-2]
	  \arrow[from=1-2, to=1-3]
	  \arrow[from=1-2, to=2-2]
	  \arrow[from=1-3, to=2-3]
	  \arrow[from=2-2, to=3-2]
	  \arrow[from=2-3, to=3-3]
	  \arrow[from=3-2, to=3-3, equals]
	  \arrow[from=2-1, to=2-2]
	  \arrow[from=2-2, to=2-3]
    \end{tikzcd}\]
    This gives us that $E'' \in \mathcal{G}^{[-(n+1)A-i,A-i]} \star \mathcal{G}^{[-(n + 5)A-i, -(n - 3)A-1-i]} \subseteq \mathcal{G}^{[-(n + 5)A-i, A-i]}$. If $\mathsf{T}$ is $\mathcal{G}$-quasiapprox, then $E'' \in \mathcal{G}^{[-(n + 5)A-i, A-i]}_{(n + 6)A^2 + A}$. 
    
    $D'\in \mathcal{G}^{(-\infty, -(n + 3)A - 1 - i ]} \subseteq \mathcal{R}_{(n-2)A+1+i}$ and $E' \in \mathcal{G}^{[-(n + 5)A-i, -(n - 3)A-1-i]} \subseteq \mathcal{R}_{(n-4)A+1+i}$, again as $\mathcal{M}^\mathcal{G}_{i+A} \subseteq \mathcal{R}_i$ by \Cref{Definition of G-approximability}(1). And hence by the triangle $E' \to D \oplus M \to D' \to \Sigma E'$, we get that $M \in \mathcal{R}_{(n-4)A+1+i}$ using the fact that $\mathcal{R}_{(n-4)A+1+i}$ is closed under summands as $\mathcal{R}$ is an orthogonal metric.

    Consider the diagram,
    \[\begin{tikzcd}
	   E & {K\oplus L} & D \\
	   {E''} & {K\oplus L\oplus M} & {D'}
	   \arrow[from=1-1, to=1-2]
	   \arrow[from=1-2, to=1-3]
	   \arrow[from=1-2, to=2-2]
	   \arrow[from=2-1, to=2-2]
	   \arrow[from=2-2, to=2-3]
    \end{tikzcd}\]
    where the vertical map is the direct sum of the identity on $L$ and the zero map. The map from the top left to the bottom right in this diagram lies in \[\HomT{\mathcal{G}^{[-(n+1)A-i,A-i]}}{\mathcal{R}_{(n+2)A+1+i}} = 0 \]
    by \Cref{Definition of G-quasiapproximability}(2) as $(n+2)A+1+i -(n+1)A-i = A+1 $.  
    
    So, we can complete to the following $3 \times 3$ diagram, 
    
    \[\begin{tikzcd}
	   E & {K\oplus L} & D \\
	   {E''} & {K\oplus L\oplus M} & {D'} \\
	   {\widetilde{E}} & {K\oplus\Sigma K\oplus M} & {D''}
	   \arrow[from=1-1, to=1-2]
	   \arrow[from=1-2, to=1-3]
	   \arrow[from=1-2, to=2-2]
	   \arrow[from=2-1, to=2-2]
	   \arrow[from=2-2, to=2-3]
	   \arrow[from=1-1, to=2-1]
	   \arrow[from=1-3, to=2-3]
	   \arrow[from=2-1, to=3-1]
	   \arrow[from=3-1, to=3-2]
	   \arrow[from=3-2, to=3-3]
	   \arrow[from=2-3, to=3-3]
	   \arrow[from=2-2, to=3-2]
    \end{tikzcd}\]
    Note that $\widetilde{E} \in \mathcal{G}^{[-(n-5)A-i,A-i+1]}$ and $D'' \in \mathcal{R}_{(n-2)A+i}$. If $\mathsf{T}$ is $\mathcal{G}$-quasiapprox, we have that $\widetilde{E} \in \mathcal{G}^{[-(n-5)A-i,A-i+1]}_{(2n + 6)A^2 + 2A}$. Applying the octahedral axiom to the composable morphisms $\widetilde{E} \to K \oplus \Sigma K \oplus M \to K\oplus \Sigma K$, where the second map is just the projection, we get, 
    \[\begin{tikzcd}
	   {\widetilde{E}} & {K \oplus\Sigma K \oplus M} & {D''} \\
	   {\widetilde{E}} & {K \oplus \Sigma K} & {\widetilde{D}} \\
	   & {\Sigma M} & {\Sigma M}
	   \arrow[from=1-1, to=1-2]
	   \arrow[from=1-2, to=1-3]
	   \arrow[from=1-2, to=2-2]
	   \arrow[from=1-1, to=2-1, equals]
	   \arrow[from=2-1, to=2-2]
	   \arrow[from=2-2, to=2-3]
	   \arrow[from=2-3, to=3-3]
	   \arrow[from=2-2, to=3-2]
	   \arrow[from=1-3, to=2-3]
	   \arrow[from=3-2, to=3-3, equals]
    \end{tikzcd}\]
    which gives us the triangle $\widetilde{E} \to K \oplus \Sigma K \to \widetilde{D} \to \Sigma \widetilde{E}$ with $\widetilde{D} \in \mathcal{R}_{(n-4)A+i}$. 
    
    So the upshot of all of the above computation is that given any positive integer $n$, there exists an integer $B^1_n$ such that for any $i\in \mathbb{Z}$ and $K \in \mathsf{T}^c\cap \mathcal{R}_i$, there exists a triangle $E_n \to K\oplus \Sigma K \to D_n \to \Sigma E_n$ with $E_n \in \mathcal{G}^{[-B^1_n-i, B^1_n-i]}$ and $D_n \in \mathcal{R}_{i+n}$. If $\mathsf{T}$ is $\mathcal{G}$-quasiapprox, we can further arrange that $E_n \in \mathcal{G}^{[-B^1_n-i, B^1_n-i]}_{B^1_n}$. 

    Now, we will inductively show that given any positive integer $n$ and any $j \in \mathbb{Z}$, there exists an integer $B^j_n$ such that for any $i\in \mathbb{Z}$ and $K \in \mathsf{T}^c\cap \mathcal{R}_i$, there exists a triangle $E_n \to K\oplus \Sigma^{2j-1} K \to D_n \to \Sigma E_n$ with $E_n \in \mathcal{G}^{[-B^j_n-i, B^j_n-i]}$ and $D_n \in \mathcal{R}_{i+n}$. If $\mathsf{T}$ is $\mathcal{G}$-quasiapprox, we can arrange so that $E_n \in \mathcal{G}^{[-B^j_n-i, B^j_n-i]}_{B^j_n}$. 
    
    
    We prove the above claim for $j \geq 1$, and $\mathsf{T}$ weakly $\mathcal{G}$-quasiapprox. The other cases are analogous. We have already proven the result for $j=1$, so, we just need to prove the inductive step now. So, assume that we have the result up to some integer $j \geq 1$. Let $i$ be any integer, $K \in \mathsf{T}^c \cap \mathcal{R}_i$ any object, and $n$ any positive integer. By the induction hypothesis, we have a triangle $E \to K \oplus \Sigma^{2j-1} K \to D \to \Sigma E$ with $E \in \mathcal{G}^{[-B^j_{n+3}-i, B^j_{n+3}-i]}$ and $D \in \mathcal{R}_{i+n+3}$. Choose an $l \geq n + 3$ such that $\HomT{\mathcal{G}^{[-B^j_{n+3}-i, B^j_{n+3}-i]}}{\mathcal{R}_{p+i}} = 0$ for all $p \geq l-1$ which exists by \Cref{Definition of G-quasiapproximability}(2).
    
    By the $j=1$ case, we have a triangle $\widetilde{E} \to K \oplus \Sigma K \to \widetilde{D} \to \Sigma \widetilde{E}$ such that $\widetilde{E} \in \mathcal{G}^{[-B^1_l-i, B^1_l-i]}$ and $\widetilde{D} \in \mathcal{R}_{i+l}$.
    
     So, we get the following diagram with the vertical map given by the identity map on $K$,
    \[\begin{tikzcd}
          {E} & {K \oplus \Sigma^{2j-1} K} & {D} \\
	   \Sigma^{-1}\widetilde{E} & {K \oplus \Sigma^{-1} K} & \Sigma^{-1}\widetilde{D}
           \arrow[from=1-2, to=1-3]
	   \arrow[from=1-1, to=1-2]
	   \arrow[from=1-2, to=2-2]
	   \arrow[from=2-1, to=2-2]
	   \arrow[from=2-2, to=2-3]
    \end{tikzcd}\] 
    As the map from the top left to the bottom right vanishes, we get a $3 \times 3$ diagram as follows,
     \[\begin{tikzcd}
	   {E} & {K \oplus \Sigma^{2j-1} K} & {D} \\
	   \Sigma^{-1}\widetilde{E} & {K \oplus \Sigma^{-1} K} & \Sigma^{-1}\widetilde{D} \\
	   {\Sigma^{-1} E_n} & {\Sigma^{-1}K \oplus \Sigma^{2j}K} & {\Sigma^{-1} D_n}
	   \arrow[from=1-2, to=1-3]
	   \arrow[from=1-1, to=1-2]
	   \arrow[from=1-2, to=2-2]
	   \arrow[from=2-1, to=2-2]
	   \arrow[from=2-2, to=2-3]
	   \arrow[from=1-1, to=2-1]
	   \arrow[from=1-3, to=2-3]
	   \arrow[from=2-1, to=3-1]
	   \arrow[from=2-2, to=3-2]
	   \arrow[from=3-1, to=3-2]
	    \arrow[from=3-2, to=3-3]
	   \arrow[from=2-3, to=3-3]
    \end{tikzcd}\]
    The shift of the bottom triangle is what we were after, which proves the claim by induction.
    
     Now we prove the statement of the lemma. By \Cref{Lemma on implication of strong generating sequence}, there exists an integer $r$ such that $\Sigma^{2r}K \in \mathcal{R}_{i + 1}$. We can do this as follows. By \Cref{Lemma on implication of strong generating sequence}, there exists an integer $a$ with $|a| \leq (2 + 2b)A$ such that $\Sigma^a K \in \mathcal{R}_{i+2}$. Therefore $\Sigma^a K, \Sigma^{a+1}K \in \mathcal{R}_{i+1}$. So we can choose $2r$ to be the even integer in the set $\{a,a+1\}$.
     
     Note that $|r| \leq (1+b)A$ independent of $K$ and $i$, where $b$ is as in \Cref{Lemma on implication of strong generating sequence}. From the above claim, we get a triangle $E \to K \oplus \Sigma^{2r-1} K \to D' \to \Sigma E$ with $E \in \mathcal{G}^{[-B^r_1-i, B^r_1-i]}$ and $D' \in \mathcal{R}_{i+1}$. If $\mathsf{T}$ is $\mathcal{G}$-approximable, we can further arrange that $E \in \mathcal{G}^{[-B^r_1-i, B^r_1-i]}_{B^r_1}$. We apply the octahedral axiom to $E \to K \oplus \Sigma^{2r-1} K \to K$ with second map being the obvious projection map, to get,
     \[\begin{tikzcd}
	   E & {K \oplus \Sigma^{2r-1} K} & {D'} \\
	   E & K & D \\
	   & {\Sigma^{2r}K} & {\Sigma^{2r}K}
	   \arrow[from=1-1, to=1-2]
	   \arrow[from=1-1, to=2-1, equals]
	   \arrow[from=1-2, to=2-2]
	   \arrow[from=2-1, to=2-2]
	     \arrow[from=2-2, to=2-3]
	   \arrow[from=1-2, to=1-3]
	   \arrow[from=1-3, to=2-3]
	   \arrow[from=2-3, to=3-3]
	   \arrow[from=2-2, to=3-2]
	   \arrow[from=3-2, to=3-3, equals]
    \end{tikzcd}\]

    $E \to K \to D \to \Sigma E$ is the required triangle. So, we have proven the result with $B = \operatorname{max}\{B_1^r : |r| \leq (1+b)A\}$, with $b$ as in \Cref{Lemma on implication of strong generating sequence}. Note that this $B$ is independent of the integer $i$ and the object $K \in \mathcal{R}_i \cap \mathsf{T}^c$.
    
\end{proof}

\begin{proposition}\label{Proposition 3 of section 5}
    Let $(\mathsf{T},\mathcal{R},A)$ be  weakly $\mathcal{G}$-quasiapprox or weakly $\mathcal{G}$-approx, see \Cref{Convention on approximability}. Choose an integer $B$ as in \Cref{Lemma 2 of section 5}. Then, for any $i \in \mathbb{Z}$ and $F \in \overline{\mathsf{T}^c} \cap \mathcal{R}_i$ (\Cref{Definition closure of compacts}), there exists a triangle $E \to F \to D \to \Sigma E$ with $E \in \mathcal{G}^{[-B-i, B-i]}$ (see \Cref{Notation Subcategories for (pre)-generating sequences}) and $D \in \mathcal{R}_{i+1}$. If $(\mathsf{T},\mathcal{R},A)$ is further assumed to be $\mathcal{G}$-quasiapprox or $\mathcal{G}$-approx, then we can further get the triangle with $E \in \mathcal{G}^{[-B-i, B-i]}_{B}$.
\end{proposition}

\begin{proof}
    Let $i \in \mathbb{Z}$ and $F \in \overline{\mathsf{T}^c}$. By the definition of $\overline{\mathsf{T}^c}$, see \Cref{Definition closure of compacts}, we can get a triangle $K \to F \to D_1 \to \Sigma K$ such that $K \in \mathsf{T}^c$ and $D_1 \in \mathcal{R}_{i + 1}$. The triangle then gives us that $K \in \mathsf{T}^c\cap \mathcal{R}_i $, as $\Sigma^{-1} D_1 \in \mathcal{R}_i$ (\Cref{Good Metric}(2)), and $K \in \Sigma^{-1} D_1 \star F$ (\Cref{Good Metric}(1)). So, we can apply \Cref{Lemma 2 of section 5} to $K$, to get a triangle $E \to K \to D_2 \to \Sigma E$ with $E \in \mathcal{G}^{[-B-i, B-i]}$ and $D_2 \in \mathcal{R}_{i+1}$. If $\mathsf{T}$ is $\mathcal{G}$-quasiapprox, we can further assume that $E \in \mathcal{G}^{[-B-i, B-i]}_B$.

    We then complete the composable morphisms $E \to K \to F$ to the octahedron,
    \[\begin{tikzcd}
	   E & K & {D_2} \\
	   E & F & D \\
	   & {D_1} & {D_1}
	   \arrow[from=1-1, to=1-2]
	   \arrow[from=1-2, to=1-3]
	   \arrow[from=1-2, to=2-2]
	   \arrow[from=1-3, to=2-3]
	   \arrow[from=1-1, to=2-1, equals]
	   \arrow[from=2-1, to=2-2]
	   \arrow[from=2-2, to=2-3]
	   \arrow[from=2-3, to=3-3]
	   \arrow[from=2-2, to=3-2]
	   \arrow[from=3-2, to=3-3, equals]
    \end{tikzcd}\]

    As $D_1$ and $D_2$ both lie in $\mathcal{R}_{i+1}$, so does $D$, and hence $E \to F \to D \to \Sigma E$ is the required triangle and we are done.

\end{proof}

Note that \Cref{Proposition 3 of section 5} shows that $\mathcal{G}$-approx (resp.\ weak $\mathcal{G}$-approx) implies $\mathcal{G}$-quasiapprox (resp.\ weak $\mathcal{G}$-quasiapprox).

\begin{corollary}\label{Corollary 4 of section 5}
    Let $(\mathsf{T},\mathcal{R},A)$ be weakly $\mathcal{G}$-quasiapprox. Then, for any object $F \in \overline{\mathsf{T}^c} \cap \mathcal{R}_i$ (\Cref{Definition closure of compacts}), there exists a sequence of compact objects $E_1 \to E_2 \to E_3 \to E_4 \to \cdots $ mapping to $F$ such that $E_n \in \mathcal{G}^{[-B-n-i+1,B-i]}$ and $D_n \colonequals \operatorname{Cone}(E_n \to F) \in \mathcal{R}_{i+n}$ for all $n \geq 1$, see \Cref{Notation Subcategories for (pre)-generating sequences}. \\
    If $(\mathsf{T},\mathcal{R},A)$ is further assumed to be $\mathcal{G}$-quasiapprox, then we can further get a sequence such that $E_n \in \mathcal{G}^{[-B-n-i,B-i]}_{nB}$ for all $n \geq 1$.\\
    In either case, the non-canonical map $\hocolim{E}_n \to F$ is an isomorphism.
\end{corollary}
\begin{proof}
    We prove this statement inductively on $n$. We get the case $n=1$ by \Cref{Proposition 3 of section 5}. Now, suppose we have the required sequence up to some integer $n\geq 1$. In particular, we have a triangle $E_n \to F \to D_n \to \Sigma E_n$ with $D_n \in \mathcal{R}_{i+n}$ and $E_n$ in $\mathcal{G}^{[-B-n-i+1,B-i]}$ or $\mathcal{G}^{[-B-n-i+1,B-i]}_{nB}$ depending on if $\mathsf{T}$ is weakly $\mathcal{G}$-quasiapprox or $\mathcal{G}$-quasiapprox respectively.
    
    We apply \Cref{Proposition 3 of section 5} to $D_n \in \mathcal{R}_{i+n}$ to get the triangle $\widetilde{E} \to D_n \to D_{n+1} \to \Sigma \widetilde{E}$ with $D_{n+1} \in \mathcal{R}_{i + n + 1}$ and $E_n$ in $\mathcal{G}^{[-B-n-i,B-i-n]}$ or $\mathcal{G}^{[-B-n-i,B-i-n]}_{B}$ depending on if $\mathsf{T}$ is weakly $\mathcal{G}$-quasiapprox or $\mathcal{G}$-quasiapprox respectively. Applying the octahedral axiom to $F \to D_n \to D_{n+1}$, we get,
    \[\begin{tikzcd}
	   E_n & E_{n+1} & \widetilde{E} \\
	   E_n & F & D_{n} \\
	   & D_{n+1} & D_{n+1}
	   \arrow[from=2-2, to=3-2]
	   \arrow[equals, from=3-2, to=3-3]
	   \arrow[from=2-2, to=2-3]
	   \arrow[from=2-3, to=3-3]
	   \arrow[from=1-1, to=2-1]
	   \arrow[from=1-1, to=1-2]
	   \arrow[from=1-2, to=2-2]
	   \arrow[from=2-1, to=2-2]
	   \arrow[from=1-2, to=1-3]
	   \arrow[from=1-3, to=2-3]
    \end{tikzcd}\]
    And so, we get the required extension of the sequence as,
    \[E_{n+1} \in \mathcal{G}^{[-B-n-i+1,B-i]} \star \mathcal{G}^{[-B-n-i,B-i-n]} \subseteq \mathcal{G}^{[-B-n-i, B-i]}\]
    if $\mathsf{T}$ is weakly $\mathcal{G}$-quasiapprox. Further, if $\mathsf{T}$ is $\mathcal{G}$-quasiapprox,
    \[E_{n+1} \in \mathcal{G}^{[-B-n-i+1,B-i]}_{nB} \star \mathcal{G}^{[-B-n-i,B-i-n]}_B \subseteq \mathcal{G}^{[-B-n-i, B-i]}_{(n+1)B}\]

    Finally, $\hocolim E_i \cong F$ by \Cref{Lemma Intersection of the metric is zero}(2).
\end{proof}
     
\begin{remark}
    Note that in \Cref{Proposition 3 of section 5} and \Cref{Corollary 4 of section 5}, we can choose $\mathcal{R}$ to be any orthogonal metric in the preferred equivalence class of orthogonal metrics (\Cref{Definition preferred equivalence class of orthogonal metrics}) by \Cref{Proposition G-approximability from orthogonal metric in preferred equivalence class}.
\end{remark}

\section{Representability results}\label{Section Representability results}
In this section, we will prove the main representability results. We start with the following definition inspired by \cite[Definition 7.3]{Neeman:2021b}.
\begin{definition}\label{Definition strong G-approximating system}
    Let $\mathsf{T}$ be a triangulated category with a generating sequence $\mathcal{G}$ (\Cref{Definition pre-generating sequence}), and an orthogonal metric $\mathcal{R}$ (\Cref{Good Metric}) in the preferred $\mathbb{N}$-equivalence of orthogonal metrics, see \Cref{Definition preferred equivalence class of orthogonal metrics}. Then, for $n$ a positive integer or $\infty$, we define a \emph{strong $\mathcal{G}_n$-approximating system} to be a sequence $E_0 \to E_1 \to E_2 \to E_3 \to \cdots $ such that each $E_i \in \mathcal{G}_n = \mathcal{G}^{(-\infty, \infty)}_n$ (\Cref{Notation Subcategories for (pre)-generating sequences}) and for all $i \geq 0$, $\operatorname{Cone}(E_i \to E_{i + 1} ) \in \mathcal{R}_{i + 1}$. Note that $\mathcal{G}_{\infty} =  \mathsf{T}^c$.
    
    Let $F \in \mathsf{T}$ and suppose we are given a map from a sequence $E_*=\{E_i \to E_{i+1}\}_{i \geq 0}$ in $\mathcal{G}_n$ to $F$. Then, the sequence is a \emph{strong $\mathcal{G}_n$-approximating system for $F$ } if, 
    \begin{enumerate}
        \item The natural map $\colim \Hom{\mathsf{T}}{-}{E_i} \to \Hom{\mathsf{T}}{-}{F}$ is an isomorphism on $\mathsf{T}^c$.  
        \item $\operatorname{Cone}(E_n \to F) \in \mathcal{R}_n$ for all $n\geq 0$.
    \end{enumerate}
    We also call a strong $\mathcal{G}_{\infty}$-approximating system a strong $\mathsf{T}^c$-approximating system as we have $\mathcal{G}_{\infty} = \mathsf{T}^c$.
\end{definition}

\begin{lemma}\label{Lemma on strong approximating sequences}
    Let $\mathsf{T}$ a triangulated category with a generating sequence $\mathcal{G}$ (\Cref{Definition of generating sequence}), and an orthogonal metric $\mathcal{R}$ (\Cref{Good Metric}) on $\mathsf{T}$ such that for all $n$, $\HomT{\mathcal{G}^n}{\mathcal{R}_i} = 0$ for all $i >> 0$. We consider the closure of the compacts $\overline{\mathsf{T}^c}$ (\Cref{Definition closure of compacts}) with respect to the metric $\mathcal{R}$. Then, for $n \in [1,\infty]$, that is, for $n$ either a positive integer or $\infty$,
    \begin{enumerate}
        \item Any strong $\mathcal{G}_n$-approximating system (\Cref{Definition strong G-approximating system}) is a strong $\mathcal{G}_n$-approximating system for $F = \hocolim{E_i}$, and further, $F \in \overline{\mathsf{T}^c}$.
        \item Given  $F \in$ $\mathsf{T}^c$ and a strong $\mathcal{G}_n$-approximating system $E_*$ for $F$ (\Cref{Definition strong G-approximating system}), the non-canonical map $\hocolim E_i \to F$ is an isomorphism.
        \item Any object $F \in \mathsf{T}^c$ has a strong $\mathsf{T}^c$-approximating system, see \Cref{Definition strong G-approximating system}.
       
    \end{enumerate}
\end{lemma}
\begin{proof}
    First we prove (1). Let $E_*$ be a strong $\mathcal{G}_n$-approximating system, and let $F = \hocolim{E_i}$. For each $i,j \geq 1$, we have triangles $E_j \to E_{j + i } \to C_{i,j} \to \Sigma E_j$. We will prove that $C_{i,j} \in \mathcal{R}_{j + 1}$ for all $i,j \geq 1$ by induction on $i$. The base case $i=1$ follows from the definition, see \Cref{Definition strong G-approximating system}. Suppose we know the result up to some positive integer $i$. In particular, we have triangles $E_j \to E_{j + i} \to C_{i,j} \to \Sigma E_j$ with $C_{i,j} \in \mathcal{R}_{j + 1}$ for all $j \geq 1$. Again, as $E_*$ is a strong $\mathcal{G}_n$-approximating system, we get triangles $E_{j+i} \to E_{j + i +1} \to D_j \to \Sigma E_{j+i}$ with $D_j \in \mathcal{R}_{j + i + 1}$ for all $j \geq 1$. Applying the octahedral axiom to $E_j \to E_{j + i} \to E_{j + i +1}$, we get,
    \[\begin{tikzcd}
	   E_j & E_{j+i} & C_{i,j} \\
	   E_j & E_{j+i+1} & C_{(i+1,j)} \\
	   & D_j & D_j
	   \arrow[from=1-1, to=1-2]
	   \arrow[equals, from=1-1, to=2-1]
	   \arrow[from=1-2, to=1-3]
	   \arrow[from=1-2, to=2-2]
	   \arrow[from=1-3, to=2-3]
	   \arrow[from=2-1, to=2-2]
	   \arrow[from=2-2, to=2-3]
	   \arrow[from=2-2, to=3-2]
	   \arrow[from=2-3, to=3-3]
	   \arrow[equals, from=3-2, to=3-3]
    \end{tikzcd}\]
    The triangle forming the middle row gives us the induction step as, from the triangle forming the rightmost column we have that,
    \[C_{(i+1,j)} \in C_{i,j} \star D_j \subseteq \mathcal{R}_{j+1} \star \mathcal{R}_{j+i+1}\subseteq \mathcal{R}_{j+1} \star \mathcal{R}_{j+1} \subseteq \mathcal{R}_{j+1}\]
    for all $j \geq 1$, see \Cref{Good Metric}.
    
     Consider the following 3$\times$3 diagram which we construct from the top left commutative square in which the vertical maps are the direct sums of the obvious map $E_j \to E_{j+i}$ for all $i \geq 0$, 
    \[\begin{tikzcd}
	{\bigoplus_{i \geq 0}E_j} & {} & {\bigoplus_{i \geq 0}E_j} & {} & {E_j} \\
	{\bigoplus_{i \geq 0 }E_{j + i}} && {\bigoplus_{i \geq 0 }E_{j + i}} && F \\
	{\bigoplus_{i \geq 0 }C_{i,j}} && {\bigoplus_{i \geq 0 }C_{i,j}} && C_j
	\arrow["{1 - \text{shift}}", from=1-1, to=1-3]
	\arrow[from=1-3, to=1-5]
	\arrow[from=1-1, to=2-1]
	\arrow[from=1-3, to=2-3]
	\arrow[from=1-5, to=2-5]
	\arrow[from=2-1, to=3-1]
	\arrow["{1 - \text{shift}}", from=2-1, to=2-3]
	\arrow[from=2-3, to=3-3]
	\arrow[from=2-5, to=3-5]
	\arrow[from=2-3, to=2-5]
	\arrow[dashed, from=3-1, to=3-3]
	\arrow[dashed, from=3-3, to=3-5]
\end{tikzcd}\]

As $C_{i,j} \in \mathcal{R}_{j + 1}$ for all $i \geq 0$ and $\mathcal{R}_{j+1}$ is closed under coproducts as it is an orthogonal metric, we get that $C_j \in \mathcal{R}_j $ (see \Cref{Good Metric}). By the triangle forming the rightmost column of this diagram, $C_j = \operatorname{Cone}(E_j \to F) \in \mathcal{R}_j$ for all $j \geq 1$. Further, for any compact object $E$, $\colim\HomT{E}{E_i} \cong \HomT{E}{F}$ by \cite[Lemma 2.8]{Neeman:1996}. And so, $E_*$ is a strong $\mathcal{G}_n$-approximating system for $F$. The rightmost column also shows that $F \in \mathsf{T}^c$.

For the proof of (2), let $E = \hocolim{E_i}$. Complete the natural map $E \to F$ to a triangle $E \to F \to D \to \Sigma E$. We apply the octahedral axiom to $E_i \to E \to F $ for any $i \geq 0$, to get,

\[\begin{tikzcd}
	{E_i} & E & {C_i} \\
	{E_i} & F & {C'_i} \\
	& D & D
	\arrow[from=1-1, to=1-2]
	\arrow[from=1-2, to=1-3]
	\arrow[from=1-1, to=2-1, equals]
	\arrow[from=1-2, to=2-2]
	\arrow[from=2-1, to=2-2]
	\arrow[from=2-2, to=2-3]
	\arrow[from=1-3, to=2-3]
	\arrow[from=2-3, to=3-3]
	\arrow[from=2-2, to=3-2]
	\arrow[from=3-2, to=3-3, equals]
\end{tikzcd}\]
By the proof of part (1), we know that $C_i \in \mathcal{R}_i$, and hence $\Sigma C_i \in \mathcal{R}_{i-1}$, see \Cref{Good Metric}(2). As $E_*$ is a strong $\mathcal{G}_n$-approximating system for $F$, we get that $C'_i \in \mathcal{R}_i \subseteq \mathcal{R}_{i-1}$. Consider the triangle $C_i' \to D \to \Sigma C_i \to \Sigma C_i'$, we get that $D \in C_i' \star  \Sigma C_i \subseteq \mathcal{R}_{i - 1} \star \mathcal{R}_{i-1} \subseteq \mathcal{R}_{i-1}$ by \Cref{Good Metric}(1).

As this is true for any $i > 0$, $D \in \bigcap_{n \geq 0 }\mathcal{R}_n = 0 $ by \Cref{Lemma Intersection of the metric is zero}. So, $E = \hocolim{E_i} \to F $ is an isomorphism, which proves (1).

Finally, for (3), we proceed inductively. By the definition of $\overline{\mathsf{T}^c}$, see \Cref{Definition closure of compacts}, there exists a triangle $E_1 \to F \to D_1 \to \Sigma E_1$ such that $E_1 \in \mathsf{T}^c$ and $D_1 \in \mathcal{R}_1$. Now suppose we have constructed the sequence up to a positive integer $i$. We
choose an integer $N > 0$ such that $\Hom{\mathsf{T}}{E_i}{\mathcal{R}_n} = 0 $ for all $n \geq N$ which exists by \Cref{Lemma Compacts orthogonal to the metric}. Finally, again by the definition of the closure of the compacts (\Cref{Definition closure of compacts}), we choose a triangle $E_{i + 1} \to F \to D_{i + 1} \to \Sigma E_{i+1}$ such that $E_{i + 1} \in \mathsf{T}^c$ and $D_{i + 1} \in \mathcal{R}_{N + i + 1}\subseteq \mathcal{R}_{i + 1}$. By the choice of $N$, we get that the composite $(E_i \to F \to D_{i + 1} ) = 0 $, and hence $E_i \to F$ factors through $E_{i + 1}$, giving us the required sequence by induction, proving (3).

\end{proof}

Other than \Cref{Section technical tools}, throughout the rest of this section, we will be working in the following setup.
\begin{setup}\label{Setup}
    Let $\mathsf{T}$ be $R$-linear triangulated category, with $R$ a commutative noetherian ring. Further, assume $\mathsf{T}$ has a finite generating sequence $\mathcal{G}$ (\Cref{Definition of generating sequence}), an orthogonal metric $\mathcal{R}$ (\Cref{Good Metric}) in the preferred $\mathbb{N}$-equivalence class of orthogonal metrics on $\mathsf{T}$ (\Cref{Definition preferred equivalence class of orthogonal metrics}), and a positive integer $A>0$ such that,
    \begin{enumerate}
        \item $\HomT{G_1}{G_2}$ is a finitely generated $R$-module for all $G_1, G_2 \in \bigcup_{n \in \mathbb{Z}} \mathcal{G}^n$.
        \item $\HomT{\mathcal{G}^n}{\mathcal{R}_i} = 0$ for all $n,i \in \mathbb{Z}$ such that $n+i \geq A$. 
    \end{enumerate}
    As $\mathcal{G}$ is a finite generating sequence, $\mathcal{G}^i = \{G^i_j : j \in J_i\}$ for finite sets $J_i$ for all $i \in \mathbb{Z}$.
    In what follows, we will consider the closure of the compacts $\mathsf{T}^c$ (\Cref{Definition closure of compacts}) corresponding to the metric $\mathcal{R}$.
\end{setup}
\begin{remark}\label{Remark Setup condition 2 implication}
    Suppose we are in \Cref{Setup}.
    Recall that the orthogonal metric $\mathcal{R}^{\mathcal{G}}$ (see \Cref{Definition of generating sequence}) also lies in the preferred $\mathbb{N}$-equivalence class of orthogonal metrics. Hence, there exists an integer $l$ such that $\mathcal{R}^{\mathcal{G}}_{i+l} \subseteq \mathcal{R}_i$ for all $ i\in \mathbb{Z}$, see \Cref{Definition equivalence relation on extended good metrics}(2). But, by the definition of $\mathcal{R}^{\mathcal{G}}$, we have that $\mathcal{G}^{-i} \subseteq \mathcal{M}^{\mathcal{G}}_i \subseteq \mathcal{R}^{\mathcal{G}}_i \subseteq \mathcal{R}_{i-l}$, see \Cref{Definition of generating sequence}. And so, $\HomT{\mathcal{G}^n}{\mathcal{G}^{-i}}=0$ for all $n+i \geq A + l$.
\end{remark}
We remind the reader of the following standard definition.
\begin{definition}\label{Definition cohomological functors}
    Let $\mathsf{S}$ be a full and replete subcategory of a $R$-linear triangulated category $\mathsf{T}$ with $\Sigma \mathsf{S}=\mathsf{S}$. A functor $H : \mathsf{S}^{\text{op}} \to \operatorname{Mod}(R)$ is called a $\mathsf{S}$-cohomological functor if it takes triangles to long exact sequences of $R$-modules. Equivalently, we get an exact sequence $H(Z) \to H(Y) \to H(X)$ in $\operatorname{Mod}(R)$ for any triangle $X \to Y \to Z \to \Sigma X$ in $\mathsf{T}$ with $X,Y$ and $Z$ in $\mathsf{S}$.
\end{definition}

We will be mostly interested in $\mathcal{G}_n$-cohomological functors. So, we define a few special classes of $\mathcal{G}_n$-cohomological functors now.

\begin{definition}\label{Definition finite bounded cohomological functors}
    Suppose we are in \Cref{Setup}, and let $n \in [1,\infty]$. Then, we say a $\mathcal{G}_n$-cohomological functor $H$ is,
    \begin{enumerate}
        \item \emph{$\mathcal{G}$-semifinite} if $H(F)$ is a finitely generated $R$-module for all $F \in \mathcal{G}_n$, and $H(\mathcal{G}^i) = 0$ for all $i >>0 $.
        \item \emph{$\mathcal{G}$-finite} if $H(F)$ is a finitely generated $R$-module for all $F \in \mathcal{G}_n$, and $H(\mathcal{G}^i) = 0$ for all $|i| >>0 $.
    \end{enumerate}
    see \Cref{Notation Subcategories for (pre)-generating sequences} for the definition of $\mathcal{G}_n$.
\end{definition}
The following example shows what the $\mathcal{G}$-finite and $\mathcal{G}$-semifinite $\mathsf{T}^c$-cohomological functors look like for some specific cases if interest.
\begin{example}\label{Example finite bounded cohomological functors}
    Suppose we are in \Cref{Setup}. Further assume that $\mathsf{T}$ has a single compact generator $G$. Recall that we had a finite generating sequence (\Cref{Definition of generating sequence}) $\mathcal{G}$ on $\mathsf{T}$. Then,
    \begin{enumerate}
        \item If $\mathcal{G}^n = \{\Sigma^{-n}G\}$ for all $n \in \mathbb{Z}$, then a $\mathsf{T}^c$-cohomological functor $H$ is $\mathcal{G}$-semifinite (resp. $\mathcal{G}$-finite) if for all objects $F \in \mathsf{T}^c$, $H(F)$ is a finitely generated $R$-module and $H(\Sigma^{-i} F) = 0$ for all $i >> 0$ (resp. $|i| >> 0|$). Note that this is the generating sequence used for approximability, see \Cref{Theorem Approximability and G-approximability}.
        \item If $\mathcal{G}^n = \{\Sigma^{n}G\}$ for all $n \in \mathbb{Z}$, then a $\mathsf{T}^c$-cohomological functor $H$ is a $\mathcal{G}$-semifinite (resp. $\mathcal{G}$-finite) if for all objects $F \in \mathsf{T}^c$, $H(F)$ is a finitely generated $R$-module and $H(\Sigma^i F) = 0$ for all $i >> 0$ (resp. $|i| >> 0|$).
    \end{enumerate}
    These follow from the \Cref{Definition finite bounded cohomological functors} and the observation that for any $F \in \mathsf{T}^c$, there exists $n \geq 0$ such that $F \in \langle G \rangle^{[-n,n]}$ as $G$ is a compact generator of $\mathsf{T}$, see \Cref{Notation from Neeman's paper}.
\end{example}
\begin{notation}
    For any $F \in \mathsf{T}$, we have a $\mathsf{T}^c$-cohomological functor (\Cref{Definition cohomological functors})  $\mathcal{Y}(F) : [\mathsf{T}^c]^{\operatorname{op}} \to \operatorname{Mod}(R)$  which is the restriction of the Yoneda functor. That is, for any $K \in \mathsf{T}^c$, we have $\mathcal{Y}(F)(K) \colonequals \Hom{\mathsf{T}}{K}{F}$.
\end{notation}

\begin{lemma}\label{Lemma Closure of compacts are finite right-bounded}
   Suppose we are in \Cref{Setup}. Then, for any $F$ in the closure of the compacts $\mathsf{T}^c$ (\Cref{Definition closure of compacts}), we have that $\mathcal{Y}(F) : [\mathsf{T}^c]^{\operatorname{op}} \to \operatorname{Mod}(R)$ is a $\mathcal{G}$-semifinite $\mathsf{T}^c$-cohomological functor, see \Cref{Definition finite bounded cohomological functors}(2).
\end{lemma}
\begin{proof}
    Let, 
    \[\mathsf{L} \colonequals \left\{H \in \mathsf{T}^c : \Hom{\mathsf{T}}{G}{H} \text{ is a finitely generated $R$-module for all } G \in \bigcup_{n \in \mathbb{Z}}\mathcal{G}^n \right\} \]
    Then, $\bigcup_{n \in \mathbb{Z}}\mathcal{G}^n \subseteq \mathsf{L}$ by \Cref{Setup}(1). As $\mathsf{L}$ is clearly closed under extensions and direct summands, and contains $\bigcup_{n \in \mathbb{Z}}\mathcal{G}^n$, we get that $\operatorname{smd}(\operatorname{coprod}(\bigcup_{n \in \mathbb{Z}}\mathcal{G}^n)) = \mathsf{T}^c \subseteq \mathsf{L}$, see \Cref{Definition of generating sequence}(1). And so, $\mathsf{L} = \mathsf{T}^c$.
    Now, for any $L \in \mathsf{T}^c$, we define, 
    \[\mathsf{K}(L) \colonequals \{ K \in \mathsf{T}^c : \Hom{\mathsf{T}}{K}{L} \text{ is a finitely generated $R$-module} \} \]
    $\mathsf{K}(L)$ is clearly closed under extensions and direct summands, and contains $\bigcup_{n \in \mathbb{Z}}\mathcal{G}^n$ by the paragraph above, and hence, $\mathsf{K}(L) = \mathsf{T}^c$. And so, $\HomT{K}{L}$ is a finitely generated module for all $K, L \in \mathsf{T}^c$.

    Now, let $F \in \mathsf{T}^c $. Then, by the definition of $\mathsf{T}^c$ (\Cref{Definition closure of compacts}), there exists a triangle $E_1 \to F \to D_1 \to \Sigma E_1$ such that, $E_1 \in \mathsf{T}^c$ and $D_1 \in \mathcal{R}_1$. There exists $n>0$ such that $E_1 \in \mathcal{G}^{[-n,n]}$ (see \Cref{Notation Subcategories for (pre)-generating sequences}) as $\mathcal{G}$ is a generating sequence, see \Cref{Definition of generating sequence}(1). By \Cref{Remark Setup condition 2 implication}, $\HomT{\mathcal{G}^i}{ \mathcal{G}^{[-n,n]}} = 0$ for all $i >> 0$, which implies that $\HomT{\mathcal{G}^i}{ E_1 } = 0$ for all $i >> 0$. Also, by \Cref{Setup}(2), $\HomT{\mathcal{G}^i}{ \mathcal{R}_1 } = 0$ for all $i >> 0$. And so, $\HomT{\mathcal{G}^i}{ F } = 0$ for all $i >> 0$. So, all that remains to show is that $\HomT{K}{F}$ is a finitely generated $R$-module for all $K \in \mathsf{T}^c$.

    Let $K\in \mathsf{T}^c$. Then, by \Cref{Lemma Compacts orthogonal to the metric}, there exists an integer $j$ such that $\HomT{K}{\mathcal{R}_i} = 0$ for all $i \geq j$. As $F \in \mathsf{T}^c$, there exists a triangle $L \to F \to D \to \Sigma L$ with $L \in \mathsf{T}^c$ and $D \in \mathcal{R}_{j + 1}$. Consider the long exact sequence we get from applying the functor $\HomT{K}{-}$ to this triangle,
    \[0 = \HomT{K}{\Sigma^{-1}D} \to \HomT{K}{L} \to \HomT{K}{F} \to \HomT{K}{D} = 0\]
    
    This gives us that $\HomT{K}{F} \cong \HomT{K}{L}$, which is a finitely generated $R$-module, as we have shown above that $\HomT{K}{L}$ is a finitely generated $R$-module for any pair of compact objects $K$ and $L$. Combined with the previous paragraph, we get that $\mathcal{Y}(F)$ is a $\mathcal{G}$-semifinite $\mathsf{T}^c$-cohomological functor
\end{proof}

\subsection{Some important technical tools}\label{Section technical tools}

We state a few important definitions and results from \cite{Neeman:2021b}, some of them suitably modified to this context, which we will be using later in this section. Note that we will not be in \Cref{Setup} in general in this subsection. 

\begin{definition}\label{Definition of an Approximating Sequence}
    \cite[Definition 5.1]{Neeman:2021b} Let $\mathsf{T}$ be a $R$-linear triangulated category, where $R$ is a commutative ring. Let $\mathsf{A}$ and $\mathsf{B}$ be full subcategories of $\mathsf{T}$, and $H$ be a $\mathsf{B}$-cohomological functor (\Cref{Definition cohomological functors}). An \emph{$\mathsf{A}$-approximating system for $H$} is a sequence $E_1 \to E_2 \to E_3 \to E_4 \to \cdots $ in $\mathsf{T}$ with a cofinal subsequence in $\mathsf{A} \cap \mathsf{B}$ such that there exists an isomorphism of functors $\colim \HomT{-}{E_i}|_{\mathsf{B}} \cong H(-)$. 
\end{definition}
\begin{remark}
    Note that for any $n \in [1, \infty]$, any strong $\mathcal{G}_n$-approximating system for an object $F$ (\Cref{Definition strong G-approximating system}) gives a $\mathcal{G}_n$-approximating system of $\mathcal{Y}(F)$, as $F \cong \hocolim E_i$ by \Cref{Lemma on strong approximating sequences}(2), which implies that $\HomT{K}{F} = \colim \HomT{K}{E_i}$ for any $K \in \mathsf{T}^c$ by \cite[Lemma 2.8]{Neeman:1996}.
\end{remark}

\begin{definition}\label{Definition weak triangle}
    \cite[Definition 6.2]{Neeman:2021b} Let $\mathsf{T}$ be a triangulated category with a triangulated subcategory $\mathsf{S}$. For any $\mathsf{S}$-cohomological functor (\Cref{Definition cohomological functors}) $H$, we define another $\mathsf{S}$-cohomological functor $\Sigma H$ by $\Sigma H(s) \colonequals H(\Sigma^{-1}s)$ for all $s \in \mathsf{S}$.

    A \emph{weak triangle} in the category of $\mathsf{S}$-cohomological functors $\Hom{}{\mathsf{S}^{\operatorname{op}}}{\operatorname{Mod}(R)}$ is a sequence $A \xrightarrow{u} B \xrightarrow{v} C \xrightarrow{w} \Sigma A$ such that the following is true for any rotation of the sequence $A \xrightarrow{u} B \xrightarrow{v} C \xrightarrow{w} \Sigma A$: For any triangle $x \xrightarrow{u'} y \xrightarrow{v'} z \xrightarrow{w'} \Sigma x $ in $\mathsf{T}$, and a commutative diagram,
    \[\begin{tikzcd}
	   \mathcal{Y}(x) & \mathcal{Y}(y) \\
	   A & B
	   \arrow["\mathcal{Y}(u')",from=1-1, to=1-2]
	   \arrow[from=1-1, to=2-1]
	   \arrow[from=1-2, to=2-2]
	   \arrow["u",from=2-1, to=2-2]
    \end{tikzcd}\]
    we can extend to a commutative diagram as follows,
    \[\begin{tikzcd}
	   \mathcal{Y}(x) & \mathcal{Y}(y) & \mathcal{Y}(z) & \mathcal{Y}(\Sigma x) \\
	   A & B & C & \Sigma A
	   \arrow["\mathcal{Y}(u')",from=1-1, to=1-2]
	   \arrow[from=1-1, to=2-1]
	   \arrow["\mathcal{Y}(v')",from=1-2, to=1-3]
	   \arrow[from=1-2, to=2-2]
          \arrow["\mathcal{Y}(w')",from=1-3, to=1-4]
          \arrow[from=1-4, to=2-4]
	   \arrow[from=1-3, to=2-3]
	   \arrow["u",from=2-1, to=2-2]
	   \arrow["v",from=2-2, to=2-3]
          \arrow["w",from=2-3, to=2-4]
    \end{tikzcd}\]
    Note that a weak triangle is exact by \cite[Lemma 6.4]{Neeman:2021b}.
    A sequence $x \to y \to z \to \Sigma x$ in $\mathsf{T}$ is called a weak triangle if $\mathcal{Y}(-)$ takes it to a weak triangle in $\Hom{R}{\mathsf{S}^{\operatorname{op}}}{\operatorname{Mod}(R)}$.
    
\end{definition}

\begin{remark}\label{Remark regarding Neeman Lemma 6.5}
     We will be using \cite[Lemma 6.5]{Neeman:2021b} later in this section. In this remark, we adapt it to \Cref{Setup}. So, assume that we are in \Cref{Setup}. Further, suppose we are given
     \begin{enumerate}[label=(\roman*)]
         \item a morphism $\hat{\alpha} : \hat{A} \to \hat{B}$ in $\mathsf{T}^c$ (\Cref{Definition closure of compacts}).
         \item and $n, n' \in [0, \infty]$ , as well as a strong $\mathcal{G}_{n'}$-approximating system $\mathfrak{U}_*$ for $\hat{A}$ and a strong $\mathcal{G}_n$-approximating system $\mathfrak{B}_*$ for $\hat{B}$, see \Cref{Definition strong G-approximating system}. 
     \end{enumerate}
     
    Then up to passing to a subsequence of $\mathfrak{B}_*$, by \cite[Lemma 5.5]{Neeman:2021b}, we have a map of sequences $\alpha_* : \mathfrak{U}_* \to \mathfrak{B}_*$ compatible with $\hat{\alpha}$. 
     As in \cite[Lemma 6.5]{Neeman:2021b}, we extend the map of sequences to triangles, to get, for each $n > 0$,

     \[\begin{tikzcd}
	{\mathfrak{U}_m} &  {\mathfrak{B}_m} & {\mathfrak{C}_m} & {\Sigma\mathfrak{U}_m} \\
	{\mathfrak{U}_{m+1}} & {\mathfrak{B}_{m + 1}} & {\mathfrak{C}_{m + 1}} & {\Sigma\mathfrak{U}_{m + 1}} \\
	{\mathfrak{X}_m} & {\mathfrak{Y}_{m}} & {\mathfrak{Z}_{m}} & {\Sigma\mathfrak{X}_m}
	\arrow[from=1-1, to=1-2]
	\arrow[from=1-2, to=1-3]
    \arrow[from=1-3, to=1-4]
	\arrow[from=1-3, to=2-3]
    \arrow[from=1-4, to=2-4]
	\arrow[from=1-2, to=2-2]
	\arrow[from=1-1, to=2-1]
	\arrow[from=2-1, to=3-1]
	\arrow[from=2-2, to=3-2]
	\arrow[from=2-3, to=3-3]
    \arrow[from=2-4, to=3-4]
	\arrow[from=3-1, to=3-2]
	\arrow[from=3-2, to=3-3]
    \arrow[from=3-3, to=3-4]
	\arrow[from=2-1, to=2-2]
	\arrow[from=2-2, to=2-3]
    \arrow[from=2-3, to=2-4]
    \end{tikzcd}\]

    By the definition of a strong approximating system, see \Cref{Definition strong G-approximating system}, we have that $\mathfrak{Y}_{m},\mathfrak{X}_m \in \mathcal{R}_{m+1} \subseteq \mathcal{R}_m$. So, $\Sigma\mathfrak{X}_m \in \mathcal{R}_m$ by \Cref{Good Metric}(2). Hence, by \Cref{Good Metric}(1) we get that $\mathfrak{Z}_m \in \mathfrak{Y}_{m} \star \Sigma\mathfrak{X}_m \subseteq \mathcal{R}_m \star \mathcal{R}_m = \mathcal{R}_m$. And so, if we define $\mathfrak{C'}$ by $\mathfrak{C}'_m \colonequals \mathfrak{C}_{m + 1}$, then, it is a strong $\mathcal{G}_{n + n'}$-approximating system. By \cite[Remark 6.3]{Neeman:2021b} and \cite[Lemma 6.5]{Neeman:2021b}, the homotopy colimit of the sequence of triangles $\mathfrak{U}_* \to \mathfrak{B}_* \to \mathfrak{C}_* \to \Sigma \mathfrak{U}_*$ gives us a weak triangle (\Cref{Definition weak triangle}) $\hat{A} \to \hat{B} \to \hat{C} = \hocolim \mathfrak{C}_i \to \Sigma \hat{A}$ in $\mathsf{T}^c$. 
    
\end{remark}     

\subsection{The results}
We prove the main representability theorems in this subsection. The main results are \Cref{Main Theorem on Tc-} and \Cref{Main Theorem on Tbc}.
\begin{lemma}\label{Lemma 1 of Section 6.2}
    Suppose we are in \Cref{Setup}. Let $H$ be a $\mathcal{G}$-semifinite $\mathcal{G}_n$-cohomological functor (\Cref{Definition finite bounded cohomological functors}(2)).
    Then, we can produce a sequence $F_1 \to \cdots \to F_n$ in $\mathsf{T}^c$ with compatible maps, $\phi_i : \mathcal{Y}(F_i)|_{\mathcal{G}_n} \to H$ such that,
    \begin{enumerate}
        \item For all $1 \leq i \leq n $, $F_i$ has a strong $\mathcal{G}_i$-approximating system (\Cref{Definition strong G-approximating system}), and the restricted map $\phi_i|_{\mathcal{G}_i} : \mathcal{Y}(F_i)|_{\mathcal{G}_i} \to H|_{\mathcal{G}_i}$ is an epimorphism.
        \item ker($\phi_i|_{\mathcal{G}_1} ) \subseteq $ ker$\big(\mathcal{Y}(F_i)|_{\mathcal{G}_1} \to \mathcal{Y}(F_{i + 1})|_{\mathcal{G}_1}\big)$, that is, the map $ \mathcal{Y}(F_i)|_{\mathcal{G}_1} \to \mathcal{Y}(F_{i + 1})|_{\mathcal{G}_1} $ factors as $\Yon{F_i}{1} \xrightarrow{\phi_i|_{\mathcal{G}_1}} H|_{\mathcal{G}_1} \to \Yon{F_{i + 1}}{1} $ for all $1 \leq i < n$
    \end{enumerate}
    In particular, there exists an object $F \in \overline{\mathsf{T}^c}$ with a strong $\mathcal{G}_n$-approximating system, and an epimorphism $\phi : \mathcal{Y}(F)|_{\mathcal{G}_n} \to H$.
\end{lemma}

\begin{proof}
    The proof is by induction on the integer $n \geq 1$. For $n = 1$, we prove a stronger statement, which helps in proving the induction step. 
     \begin{enumerate}
        \setcounter{enumi}{2}
         \item Let $H$ be a $\mathcal{G}$-semifinite $\mathcal{G}_1$-cohomological functor (\Cref{Definition finite bounded cohomological functors}). Then, there exists an object $F \in \mathsf{T}^c$ (\Cref{Definition closure of compacts}) and an epimorphism $\phi : \mathcal{Y}(F)|_{\mathcal{G}_1} \to H$, with $F$ having a strong $\mathcal{G}_1$-approximating system $E_*$ (\Cref{Definition strong G-approximating system}) such that each map $E_n \to E_{n + 1}$ is a split monomorphism.
     \end{enumerate}
     We now prove (3). Let $H$ be a $\mathcal{G}$-semifinite $\mathcal{G}_1$-cohomological functor. By the definition of a $\mathcal{G}$-semifinite functor, $H(G)$ is a finite $R$-module for all $G \in \bigcup_{n \in \mathbb{Z}}\mathcal{G}^n$, and is zero for $G \in \mathcal{G}^n$ for $n >> 0$. For each non-zero module $H(G^i_j)$, $i \in \mathbb{Z}$, $j \in J_i$, choose a finite set of generators $\{f_{ijk} : k \in K^i_j\}$, see \Cref{Setup} for the notation. By Yoneda lemma, each such element corresponds to a map, $\phi_{ijk} : \mathcal{Y}(G^i_j)|_{\mathcal{G}_1} \to H $. 

     Let $F = \bigoplus_{i \in \mathbb{Z}}\bigoplus_{j \in J_i} \bigoplus_{k \in K^i_j}G^i_j$. Then we have the obvious epimorphism $\phi : \mathcal{Y}(F)|_{\mathcal{G}_1} \to H$, given by,
     \[ \mathcal{Y}(F)|_{\mathcal{G}_1} \xrightarrow{\cong} \bigoplus_{i \in \mathbb{Z}}\bigoplus_{j \in J_i}\bigoplus_{k \in K^i_j}\mathcal{Y}(G^i_j)|_{\mathcal{G}_1} \xrightarrow{(\phi_{ijk})} H \]
     
     Now, we need to show the existence of a strong approximating system for $F$ with the required properties. As $\mathcal{R}$ is in the preferred $\mathbb{N}$-equivalence class of orthogonal metrics (\Cref{Definition preferred equivalence class of orthogonal metrics}), by \Cref{Remark on preferred equivalence class of orthogonal metrics}, there exists an integer $B>0$ such that $\mathcal{G}^{(-\infty,-i-B]} = \mathcal{M}^{\mathcal{G}}_{i+B} \subseteq \mathcal{R}_i$, see \Cref{Notation Subcategories for (pre)-generating sequences} and \Cref{Definition of generating sequence}(2). We define $E_i = \bigoplus_{n \geq -i - B + 1}\bigoplus_{j \in J_n}\bigoplus_{k \in K^n_j}G^n_j$. This is a finite direct sum by the construction of $F$, as $H$ is $\mathcal{G}$-semifinite, and so $E_i \in \mathcal{G}_1$ for all $i \in \mathbb{Z}$. The canonical map $E_i \to E_{i + 1}$ is clearly a split monomorphism. Further, the cone of the canonical map $E_i \to F$ lies in $\mathcal{G}^{(-\infty, -i - B]} \subseteq \mathcal{R}_{i}$ for all $i \in \mathbb{Z}$, and hence $E_*$ is a strong $\mathcal{G}_1$-approximating system for $F$. So, we have proved (3).

     Now, for the induction step. Suppose we know the result for all positive integers up to $n\geq 1$. Let $H$ be a $\mathcal{G}$-semifinite $\mathcal{G}_{n + 1}$-cohomological functor. Then $H|_{\mathcal{G}_n}$ is a $\mathcal{G}$-semifinite $\mathcal{G}_n$-cohomological functor, and so, by induction hypothesis, there exists a sequence $F_1 \to \cdots \to F_n$ in $\mathsf{T}^c$, with compatible epimorphisms $\phi_i : \mathcal{Y}(F_i)|_{\mathcal{G}_n} \to H|_{\mathcal{G}_n}$. In particular, we have the epimorphism $\phi_n : \Yon{F}{n} \to H|_{\mathcal{G}_n}$. Completing to a short exact sequence, we get, $0 \to H' \to \Yon{F}{n} \xrightarrow[]{\phi_n} H|_{\mathcal{G}_n} \to 0$. It is clear that $H'$ is also a $\mathcal{G}$-semifinite $\mathcal{G}_n$-cohomological functor. Now, we can use (3) to get an object $F' \in \mathsf{T}^c $ with a strong $\mathcal{G}_1$-approximating system $E'_*$ with the maps being split monomorphisms, and an epimorphism $\Yon{F'}{1} \to H'|_{\mathcal{G}_1}$.

     We have the maps $\Yon{F'}{1}\to H'|_{\mathcal{G}_1} \to \Yon{F_n}{1}$. As $F'$ admits a $\mathcal{G}_1$-approximating system (as any strong $\mathcal{G}_n$-approximating system gives a $\mathcal{G}_n$-approximating system as in \Cref{Definition of an Approximating Sequence}), by \cite[Lemma 5.8]{Neeman:2021b}, this composite is given by $\Yon{\alpha_n}{1}$ for a morphism $\alpha_n : F' \to F_n $.\\
    Now, we can apply \cite[Lemma 6.5]{Neeman:2021b} using \Cref{Remark regarding Neeman Lemma 6.5} to get that,
    \begin{enumerate}[label=(\roman*)]
        \item There exists a weak triangle (\Cref{Definition weak triangle}) $F' \xrightarrow{\alpha_n} F_n \xrightarrow{\beta_n} F_{n + 1} \to \Sigma F'$ in the category $\mathsf{T}^c$, and $F_{n + 1}$ admits a strong $\mathcal{G}_{n + 1}$-approximating system.
        \item There exists a map $\phi_{n + 1} : \mathcal{Y}(F_{n + 1}) \to H $ such that, $\phi_n$ is equal to the composite 
        \[ \Yon{F_{n}}{n} \xrightarrow{\Yon{\beta_n}{n}}\Yon{F_{n + 1}}{n} \xrightarrow{\phi_{n + 1}|_{\mathcal{G}_n}} H|_{\mathcal{G}_n}\]
    \end{enumerate}

    As $\Yon{F'}{1} \to H'|_{\mathcal{G}_{1}}$ is an epimorphism, we have that,
    \[ \Yon{F'}{1} \xrightarrow{\Yon{\alpha_n}{1}} \Yon{F_n}{1} \xrightarrow{\phi_n|_{\mathcal{G}_1}} H \]
    is exact from the short exact sequence $0 \to H' \to \mathcal{Y}(F_n)|_{\mathcal{G}_n} \xrightarrow{\phi_n} H \to 0$. By (ii), we further get that 
    \[ \Yon{F'}{1} \xrightarrow{\Yon{\alpha_n}{1}} \Yon{F_n}{1} \xrightarrow{\Yon{\beta_n}{1}} \Yon{F_{n + 1}}{1}\]
    is also exact by \cite[Lemma 6.4]{Neeman:2021b}, which gives us $(2)$.\\
    To prove (1), it just remains to show that $\phi_{n + 1}$ is an epimorphism. We note that we are in the setting of \cite[Lemma 6.6]{Neeman:2021b}, with $\mathcal{A} \colonequals \mathcal{G}_1$, $\mathcal{B} \colonequals \mathcal{G}_n$, and the weak triangle $F' \xrightarrow{\alpha_n} F_n \xrightarrow{\beta_n} F_{n + 1} \to \Sigma F'$. So, by \cite[Lemma 6.6]{Neeman:2021b}, we get that $\phi_{n + 1}$ is an epimorphism.

\end{proof}

\begin{remark}\label{Remark 2 of Section 6.2}
    Suppose we are in \Cref{Setup}. As the proof of the \Cref{Lemma 1 of Section 6.2} is using induction, if we are given $j < i$, and a $\mathcal{G}$-semifinite $\mathcal{G}_i$-cohomological functor (\Cref{Definition finite bounded cohomological functors}) $H$ with a sequence $F_1 \to \cdots \to F_j$ satisfying \Cref{Lemma 1 of Section 6.2} for $H|_{\mathcal{G}_j}$, we can extend the sequence to $F_1 \to \cdots \to F_i$ which satisfies \Cref{Lemma 1 of Section 6.2} for $H$. So, if we are given a $\mathcal{G}$-semifinite $\mathsf{T}^c$-cohomological functor $H$, we get an infinite sequence $F_1 \to F_2 \to F_3 \to \cdots $, such that for any $i>0$, $F_1 \to \cdots \to F_i$ satisfy \Cref{Lemma 1 of Section 6.2} for $H|_{\mathcal{G}_i}$. As $F_n$ has a strong $\mathcal{G}_n$-approximating system, by \cite[Corollary 5.4]{Neeman:2021b} we get that the morphisms $\phi_n : \mathcal{Y}(F_n)|_{\mathcal{G}_n} \to H|_{\mathcal{G}_n} $ lifts uniquely to $\phi_n : \mathcal{Y}(F_n) \to H$. Further, this lift is compatible with the maps $F_n \to F_{n + 1}$ because of the uniqueness.
\end{remark}

\begin{proposition}\label{Proposition 3 of Section 6.2}
    Suppose we are in \Cref{Setup} and let $\mathcal{K} = \bigoplus_{K\in \mathsf{T}^c} K$, where we take the direct sum over all the isomorphism classes of compacts. 
    Then, for any $\mathcal{G}$-semifinite $\mathsf{T}^c$-cohomological functor (\Cref{Definition finite bounded cohomological functors}) $H$ there exists an object $F \in \overline{\langle \mathcal{K} \rangle}_4 \subseteq \mathsf{T}$ and an isomorphism $\phi : \mathcal{Y}(F) \to H$, see \Cref{Notation from Neeman's paper}.
\end{proposition}
\begin{proof}
    By \Cref{Remark 2 of Section 6.2}, there exists a sequence $F_1 \xrightarrow{\beta_1} F_2 \xrightarrow{\beta_2} F_3 \xrightarrow{\beta_3} F_4 \xrightarrow{\beta_4} \cdots $ with compatible maps $\phi_n : \mathcal{Y}(F_n) \to H$. Let $F \colonequals \hocolim F_n$. Then, the natural map $\colim \mathcal{Y}(F_n) \to \mathcal{Y}(F)$ is an isomorphism by \cite[Lemma 2.8]{Neeman:1996}. So, we get a map $\phi : \mathcal{Y}(F) \to H$. We will now show that this map is an isomorphism.

    Consider the full subcategory of $\mathsf{T}^c$ defined as,
    \begin{equation*}
        \mathsf{S} = \left\{ E \in \mathsf{T}^c \Big| \ \phi(E) :\Hom{\mathsf{T}}{E}{F} \to H(E) \text{ is an isomorphism}\right\}
    \end{equation*}
    It is clear that $\mathsf{S}$ is closed under extensions and summands in $\mathsf{T}^c$. We just need to show that $\bigcup_{i \in \mathbb{Z}}\mathcal{G}^i \subseteq \mathsf{S}$, as $\operatorname{smd}(\operatorname{coprod}(\bigcup_{i \in \mathbb{Z}}\mathcal{G}^i)) = \mathsf{T}^c$ by \Cref{Definition of generating sequence}(1). For that, consider $\phi|_{\mathcal{G}_1} : \mathcal{Y}(F)|_{\mathcal{G}_1} \to H|_{\mathcal{G}_1}$.
    As $\mathcal{Y}(F) = \colim \mathcal{Y}(F_i)$, $\phi|_{\mathcal{G}_1}$ is a map to $H|_{\mathcal{G}_1}$ from the sequence, 
    \[\Yon{F_1}{1} \xrightarrow{\Yon{\beta_1}{1}} \Yon{F_2}{1} \xrightarrow{\Yon{\beta_2}{1}} \Yon{F_3}{1} \xrightarrow{\Yon{\beta_3}{1}} \cdots \]
    By \Cref{Lemma 1 of Section 6.2}, the map $\Yon{\beta_n}{1}$ factors as $\Yon{F_n}{1} \to H|_{\mathcal{G}_1} \to \Yon{F_{n + 1}}{1}$. So, the sequence is ind-isomorphic to the constant sequence with each term being $H|_{\mathcal{G}_1}$. And so, $\phi|_{\mathcal{G}_1}$ is an isomorphism, and hence $\bigcup_{i \in \mathbb{Z}}\mathcal{G}^i \subseteq \mathcal{G}_1 \subseteq \mathsf{S}$. Hence, $\phi$ is an isomorphism. 
    
    It just remains to show that $F \in \overline{\langle \mathcal{K} \rangle}_4$. Note that $F = \hocolim F_n$, and each $F_n$ has a strong $\mathsf{T}^c$-approximating system (see \Cref{Remark 2 of Section 6.2}), so they in turn are homtopy colimits of compact objects. That is, $F_n \in \overline{\langle \mathcal{K} \rangle}_2$ for all $n \in \mathbb{N}$, and hence 
    \[ F = \hocolim F_n \in \overline{\langle \mathcal{K} \rangle}_2 \star \overline{\langle \mathcal{K} \rangle}_2 = \overline{\langle \mathcal{K} \rangle}_4 \]
    and we are done.

\end{proof}
\begin{remark}\label{Remark on choice of metric for computing Tc and Tbc}
    From now on, we will be adding the assumption that $\mathsf{T}$ is weakly $\mathcal{G}$-quasiapprox (see \Cref{Definition of G-quasiapproximability}) on top of the assumptions in \Cref{Setup}. We note here that with the assumption of weak $\mathcal{G}$-quasiapprox, the extended good metrics $\mathcal{M}^{\mathcal{G}}$ (\Cref{Definition of generating sequence}) and $\mathcal{R} \cap \mathsf{T}^c$ on $\mathsf{T}^c$ are $\mathbb{N}$-equivalent by \Cref{Proposition approximable implies preferred equivalence class} for any $\mathcal{R}$ in the preferred equivalence class of orthogonal metrics on $\mathsf{T}$ (see \Cref{Definition preferred equivalence class of orthogonal metrics}).
\end{remark}
\begin{lemma}\label{Lemma 3 of Section 6.2}
    In addition to the assumptions of \Cref{Setup}, we assume that $\mathsf{T}$ is weakly $\mathcal{G}$-quasiapprox, see \Cref{Definition of G-quasiapproximability}. Let $H$ be a $\mathcal{G}$-semifinite $\mathsf{T}^c$-cohomological functor (\Cref{Definition finite bounded cohomological functors}). Then, there exists a positive integer $\widetilde{A}$ such that for each $n\geq 1$ there exists
    \begin{enumerate}
        \item an object $F_n \in \mathsf{T}^c\cap \mathcal{R}_{-\widetilde{A}}$ (see \Cref{Definition closure of compacts}) and,
        \item a natural transformation $\phi_n : \mathcal{Y}(F_n) \to H$ for each $n \geq 1 $ such that $\phi_n|_{\mathcal{G}_n}$ is an epimorphism.
    \end{enumerate} 
\end{lemma}
\begin{proof}
    Firstly, as $\mathsf{T}$ is weakly $\mathcal{G}$-quasiapprox and $\mathcal{R}$ lies in the preferred equivalence class of orthogonal metrics (\Cref{Definition preferred equivalence class of orthogonal metrics}), by \Cref{Proposition G-approximability from orthogonal metric in preferred equivalence class}, there exists an integer $A > 0$ such that $\mathcal{G}$, $\mathcal{R}$, and $A$ satisfy \Cref{Definition of G-quasiapproximability}(1)-(3).
    
    Note that by \Cref{Lemma 1 of Section 6.2} and \Cref{Remark 2 of Section 6.2}, we can always find an object $F_n \in \mathsf{T}^c$, with a map $\phi_n : \mathcal{Y}(F_n) \to H$, which is an epimorphism when restricted to $\mathcal{G}_n$. Further, by \Cref{Lemma Definition of G-approximability (1)}(2), $\mathsf{T}^c \subseteq \bigcup_{i \in \mathbb{Z}} \mathcal{R}_i$. So, the additional assertion is on the uniform bound on which $\mathcal{R}_i$ each of the $F_n$ is contained in.

    As $H$ is $\mathcal{G}$-semifinite, there exists an integer $B' > 0$ such that $H(\mathcal{G}[B', \infty)) = 0$, see \Cref{Notation Subcategories for (pre)-generating sequences}. And so, $H$ is zero on $\mathcal{G}^{[B', \infty)}$. By \Cref{Corollary 4 of section 5}, there exists an integer $B > 0$ such that for all $F \in \mathsf{T}^c \cap \mathcal{R}_j$ and $ i > 0$, there exists a triangle, $E_i \to F \to D_i \to \Sigma E_i$ such that $E_i \in \mathcal{G}^{[-B-j-i+1,\infty)} $ and $D_i \in \mathcal{R}_{j+i}$. Let $\widetilde{A} = B+B'+1$.

    Consider the object $F_n$ we started with. As $F_n \in \mathsf{T}^c \subseteq \bigcup_{i \in \mathbb{Z}} \mathcal{R}_i$ by \Cref{Lemma Definition of G-approximability (1)}(2), there exists a positive integer $l$ such that $F_n \in \mathcal{R}_{-l}$. 
    By \Cref{Corollary 4 of section 5} on $F_n$, with $i = l - B - B'$, we get a triangle $E \to F_n \to D \to \Sigma E$ with $D \in \mathcal{R}_{-l+i} = \mathcal{R}_{-(B + B')}$ and $E \in \mathcal{G}^{[B'+1,\infty) } $. So, $H(E) = 0 $. And so, by the Yoneda lemma, the composite $\mathcal{Y}(E) \xrightarrow{\mathcal{Y}(\alpha)} \mathcal{Y}(F_n) \xrightarrow{\phi}H $ vanishes.

    We now apply \cite[Lemma 6.5]{Neeman:2021b} in full force, with $\mathcal{A} = \mathcal{B} = \mathsf{T}^c $, see also \Cref{Remark regarding Neeman Lemma 6.5}.  There exists the constant $\mathsf{T}^c$-approximating system for $E \in \mathsf{T}^c$. The connecting morphisms are all identity, hence in particular are split monomorphisms, and so condition (vii) of \cite[Lemma 6.5]{Neeman:2021b} is satisfied. There exists a strong $\mathcal{G}_n$-approximating system $C_*$ for $F_n \in \mathsf{T}^c$ by \Cref{Lemma on strong approximating sequences}. And finally, we have shown in the previous paragraph that the composite $\mathcal{Y}(E) \xrightarrow{\mathcal{Y}(\alpha)} \mathcal{Y}(F_n) \xrightarrow{\phi}H $ vanishes. Up to passing to a subsequence, the map $E \to F_n$ induces a map between the constant sequence given by $E$ and the sequence $C_*$. So, by \cite[Lemma 6.5]{Neeman:2021b} there exists a weak triangle $E \xrightarrow{\alpha} F_n \xrightarrow{\beta} \widetilde{D} \to \Sigma F_n$, where $\widetilde{D} = \hocolim D_i$, where $D_i = \operatorname{Cone}(E \to C_i)$ (see \Cref{Remark regarding Neeman Lemma 6.5}), such that $\phi : \mathcal{Y}(F_n) \to H $, factors as, 
    \[ \mathcal{Y}(F_n) \xrightarrow{\mathcal{Y}(\beta) }  \mathcal{Y}(\widetilde{D}) \xrightarrow{\psi} H \]
     The surjectivity of $\phi|_{\mathcal{G}_n}$ implies that of $\psi|_{\mathcal{G}_n}$.

     Finally, we show that $\widetilde{D}$ belongs to $\mathsf{T}^c \cap \mathcal{R}_{-\widetilde{A}}$. Note that $D = \operatorname{Cone}(E \to F_n ) \in \mathcal{R}_{-(B+B')}$, and as $C_\star$ is a strong $\mathcal{G}_n$-approximating system for $F$, we have that $\operatorname{Cone}(C_i \to F_n) \in \mathcal{R}_{i+1}$ for all $i$. So, by a simple application of the octahedral axiom to $E \to C_i \to F_n $  we get that $D_i = \operatorname{Cone}(E \to C_i) \in \mathcal{R}_{-(B+B')}$. By \Cref{Remark regarding Neeman Lemma 6.5}, $\widetilde{D} = \hocolim{D_i} \in \mathcal{R}_{-(B+B'+1)}$. So, the requited result holds with $\widetilde{A} = B + B' + 1 $.

\end{proof}

\begin{lemma}\label{Lemma 4 of Section 6.2}
     In additions to the assumptions of \Cref{Setup}, we further assume that $\mathsf{T}$ is $\mathcal{G}$-quasiapprox (\Cref{Definition of G-quasiapproximability}). Let $B>0$ be the integer we get from \Cref{Lemma 2 of section 5}. Let $H$ be a $\mathcal{G}$-semifinite $\mathsf{T}^c$-cohomological functor (\Cref{Definition finite bounded cohomological functors}). 
     Assume we are give objects $\widetilde{F}, F' \in \mathsf{T}^c \cap \mathcal{R}_i$ and $E \in \mathsf{T}^c\cap \mathcal{R}_i$, and a morphism $\alpha : E \to \widetilde{F}$ in $\mathsf{T}^c$, see \Cref{Definition closure of compacts}. Further, assume we are given a positive integer $m$, and a morphism $\widetilde{\phi} : \mathcal{Y}(\widetilde{F}) \to H $, which restricts to an epimorphism on $\mathcal{G}_{mB}$ and a morphism $\phi'  : \mathcal{Y}(F') \to H$ which restricts to an epimorphism on $\mathcal{G}_{(m+1)B}$. 

     Then, we get the following,
     \begin{enumerate}
         \item There exists a commutative diagram in $\mathsf{T}^c \cap \mathcal{R}_i$ as follows,
    \[\begin{tikzcd}
	E && {E'} && {F'} \\
	& {\widetilde{F}} && {\widetilde{F}'}
	\arrow["\alpha"', from=1-1, to=2-2]
	\arrow["\gamma", from=1-3, to=2-2]
	\arrow["\varepsilon", from=1-1, to=1-3]
	\arrow["{\alpha '}", from=1-3, to=2-4]
	\arrow["\beta"', from=1-5, to=2-4]
    \end{tikzcd}\]

    \item Furthermore, there exists a morphism $\widetilde{\phi}' : \mathcal{Y}(\widetilde{F}') \to H $ such that :
     \begin{enumerate}[label=(\roman*)]
         \item $E' \in \mathsf{T}^c$
         \item $\operatorname{Cone}(\alpha'), \operatorname{Cone}(\gamma) \in \mathcal{R}_{m-1+i}$
         \item The following diagram commutes,
         \[\begin{tikzcd}
	& {\mathcal{Y}(\widetilde{F}')} \\
	{\mathcal{Y}(F')} & H
	\arrow["\mathcal{Y}(\beta)", from=2-1, to=1-2]
    \arrow["\widetilde{\phi}'",from=1-2, to=2-2]
	\arrow["\phi'",from=2-1, to=2-2]
        \end{tikzcd}\]
         which implies that $\widetilde{\phi}'|_{\mathcal{G}_{(m+1)B}}$ is an epimorphism.
        \item The following square commutes,
        \[\begin{tikzcd}
	{\mathcal{Y}(E')} && {\mathcal{Y}(\widetilde{F})} \\
	{\mathcal{Y}(\widetilde{F}')} && H
	\arrow["{\widetilde{\phi}}", from=1-3, to=2-3]
	\arrow["{\mathcal{Y}(\alpha ')}"', from=1-1, to=2-1]
	\arrow["{\widetilde{\phi}'}"', from=2-1, to=2-3]
	\arrow["{\mathcal{Y}(\gamma)}", from=1-1, to=1-3]
        \end{tikzcd}\]
    \end{enumerate}
    \end{enumerate}

\end{lemma}

\begin{proof}
    Firstly, as $\mathsf{T}$ is $\mathcal{G}$-quasiapprox and $\mathcal{R}$ lies in the preferred equivalence class of orthogonal metrics (\Cref{Definition preferred equivalence class of orthogonal metrics}), by \Cref{Proposition G-approximability from orthogonal metric in preferred equivalence class}, there exists an integer $A > 0$ such that $\mathcal{G}$, $\mathcal{R}$, and $A$ satisfy the \Cref{Definition of G-quasiapproximability}(1)-(2) and (3$'$).
    
    We apply \Cref{Corollary 4 of section 5} to $F'$ and $m > 0 $ to get a triangle $E_m' \xrightarrow{a} F' \xrightarrow{b} D_m \to \Sigma E_m'$ with $D_m \in \mathcal{R}_{i+m}$ and $E_m' \in \mathcal{G}^{[-B-m+1-i,B-i]}_{mB} \subseteq \mathcal{G}_{mB}$, see \Cref{Notation Subcategories for (pre)-generating sequences}.

    As $E_m' \in \mathsf{T}^c$, by Yoneda lemma, the map $\mathcal{Y}(E_m') \xrightarrow{\mathcal{Y}(f)} \mathcal{Y}(F') \xrightarrow{\phi'} H$ corresponds to an element $x \in H(E_m')$. As $E_m' \in \mathcal{G}_{mB} $, the map 
    \[\widetilde{\phi}(E_m') : \Hom{\mathsf{T}}{E_m'}{\widetilde{F}} = \mathcal{Y}(\widetilde{F})(E_m') \to H(E_m')\]
    is surjective by hypothesis. So, there exists a map $f : E_m' \to \widetilde{F} $ such that $\widetilde{\phi}(f) = x$, which by Yoneda lemma give us the following commutative square, 
    \begin{equation}\label{commutative diagram 1}
        \begin{tikzcd}
	   \mathcal{Y}(E_m') && \mathcal{Y}(\widetilde{F}) \\
	   \mathcal{Y}(F') && H
	   \arrow["\mathcal{Y}(f)", from=1-1, to=1-3]
	   \arrow["\mathcal{Y}(a)"', from=1-1, to=2-1]
	   \arrow["{\phi'}"', from=2-1, to=2-3]
	   \arrow["{\widetilde{\phi}}", from=1-3, to=2-3]
        \end{tikzcd}
    \end{equation}

    As $\widetilde{F} \in \mathsf{T}^c$ it has a strong $\mathsf{T}^c$-approximating system $\widetilde{E}_*$ by \Cref{Lemma on strong approximating sequences}. Again, by \Cref{Lemma on strong approximating sequences}, $\hocolim \widetilde{E}_i \cong \widetilde{F}$. Consider the map $(\alpha, f) : E \oplus E_m' \to \widetilde{F}$. By \cite[Lemma 2.8]{Neeman:1996}, this map factors through $\widetilde{E}_j$ for some $j \geq 1$. We can choose $ \widetilde{E}_j $ with $j \geq m + i $, and define $E' \colonequals \widetilde{E}_j$. So, by above, we get that $(\alpha, f) : E \oplus E_m' \to \widetilde{F}$ is given by the composite,
    \[ E \oplus E_m' \xrightarrow{(\varepsilon, g) } E' \xrightarrow{\gamma} \widetilde{F} \] 
    By construction, $\operatorname{Cone}(\gamma) \in \mathcal{R}_{m+1+i}\subseteq \mathcal{R}_{i+1}$. Further, as $\widetilde{F} \in \mathcal{R}_i$, we get that $E' \in \mathcal{R}_i$, as $E' \in \Sigma^{-1}\operatorname{Cone}(\gamma) \star \widetilde{F}$. Similarly, from the triangle on the morphism $a$, along with the fact that $F', \Sigma^{-1}\operatorname{Cone}(a) = D_m \in \mathcal{R}_i$ give us that $E_m' \in \mathcal{R}_i$. 

    Now, we get the following commutative diagram from \Cref{commutative diagram 1}, using $f = \gamma \circ g$,
    \begin{equation}\label{commutative diagram 2}
        \begin{tikzcd}
	   \mathcal{Y}(E_m') && \mathcal{Y}(E') \\
	   \mathcal{Y}(F') && H
	   \arrow["\mathcal{Y}(g)", from=1-1, to=1-3]
	   \arrow["\mathcal{Y}(a)"', from=1-1, to=2-1]
	   \arrow["{\phi'}"', from=2-1, to=2-3]
	   \arrow["{\varrho = \widetilde{\phi}\circ \mathcal{Y}(\gamma)}", from=1-3, to=2-3]
        \end{tikzcd}
    \end{equation}
    From this diagram, we get a map $ \sigma =  \begin{pmatrix}
        -g \\
        \ \ a
    \end{pmatrix}  : E_m' \to E' \oplus F' $ in $ \mathsf{T}^c \cap \mathcal{R}_i$, and a morphism $( \varrho, \phi' ) : \mathcal{Y}( E' \oplus F' ) \to H $, such that the following composite vanishes,
    \[\begin{tikzcd}
	   \mathcal{Y}(E_m') && \mathcal{Y}(E' \oplus F') && H
	   \arrow["{\mathcal{Y}(\sigma)}", from=1-1, to=1-3]
	   \arrow["{( \varrho, \phi}' )", from=1-3, to=1-5]
    \end{tikzcd}\]

    $E_m'$ is a compact object so it has a trivial strong $\mathsf{T}^c$-approximating system, given by the constant sequence. $E' \oplus F' \in \mathsf{T}^c$, and hence by \Cref{Lemma on strong approximating sequences}(3) it also has a strong $\mathsf{T}^c$-approximating system. For the purposes of this proof, we make a choice of the first few terms of this sequence. We can do this as the proof of \Cref{Lemma on strong approximating sequences}(3) goes via induction. It is clear that $E' \oplus F'_*$ is a strong $\mathsf{T}^c$-approximating system for $E' \oplus F'$ for any strong $\mathsf{T}^c$-approximating system $F'_* $ for $F'$. We can choose the sequence $F'_*$ so that $F_n' = E_m'$ for all $n \leq m+i$ with identity maps, and the map to $F'$ being $a : E_m' \to F' $, as $\operatorname{Cone}(a) = D_m \in \mathcal{R}_{i + m}$. This is our choice of the approximating system for $E' \oplus F'$.

    By \cite[Lemma 6.5]{Neeman:2021b} and \Cref{Remark regarding Neeman Lemma 6.5}, we can complete the map $\sigma$ to a weak triangle in $\mathsf{T}^c$,
    \[\begin{tikzcd}
	E_m' && E' \oplus F'  && \widetilde{F}' && \Sigma E_m'
	\arrow["\sigma", from=1-1, to=1-3]
	\arrow["{(\alpha', \beta)}" , from=1-3, to=1-5]
	\arrow["{\tau}", from=1-5, to=1-7]
    \end{tikzcd}\]
    such that $(\varrho, \phi ') = \widetilde{\phi}' \circ \mathcal{Y}((\alpha', \beta)) $ for a morphism $   \widetilde{\phi}' : \mathcal{Y}(\widetilde{F}') \to H $.

    Recall that the strong $\mathsf{T}^c$-approximating system $\widetilde{F}_*'$ we get for $\widetilde{F}'$ by \Cref{Remark regarding Neeman Lemma 6.5} is given by $\widetilde{F}_n' \colonequals \operatorname{Cone}(E_m' \to E' \oplus F_{n + 1}')$. In particular, for $n = m + i - 1 $, as $F'_{m+i} = E_{m}'$ by the previous paragraph, we get that,
    \[\widetilde{F}_{m + i - 1}' \colonequals \operatorname{Cone}(E_m' \to E' \oplus F_{m +i}' = E' \oplus E_{m}') \cong E'\]
    and as can be easily checked, the map from $E' \oplus F_{m+i}'$ to $\widetilde{F}_{m + i - 1}'$ is exactly the projection map. This gives us that the obvious map $\widetilde{F}_{m +i - 1}' \to \widetilde{F}'$ is the same as $\alpha' : E' \to \widetilde{F}'$, which implies that $\operatorname{Cone}(\alpha') \in \mathcal{R}_{m+i-1} $. Further, as we already know that $E' \in \mathcal{R}_i $, this gives us that $\widetilde{F}'$ is in $\mathsf{T}^c \cap \mathcal{R}_i$.
    
    From \Cref{commutative diagram 2} and the fact that $(\varrho, \phi ')$ factors through $\mathcal{Y}((\alpha', \beta))$, Yoneda Lemma gives us the following commutative square in $ \mathsf{T}^c $,
    \[
        \begin{tikzcd}
	   E_m' && E' \\
	   F' && \widetilde{F}'
	   \arrow["g", from=1-1, to=1-3]
	   \arrow["a"', from=1-1, to=2-1]
	   \arrow["\beta"', from=2-1, to=2-3]
	   \arrow["{\alpha'}", from=1-3, to=2-3]
        \end{tikzcd}
    \]
    Further, we also get the following commutative diagram from above, 

    \[\begin{tikzcd}
	   \mathcal{Y}(E_m') && \mathcal{Y}(E') && \mathcal{Y}(\widetilde{F}) \\
	   \mathcal{Y}(F') && \mathcal{Y}(\widetilde{F}') \\
	   &&&& H
	   \arrow["{\mathcal{Y}(g)}", from=1-1, to=1-3]
	   \arrow["{\mathcal{Y}(a)}", from=1-1, to=2-1]
	   \arrow["{\mathcal{Y}(\beta)}", from=2-1, to=2-3]
	   \arrow["{\mathcal{Y}(\alpha')}", from=1-3, to=2-3]
	   \arrow["{\mathcal{Y}(\gamma)}", from=1-3, to=1-5]
	   \arrow["{\widetilde{\phi}}", from=1-5, to=3-5]
	   \arrow["{\widetilde{\phi}'}", from=2-3, to=3-5]
	   \arrow["{\phi'}", bend right=20, from=2-1, to=3-5]
    \end{tikzcd}\]
    
    which shows (iii) and (iv) completing the proof.

\end{proof}

\begin{lemma}\label{Lemma 5 of Section 6.2}
    Suppose we are in \Cref{Setup}. Further assume that $\mathsf{T}$ is $\mathcal{G}$-quasiapprox (\Cref{Definition of G-quasiapproximability}). Let $H$ be a $\mathcal{G}$-semifinite $\mathsf{T}^c$-cohomological functor (\Cref{Definition finite bounded cohomological functors}). Then, there exists an object $F \in \mathsf{T}^c$ (\Cref{Definition closure of compacts}) and an epimorphism $\mathcal{Y}(F) \to H$.
\end{lemma}
\begin{proof}
    As $\mathsf{T}$ is $\mathcal{G}$-quasiapproximable and $\mathcal{R}$ lies in the preferred equivalence class of orthogonal metrics (\Cref{Definition preferred equivalence class of orthogonal metrics}), by \Cref{Proposition G-approximability from orthogonal metric in preferred equivalence class}, there exists an integer $A > 0$ such that $\mathcal{G}$, $\mathcal{R}$, and $A$ satisfy the \Cref{Definition of G-quasiapproximability}(1)-(2) and (3$'$).
    
    By \Cref{Lemma 3 of Section 6.2}, there exists a positive integer $\widetilde{A}$, and for each positive integer $n$, an object $F_n \in \mathsf{T}^c \cap \mathcal{R}_{-\widetilde{A}}$ and a morphism $\phi_n : \mathcal{Y}(F_n) \to H $ which is an epimorphism on $\mathcal{G}_n$. Let $B > 0$ be the integer we get from \Cref{Lemma 2 of section 5}.

    We will use induction to produce in $ \mathsf{T}^c \cap \mathcal{R}_{-\widetilde{A}}$ a sequence,
    \[\begin{tikzcd}
	{E_1} && {E_2} && E_3 && \cdots \\
	& {\widetilde{F}_1} && {\widetilde{F}_2} && {\widetilde{F}_3}
	\arrow["{\varepsilon_2}", from=1-1, to=1-3]
	\arrow["{\alpha_1}"', from=1-1, to=2-2]
	\arrow["{\gamma_2}", from=1-3, to=2-2]
	\arrow["{\varepsilon_3}", from=1-3, to=1-5]
	\arrow["{\alpha_2}"', from=1-3, to=2-4]
	\arrow["{\gamma_3}", from=1-5, to=2-4]
	\arrow["{\alpha_3}"', from=1-5, to=2-6]
	\arrow["{\varepsilon_4}", from=1-5, to=1-7]
        \end{tikzcd}\]

    such that for all $i \geq 1$, $E_i \in \mathsf{T}^c$ and there exist maps $\widetilde{\phi}_i : \mathcal{Y}(\widetilde{F}_i) \to H$ satisfying,
    \begin{enumerate}
        \item $\widetilde{\phi}_i|_{\mathcal{G}_{(i + 3)B}}$ is an epimorphism.
        \item 
        \[\begin{tikzcd}
	       \mathcal{Y}(E_{i + 1}) && \mathcal{Y}(\widetilde{F}_{i + 1}) \\
	       \mathcal{Y}(\widetilde{F}_i) && H
	       \arrow["{\mathcal{Y}(\alpha_{i + 1})}", from=1-1, to=1-3]
	       \arrow["{\mathcal{Y}(\gamma_{i + 1})}", from=1-1, to=2-1]
	       \arrow["{\widetilde{\phi}}_i",from=2-1, to=2-3]
	       \arrow["\widetilde{\phi}_{i+1}",from=1-3, to=2-3]
        \end{tikzcd}\]
        commutes for each $i\geq 1$.
        \item $\operatorname{Cone}(\gamma_i), \operatorname{Cone}(\alpha_i) \in \mathcal{R}_{i + 1 - \widetilde{A}}$
    \end{enumerate}
    We proceed via induction on $n \geq 1$. For $n = 1$, we define $\widetilde{F}_1 = F_{4B}$, and $\widetilde{\phi}_1 = \phi_{4B}$. As $\widetilde{F}_1 \in \mathsf{T}^c$, there exists a triangle $E_1 \xrightarrow{\alpha_1} \widetilde{F}_1 \to D_1 \to \Sigma E_1$ with $E_1 \in \mathsf{T}^c$ and $D_1 \in \mathcal{R}_{1-\widetilde{A}}$. As $\widetilde{F}_1 \in \mathcal{R}_{-\widetilde{A}}$, we get that $E_1 \in \mathsf{T}^c \cap \mathcal{R}_{-\widetilde{A}} $. This proves the result for $n = 1$

    Now assume we know the result for up to a positive integer $n$. In particular, we have a map $\alpha_n : E_n \to \widetilde{F}_n$ in $\mathsf{T}^c \cap \mathcal{R}_{-\widetilde{A}}$, with $E_n \in \mathsf{T}^c$ and $\operatorname{Cone}(\alpha_n) \in \mathcal{R}_{n+1-\widetilde{A}}$, and a map $\widetilde{\phi}_n : \mathcal{Y}(\widetilde{F}_n) \to H$ which is an epimorphism on $\mathcal{G}_{(n + 3)B}$. 
    
    Again, by \Cref{Lemma 3 of Section 6.2}, there exists an object $F_{(n + 4)B}\in \mathsf{T}^c \cap \mathcal{R}_{-\widetilde{A}} $ with a morphism  $\phi_{(n + 4)B} : \mathcal{Y}(F_{(n + 4)B}) \to H $, such that the restriction to $\mathcal{G}_{(n + 4)B}$ is a surjection. By applying \Cref{Lemma 4 of Section 6.2} with $m = n + 3$ and $i = - \widetilde{A}$, we get in $\mathsf{T}^c \cap \mathcal{R}_{ - \widetilde{A}} $ a diagram,

    \[\begin{tikzcd}
	   {E_n} && {E_{n + 1}} && {F_{(n + 4)B}} \\
	   & {\widetilde{F}_n} && {\widetilde{F}_{n + 1}}
	   \arrow["{\alpha_n}"', from=1-1, to=2-2]
	   \arrow["{\gamma_{n + 1}}", from=1-3, to=2-2]
	   \arrow["{\varepsilon_{n + 1}}", from=1-1, to=1-3]
	   \arrow["{\alpha_{n + 1}}", from=1-3, to=2-4]
	   \arrow["\beta"', from=1-5, to=2-4]
    \end{tikzcd}\]
    as well as a map $\widetilde{\phi}_{n + 1} : \mathcal{Y}(\widetilde{F}_{n + 1}) \to H$ satisfying the properties listed in the statement of \Cref{Lemma 4 of Section 6.2}. \Cref{Lemma 4 of Section 6.2}(i) gives us that $E_{n + 1} \in \mathsf{T}^c$. \Cref{Lemma 4 of Section 6.2}(ii) tells us that $\operatorname{Cone}(\gamma_{n+1}), \operatorname{Cone}(\alpha_{n+1}) \in \mathcal{R}_{m-1+i} = \mathcal{R}_{n+2-\widetilde{A}} = \mathcal{R}_{(n+1) + 1 -\widetilde{A}}$. \Cref{Lemma 4 of Section 6.2}(iii) says that $\widetilde{\phi}_{n + 1}|_{\mathcal{G}_{(n + 4)B}}$ is an epimorphism. And finally, \Cref{Lemma 4 of Section 6.2}(iv) tells us that the square in (2) commutes. This completes the induction process.

    So, we have a sequence $E_1 \xrightarrow{\varepsilon_1} E_2 \xrightarrow{\varepsilon_2} E_3 \xrightarrow{\varepsilon_3} \cdots$. We define the map $\psi_i : \mathcal{Y}(E_i) \to H$ to be the composite $\mathcal{Y}(E_i) \xrightarrow{\mathcal{Y}(\alpha_i)} \mathcal{Y}(\widetilde{F}_i) \xrightarrow{\widetilde{\phi}_i} H$. Next, we prove that the following diagram commutes,
    \begin{equation}\label{Diagram 4}
        \begin{tikzcd}
	       &&& {\mathcal{Y}(E_{i + 1})} \\
	       {\mathcal{Y}(E_i)} \\
	       &&& H
	       \arrow["{\mathcal{Y}(\varepsilon_{i + 1})}",from=2-1, to=1-4]
	       \arrow["{\psi_{i + 1}}",from=1-4, to=3-4]
	       \arrow["{\psi_i}",from=2-1, to=3-4]
        \end{tikzcd}
    \end{equation}
    Notice that the commutativity of \Cref{Diagram 4} is the same as the commutativity of the boundary of the diagram \Cref{Diagram 5} below,

    \begin{equation}\label{Diagram 5}
        \begin{tikzcd}
	       && \mathcal{Y}(E_{i + 1}) && \mathcal{Y}(\widetilde{F}_{i + 1}) \\
	       \mathcal{Y}(E_i) && \mathcal{Y}(\widetilde{F}_i) && H
	       \arrow["{\mathcal{Y}(\gamma_{i + 1})}", from=1-3, to=2-3]
	       \arrow["{\mathcal{Y}(\alpha_{i + 1})}", from=1-3, to=1-5]
	       \arrow["{\widetilde{\phi}_i }", from=2-3, to=2-5]
	       \arrow["{\widetilde{\phi}_{i + 1}}",  from=1-5, to=2-5]
	       \arrow["{\mathcal{Y}(\varepsilon_{i + 1})}",from=2-1, to=1-3]
	       \arrow["{\mathcal{Y}(\alpha_i)}",from=2-1, to=2-3]
        \end{tikzcd}
    \end{equation}
    The square in the above diagram commutes by (2), and the triangle commutes as the diagram \Cref{Diagram 6} below commutes,
    
    \begin{equation}\label{Diagram 6}
        \begin{tikzcd}
	       && {E_{i + 1}} && {} \\
	       {E_i} && {\widetilde{F}_i}
	       \arrow["{\gamma_{i + 1}}", from=1-3, to=2-3]
	       \arrow["{\varepsilon_{i + 1}}", from=2-1, to=1-3]
	       \arrow["{\alpha_i}"', from=2-1, to=2-3]
        \end{tikzcd}
    \end{equation}
    
    This gives us the commutativity of the boundary of \Cref{Diagram 5}, and hence the commutativity of \Cref{Diagram 4}. This in turn helps us define a map to $H$ as follows. Let $F \colonequals \hocolim E_i$. We define $\psi : \mathcal{Y}(F) \to H$ as the colimit of $\psi_i$, noting that $\mathcal{Y}(F) = \colim \mathcal{Y}(E_i)$ by \cite[Lemma 2.8]{Neeman:1996}.

    Now applying the octahedral axiom to \Cref{Diagram 6}, we get that $\operatorname{Cone}(\varepsilon_{i + 1}) \in \mathcal{R}_{i+1-\widetilde{A}}$, as $\operatorname{Cone}(\alpha_i),\Sigma^{-1}\operatorname{Cone}(\gamma_{i+1}) \in \mathcal{R}_{i+1-\widetilde{A}}$, which gives us that up to passing to a subsequence, $E_*$ is a strong $\mathsf{T}^c$-approximating system (\Cref{Definition strong G-approximating system}), and so, by \Cref{Lemma on strong approximating sequences}, $F = \hocolim E_i \in \mathsf{T}^c$.

    Finally, it remains to show that $\psi$ is an epimorphism. Let $C \in \mathsf{T}^c$. As $\mathcal{G}$ is a generating sequence (see \Cref{Definition of generating sequence}(1)), we have that $\mathsf{T} = \bigcup_{i \geq 1} \mathcal{G}_i$, see \Cref{Notation Subcategories for (pre)-generating sequences}. So, $C \in \mathcal{G}_{(n + 3)B}$ for some $n > 0$. By \Cref{Lemma Compacts orthogonal to the metric}, we can further choose $n$ such that $\Hom{\mathsf{T}}{C}{\mathcal{R}_{n+1-\widetilde{A}}} = 0$. So, we get that the map $\Hom{\mathsf{T}}{C}{\widetilde{F}_n} \xrightarrow{\widetilde{\phi}_n(C)} H(C) $ is surjective as $\widetilde{\phi}_n|_{\mathcal{G}_{(n + 3)B}}$ is an epimorphism. We know that $\widetilde{D}_n \colonequals \operatorname{Cone}(\alpha_{n}) \in \mathcal{R}_{n+1-\widetilde{A}}$. So, we get the exact sequence,
    \[ \Hom{\mathsf{T}}{C}{E_n} \xrightarrow{\Hom{\mathsf{T}}{C}{\alpha_n}} \Hom{\mathsf{T}}{C}{\widetilde{F}_n} \xrightarrow{}   \Hom{\mathsf{T}}{C}{\widetilde{D}_n} = 0\]
    So, $\psi_n(C) : \Hom{\mathsf{T}}{C}{E_n} \to H(C)$ is surjective as it is a composite of two surjective morphism, namely $\Hom{}{C}{\alpha_n}$ and $\widetilde{\phi}_n(C)$. But, $\psi_n(C)$ factors via $\psi(C) : \Hom{\mathsf{T}}{C}{F} \to H(C)$, that is $\psi_n(C)= \Hom{}{C}{F_n} \to \Hom{}{C}{F} \xrightarrow{\psi(C)} H(C)$, so $\psi(C)$ also has to be surjective. As this is true for any $C \in \mathsf{T}^c$, we are done.

\end{proof}

\begin{corollary}\label{Corollary 7.17}
    In addition to the conditions of \Cref{Setup}, let $\mathsf{T}$ be $\mathcal{G}$-quasiapprox (\Cref{Definition of G-quasiapproximability}). Let $F' \in \mathsf{T}$ be an object such that $H = \mathcal{Y}(F')$ is a $\mathcal{G}$-semifinite $\mathsf{T}^c$-cohomological functor (\Cref{Definition finite bounded cohomological functors}). Then, there exists a triangle $F \xrightarrow{f} F' \xrightarrow{g} D \to \Sigma F$ with $F \in \mathsf{T}^c$ (\Cref{Definition closure of compacts}) and $g$ a phantom map, see \Cref{Definition of a Phantom map}.
\end{corollary}
\begin{proof}
    
    By \Cref{Lemma 5 of Section 6.2}, we get $F \in \mathsf{T}^c$ and an epimorphism $\phi : \mathcal{Y}(F) \to H = \mathcal{Y}(F')$. By \Cref{Lemma on strong approximating sequences}(3), there exists a strong approximating system for $F \in \mathsf{T}^c$. So, by \cite[Lemma 5.8]{Neeman:2021b}, there exists a map $f : F \to F'$ so that $\phi = \mathcal{Y}(f)$. Consider the triangle $F \xrightarrow{f} F' \xrightarrow{g} D \to \Sigma F$ we get from completing $f$ to a triangle. For any object $C \in \mathsf{T}^c$, we get an exact sequence,
    \[\Hom{\mathsf{T}}{C}{F} \xrightarrow{\Hom{\mathsf{T}}{C}{f}} \Hom{\mathsf{T}}{C}{F'} \xrightarrow{\Hom{\mathsf{T}}{C}{g}} \Hom{\mathsf{T}}{C}{D} \]
    As the first map is surjective, we get that the second map is zero, and hence $g$ is a phantom.
    
\end{proof}
Now we state and prove the the main results of this chapter.
\begin{theorem}\label{Main Theorem on Tc-}
    Let $\mathsf{T}$ be a $\mathcal{G}$-quasiapproximable triangulated category (\Cref{Definition of G-quasiapproximability}) for a finite generating sequence $\mathcal{G}$ (\Cref{Definition of generating sequence}). Further, assume that $\mathsf{T}$ is $R$-linear for a commutative noetherian ring $R$, and that $\Hom{\mathsf{T}}{G}{G'}$ is a finitely generated $R$-module for all $G,G'$ in $\bigcup_{n \in \mathbb{Z}} \mathcal{G}^n$. Then, the functor $\mathcal{Y} : \mathsf{T}^c \to $ $\operatorname{Hom}\big([\mathsf{T}^c]^{\operatorname{op}}, \operatorname{Mod}(R) \big)$ taking $F \in \mathsf{T}^c $ (\Cref{Definition closure of compacts}) to $\Hom{\mathsf{T}}{-}{F}|_{\mathsf{T}^c}$, satisfies the following : 
    \begin{enumerate}
        \item The essential image of $\mathcal{Y}$ is exactly the $\mathcal{G}$-semifinite $\mathsf{T}^c$-cohomological functors, see \Cref{Definition finite bounded cohomological functors}.
        \item $\mathcal{Y}$ is full.
    \end{enumerate}
\end{theorem}
\begin{proof}
    
    We already know that $\mathcal{Y}(F)$ is a $\mathcal{G}$-semifinite $\mathsf{T}^c$-cohomological functor for any $F \in \mathsf{T}^c$ by \Cref{Lemma 1 of Section 6.2}. By \Cref{Lemma on strong approximating sequences}(3), any $F \in \mathsf{T}^c$ admits a strong $\mathsf{T}^c$-approximating system (\Cref{Definition strong G-approximating system}), and hence by \cite[Lemma 5.8]{Neeman:2021b}, the functor $\mathcal{Y}$ is full on $\mathsf{T}^c$.

    Now, let $H$ be any $\mathcal{G}$-semifinite $\mathsf{T}^c$-cohomological functor. By \Cref{Proposition 3 of Section 6.2}, there exists $F \in \overline{\langle \mathcal{K} \rangle}_4$ such that $H \cong \mathcal{Y}(F)$. Recall that $\mathcal{K} = \bigoplus_{K \in \mathsf{T}^c} K$, where the coproduct is over the equivalence class of compact objects, and hence a set-indexed coproduct. It remains to show that $F \in \mathsf{T}^c$. 
    
    Denote the ideal of phantom maps by $\mathcal{I}$ (\Cref{Definition of a Phantom map}). We use induction to prove that for any $n \geq 1$, there exists a triangle $F_n \to F \xrightarrow{\beta_n} D_n \to \Sigma F_n$ with $F_n \in \mathsf{T}^c $ and $\beta_n \in \mathcal{I}^n$. $n = 1 $ is given by \Cref{Corollary 7.17}. So, assume now that we know the result up to a positive integer $n$, i.e. a triangle $F_n \to F \xrightarrow{\beta_n} D_n \to \Sigma F_n$ with $F_n \in \mathsf{T}^c$ and $\beta_n \in \mathcal{I}^n$. Consider the exact sequence,
    \[ \mathcal{Y}(\Sigma^{-1}F)\xrightarrow{\mathcal{Y}(\Sigma^{-1}\beta_n) \ = \ 0 \ } \mathcal{Y}(\Sigma^{-1}D_n)\xrightarrow{}\mathcal{Y}(F_n)\xrightarrow{}\mathcal{Y}(F)\xrightarrow{\mathcal{Y}(\beta_n) \ = \ 0 \  }\mathcal{Y}(D_n)\]
    As $\mathcal{Y}(F_n)$ and $\mathcal{Y}(F)$ are both $\mathcal{G}$-semifinite $\mathsf{T}^c$-cohomological functors, the above exact sequence tells us that so is $\mathcal{Y}(D_n)$. So, by \Cref{Corollary 7.17}, we have a triangle $F' \to D_n \xrightarrow{\gamma} D_{n + 1} \to \Sigma F'$ with $F' \in \mathsf{T}^c$ and $\gamma \in \mathcal{I}$. Let $\beta_{n + 1 } \colonequals \gamma \circ \beta_{n} $. Note that $\beta_{n + 1 } \in \mathcal{I}^{n + 1 } $. Consider the triangle $ F_{n + 1} \to F \xrightarrow{\beta_{n + 1}} D_{n + 1} \to \Sigma F_{n + 1}$ we get from $\beta_{n + 1}$. Applying the octahedral axiom to $ \gamma \circ \beta_{n } $, we get a triangle $F_n \to F_{n + 1 } \to F' \to \Sigma F_n $. By \Cref{Proposition Closure of compacts is triangulated} we get that $F_{n + 1} \in \mathsf{T}^c$ as $F_n, F' \in \mathsf{T}^c$, which completes the induction argument.

    Now, consider the triangle $F_4 \to F \to D_4 \to \Sigma F_4$. Note that $F \to D_4 \in \mathcal{I}^4 $ and $F \in \overline{\langle \mathcal{K} \rangle}_4 $. $(\overline{\langle \mathcal{K} \rangle}_1, \mathcal{I})$ forms a projective class, see \Cref{Definition of a projective class}. And so, by \cite[Theorem 1.1]{Christensen:1998}, so does $(\overline{\langle \mathcal{K} \rangle}_4, \mathcal{I}^4 ) $. This gives us that $F \to D_4 $ vanishes, and hence $F$ is a direct summand of $F_4 \in \mathsf{T}^c$. As $ \mathsf{T}^c $ is thick by \Cref{Proposition Closure of compacts is thick}, we get $F \in \mathsf{T}^c$, and we are done.

\end{proof}

\begin{lemma}\label{Lemma 6 of Section 6.2}
    With the same assumptions as \Cref{Main Theorem on Tc-}, let $ f : F \to F' $ be a morphism in $\mathsf{T}^c$, with $F' \in T^b_c$, see \Cref{Definition closure of compacts}. Then $\mathcal{Y}(f) = 0 $ implies $f = 0$.
\end{lemma}
\begin{proof}
    As $F' \in \mathsf{T}^b_c$, there exists an integer $n$ such that $F' \in (\mathcal{R}^{\mathcal{G}}_n)^{\perp} = (\mathcal{M}^{\mathcal{G}}_n)^{\perp}$, see \Cref{Definition of generating sequence}. As $F \in \mathsf{T}^c $, we have a triangle $E \xrightarrow{g} F \xrightarrow{h} D \to  \Sigma E$ with $E \in \mathsf{T}^c $ and $D \in \mathcal{R}^{\mathcal{G}}_{n}$. As $\mathcal{Y}(f) = 0 $, we get that $fg = 0 $, which tells us that $f : F \to F' $ factors as $F \xrightarrow{h} D \to F' $. As $D \in \mathcal{R}^{\mathcal{G}}_n $ and $F' \in (\mathcal{R}^{\mathcal{G}}_n)^{\perp} $, $(D \to F' ) = 0 $, and hence so is $f$, and we are done.
    
\end{proof}

\begin{theorem}\label{Main Theorem on Tbc}
    Let $\mathsf{T}$ be a $\mathcal{G}$-quasiapproximable triangulated category (\Cref{Definition of G-quasiapproximability}) for a finite generating sequence $\mathcal{G}$ (\Cref{Definition of generating sequence}). Further, assume that $\mathsf{T}$ is $R$-linear for a commutative noetherian ring $R$, and that $\Hom{\mathsf{T}}{G}{G'}$ is a finitely generated $R$-module for all $G,G'$ in $\bigcup_{n \in \mathbb{Z}} \mathcal{G}^n$. Then, the functor $\mathcal{Y}|_{\mathsf{T}^b_c} : \mathsf{T}^b_c \to $ $\operatorname{Hom}\big([\mathsf{T}^c]^{\operatorname{op}}, \operatorname{Mod}(R) \big)$ taking $F \in \mathsf{T}^b_c $ (\Cref{Definition closure of compacts}) to $\Hom{\mathsf{T}}{-}{F}|_{\mathsf{T}^c}$, satisfies the following : 
    \begin{enumerate}
        \item The essential image of $\mathcal{Y}|_{\mathsf{T}^b_c}$ is exactly the $\mathcal{G}$-finite $\mathsf{T}^c$-cohomological functors, see \Cref{Definition finite bounded cohomological functors}.
        \item $\mathcal{Y}|_{\mathsf{T}^b_c}$ is fully faithful.
    \end{enumerate}
\end{theorem}
\begin{proof}

    It just remains to identify the essential image of $\mathsf{T}^b_c $ under $\mathcal{Y}$, as \Cref{Main Theorem on Tc-} and \Cref{Lemma 6 of Section 6.2} show that $\mathcal{Y}|_{\mathsf{T}^b_c}$ is fully faithful.

    Let $F \in \mathsf{T}^b_c \subseteq\mathsf{T}^c$, see \Cref{Definition closure of compacts}. By \Cref{Lemma Closure of compacts are finite right-bounded}, we already know that $\mathcal{Y}(F)$ is a $\mathcal{G}$-semifinite $\mathsf{T}^c$-cohomological functor (\Cref{Definition finite bounded cohomological functors}). By \Cref{Definition closure of compacts}, $\mathsf{T}^b_c = \mathsf{T}^c \cap \big(\mathcal{M}^{\mathcal{G}}\big)^{\perp}$ see \Cref{Remark on choice of metric for computing Tc and Tbc}. And so, $\HomT{\mathcal{M}^{\mathcal{G}}_i}{F}= 0$ for $i > > 0$. Therefore, $\HomT{\mathcal{G}^{-i}}{F} = 0$ for $i >> 0$ as $\mathcal{G}^i \subseteq \mathcal{M}^{\mathcal{G}}_i$ for all $i \in \mathbb{Z}$, see \Cref{Definition closure of compacts}. This gives us that $\mathcal{Y}(F)$ is a bounded $\mathsf{T}^c$-cohomological functor for any $F \in \mathsf{T}^b_c $. 
    
    Conversely, let $H$ be a $\mathcal{G}$-finite $\mathsf{T}^c$-cohomological functor. As this implies that $H$ is a $\mathcal{G}$-semifinite $\mathsf{T}^c$-cohomological functor, by \Cref{Main Theorem on Tc-}(1) there exists $F \in \mathsf{T}^c$ such that $H \cong \mathcal{Y}(F)$. As $H$ is a $\mathcal{G}$-finite $\mathsf{T}^c$-cohomological functor, there exists a positive integer $n$ such that $H(\mathcal{G}^{-i}) = 0$  for all $i \geq n $. And so, as $H \cong \mathcal{Y}(F) $, we get that $\Hom{\mathsf{T}}{\mathcal{G}^{-i}}{F} = 0 $ for all $i \geq n $. So, $F \in (\mathcal{M}^{\mathcal{G}}_n)^{\perp} $, see \Cref{Definition closure of compacts}. As $F \in \mathsf{T}^c$ already, we have shown that $F \in \mathsf{T}^b_c = \mathsf{T}^c \cap \big(\mathcal{M}^{\mathcal{G}}\big)^{\perp} $, and we are done.
\end{proof}

\section{Examples of approximability}
In this section we discuss some examples of (weakly) $\mathcal{G}$-approximable and (weakly) $\mathcal{G}$-quasiapproximable triangulated categories, see \Cref{Definition of G-approximability} and \Cref{Definition of G-quasiapproximability}. We first focus on an abstract class of examples, and later, we give more concrete examples.

We have already seen in \Cref{Theorem Approximability and G-approximability} that approximable triangulated categories are examples of $\mathcal{G}$-approximable triangulated categories. We begin by introducing another class of examples of $\mathcal{G}$-approximable triangulated categories, namely co-approximable triangulated categories. In this setting, co-t-structures will play the role t-structures do for approximability.

We begin with some definitions.
\begin{definition}
    Two co-t-structures $(\mathsf{U}_1,\mathsf{V}_1)$ and $(\mathsf{U}_2,\mathsf{V}_2)$ on a triangulated category $\mathsf{T}$ are said to be equivalent if there exists a positive integer $A$ such that,
    \[ \Sigma^{- A}\mathsf{U}_1 \subseteq \mathsf{U}_2 \subseteq \Sigma^{A}\mathsf{U}_1 \]
    Equivalently, by taking orthogonals and shifting, 
    \[ \Sigma^{ A}\mathsf{V}_1 \subseteq \mathsf{V}_2 \subseteq \Sigma^{- A}\mathsf{V}_1 \]
    It is easy to verify that this indeed gives an equivalence relation on the collection of co-t-structures on $\mathsf{T}$.
\end{definition}
Recall that we can construct a co-t-structure on a triangulated category with coproducts from a given set of compact objects, see \Cref{Theorem Compactly generated co-t-structures}. We will in particular be interested in co-t-structures generated by a single compact object. 
\begin{definition}\label{Definition of compactly generated co-t-structure}
    Let $\mathsf{T}$ be a triangulated category with all small coproducts. Then for any $H \in \mathsf{T}^c$, we define $\mathsf{V}_H \colonequals \{\Sigma^{-n}H : n \geq 1\}^{\perp}$. Then, the \emph{co-t-structure generated by $H$} is defined to be $(\mathsf{U}_H, \mathsf{V}_H)$, where $\mathsf{U}_H \colonequals \ ^{\perp}(\Sigma\mathsf{V}_H)$. It is a co-t-structure by \Cref{Theorem Compactly generated co-t-structures}.
\end{definition}

\begin{remark}\label{Remark 1 on co-approx and G-approx}
    Let $\mathsf{T}$ be a triangulated category with a single compact generator $G$. We observe that we are in the setting of \Cref{Section Generating sequence and orthogonal metrics} by considering the finite generating sequence (\Cref{Definition of generating sequence}) $\mathcal{G}$ given by $\mathcal{G}^n \colonequals \{\Sigma^n G\}$. The metric $\mathcal{R}^{\mathcal{G}}$, see \Cref{Definition of generating sequence}, is then given by $\mathcal{R}^{\mathcal{G}}_n = \Sigma^{-n}\mathsf{U}_G$ where $\mathsf{U}_G$ is as defined in \Cref{Definition of compactly generated co-t-structure}.
\end{remark}
\begin{lemma}\label{Lemma eq. class co-t-structures}
    Let $G_1, G_2 \in \mathsf{T}$ be two compact generators of a triangulated category $\mathsf{T}$. Then, the co-t-structures $(\mathsf{U}_{G_1}, \mathsf{V}_{G_1})$ and $(\mathsf{U}_{G_2}, \mathsf{V}_{G_2})$ these generate are equivalent.
\end{lemma}
\begin{proof}
    Follows from \Cref{Remark 1 on co-approx and G-approx}, \Cref{Lemma eq. class}(2), and \Cref{Examples of N-equivalent generating sequences}.
\end{proof}

This motivates the following definition.

\begin{definition}\label{Definition preferred equivalence of co-t-structures}
    Let $\mathsf{T}$ be a compactly generated triangulated category with a single compact generator $G$. Then, we define the \emph{preferred equivalence class of co-t-structures on $\mathsf{T}$} to be all co-t-structures which are equivalent to $(\mathsf{U}_G, \mathsf{V}_G)$, see \Cref{Definition of compactly generated co-t-structure}. Note that by \Cref{Lemma eq. class co-t-structures}, this definition is independent of the choice of the compact generator $G \in \mathsf{T}$.
\end{definition}

\begin{definition}\label{Definition T^+_c and T^-}
Let $(\mathsf{U},\mathsf{V})$ be a co-t-structure on $\mathsf{T}$. We define the following full subcategories,
\begin{equation*}
\mathsf{T}^- = \bigcup_{n \in \mathbb{Z}}\Sigma^n\mathsf{V} \ , \  \mathsf{T}_c^+ \colonequals  \bigcap_{n \in \mathbb{Z}} (\mathsf{T}^c \star \Sigma^{-n} \mathsf{U})
\end{equation*}
It is clear that these categories remain the same if replace the co-t-structure $(\mathsf{U},\mathsf{V})$ by an equivalent one. Hence, if $\mathsf{T}$ has a single compact generator, $\mathsf{T}^-$ and $\mathsf{T}_c^+$ defined by a co-t-structure in the preferred equivalence class (\Cref{Definition preferred equivalence of co-t-structures}) are intrinsic subcategories of $\mathsf{T}$.
\end{definition}

The following remark is just an exercise in unwinding definitions.

\begin{remark}\label{Remark preferred equivalence class co-t-structure and metric}
    Let $\mathsf{T}$ be a triangulated category with a single compact generator $G$. We define the finite generating sequence (\Cref{Definition of generating sequence}) $\mathcal{G}$ by $\mathcal{G}^n \colonequals \{\Sigma^{n}G\}$ for all $n \in \mathbb{Z}$. Then, there is a bijection between the preferred equivalence class of co-t-structures on $\mathsf{T}$ and those metrics $\mathcal{M}$ in the preferred $\mathbb{N}$-equivalence class of orthogonal metrics, such that $\mathcal{M}_n = \Sigma^{n}\mathsf{U}$ for all $n \in \mathbb{Z}$ for a co-t-structure $(\mathsf{U},\mathsf{V})$ on $\mathsf{T}$. 
\end{remark}

Now, we come to the definition of co-approximability.

\begin{definition} \label{Definition of co-approx}
We say a triangulated category $\mathsf{T}$ is \emph{weakly co-approximable} if there exists a co-t-structure $(\mathsf{U},\mathsf{V})$, a compact generator $G$, and an integer $A > 0$ such that,
    \begin{enumerate}
        \item $\Sigma^{-A} G \in \mathsf{U}$, 
        \item \label{co-approx:Hom} $\Hom{\mathsf{T}}{\Sigma^A G}{\mathsf{U}} = 0$, and
        \item \label{co-approx:tri} Every object $F \in \mathsf{U}$ admits a triangle $E \to F \to D \to \Sigma E$ such that $ D \in \Sigma^{-1} \mathsf{U}$  $  \text{ and }E \in \overline{\langle G \rangle}^{[-A,A]}$, see \Cref{Notation from Neeman's paper}.
    \end{enumerate}
We say $\mathsf{T}$ is \emph{co-approximable} if further,
\begin{enumerate}[label= (\arabic*)$'$]
    \setcounter{enumi}{2}
    \item Every object $E \in \mathsf{U}$ admits a triangle
    $E \to F \to D \to \Sigma E$ such that $D \in \Sigma^{-1} \mathsf{U}$ and $E \in \overline{\langle G \rangle}^{[-A,A]}_A$, see \Cref{Notation from Neeman's paper}.
\end{enumerate}

\end{definition}

\begin{definition}\label{Definition of co-quasiapprox}
We say a triangulated category $\mathsf{T}$ is \emph{weakly co-quasiapproximable} if there exists a co-t-structure $(\mathsf{U},\mathsf{V})$, a compact generator $G$, and an integer $A > 0$ such that,
    \begin{enumerate}
        \item $\Sigma^{-A} G \in \mathsf{U}$, 
        \item $\Hom{\mathsf{T}}{\Sigma^A G}{\mathsf{U}} = 0$, and
        \item Every object $F \in \mathsf{U} \cap \mathsf{T}^+_c$ admits a triangle $E \to F \to D \to \Sigma E$ such that $ D \in \Sigma^{-1} \mathsf{U} \cap \mathsf{T}^+_c$  $  \text{ and }E \in \overline{\langle G \rangle}^{[-A,A]}$, see \Cref{Notation from Neeman's paper}.
    \end{enumerate}
We say $\mathsf{T}$ is \emph{co-quasiapproximable} if further,
\begin{enumerate}[label= (\arabic*)$'$]
    \setcounter{enumi}{2}
    \item Every object $E \in \mathsf{U} \cap \mathsf{T}^+_c $ admits a triangle
    $E \to F \to D \to \Sigma E$ such that $D \in \Sigma^{-1} \mathsf{U} \cap \mathsf{T}^+_c$ and $E \in \overline{\langle G \rangle}^{[-A,A]}_A$, see \Cref{Notation from Neeman's paper}.
\end{enumerate}    
    
\end{definition}

It is easy to verify that if $\mathsf{T}$ is (weakly) co-approx (resp.\ (weakly) co-quasiapprox) with respect to one compact generator, then it is (weakly) co-approx (resp.\ (weakly) co-quasiapprox) with respect to any other compact generator up to increasing the integer $A$ if necessary.

The following result shows that co-approximability is a special case of $\mathcal{G}$-approximability.

\begin{theorem}\label{Theorem Co-Approximability and G-Approximability}
    Let $\mathsf{T}$ be a triangulated category with a single compact generator $G$. We define the finite generating sequence (\Cref{Definition of generating sequence}) $\mathcal{G}$ by $\mathcal{G}^n \colonequals \{\Sigma^{n}G\} $. Then,  \begin{enumerate}[label=(\roman*)]
        \item $\mathsf{T}$ is a weakly co-approximable triangulated category (\Cref{Definition of co-approx}) if and only if $\mathsf{T}$ is weakly $\mathcal{G}$-approximable (\Cref{Definition of G-approximability}).
        \item $\mathsf{T}$ is a co-approximable triangulated category (\Cref{Definition of co-approx}) if and only if $\mathsf{T}$ is $\mathcal{G}$-approximable (\Cref{Definition of G-approximability}).
    \end{enumerate}
    The analogous result also holds for (weak) co-quasiapprox and (weak) $\mathcal{G}$-quasiapprox.
\end{theorem}
\begin{proof}
    Both (i) and (ii) follow the same way, so we only prove (i). Suppose $\mathsf{T}$ is weakly $\mathcal{G}$-approximable. Then, by \Cref{Remark 1 on co-approx and G-approx} we get that that the co-t-structure compactly generated by $G$ gives $\mathcal{R}^{\mathcal{G}}$ under the bijection given in \Cref{Remark preferred equivalence class co-t-structure and metric}, which is a metric in the preferred $\mathbb{N}$-equivalence class of orthogonal metric. By \Cref{Proposition G-approximability from orthogonal metric in preferred equivalence class}, $\mathsf{T}$ is weakly $\mathcal{G}$-approximable with respect to this orthogonal metric. It is easy to see that this gives us that $\mathsf{T}$ is weakly co-approximable using the co-t-structure $(\mathsf{U}_G,\mathsf{V}_G)$, see \Cref{Definition of compactly generated co-t-structure}. \Cref{Definition of G-approximability}(1), (2), and (3) imply \Cref{Definition of co-approx}(1), (2), and (3) respectively.

    Conversely, suppose $\mathsf{T}$ is weakly co-approximable. So, there exists a co-t-structure $(\mathsf{U},\mathsf{V})$, and an integer $A > 0$ such that the conditions (1)-(3) of \Cref{Definition of co-approx} are satisfied with the compact generator $G$. We will show that $\mathsf{T}$ is weakly $\mathcal{G}$-approximable by showing the conditions (0)-(3) of \Cref{Remark conditions of G-approximability} hold with the orthogonal good metric $\mathcal{R}_i = \Sigma^{-i} \mathsf{U}$ for all $i \in \mathbb{Z}$ and the integer $A > 0$.
    \begin{enumerate}
    \setcounter{enumi}{-1}
        \item As $\mathcal{G}^n = \{\Sigma^n G\}$, we get that $\mathcal{M}^{\mathcal{G}}_i = \Sigma \mathcal{M}^{\mathcal{G}}_{i+1}$ for all $i \in \mathbb{Z}$, as $\mathcal{M}^{\mathcal{G}}_i = \mathcal{G}^{(-\infty,-i]} = \langle G \rangle^{[i,\infty)}$ by definition. See \Cref{Definition of generating sequence}(2) for the definition of $\mathcal{M}^{\mathcal{G}}$. Also see, \Cref{Notation Subcategories for (pre)-generating sequences} and \Cref{Notation from Neeman's paper}.
        \item As $\Sigma^{-A} G \in \mathsf{U} = \mathcal{R}_0$, $\mathcal{M}^G_{i+A} = \langle G \rangle^{[i+A,\infty)}\subseteq \mathcal{R}_i \cap \mathsf{T}^c$ for all $i \in \mathbb{Z}$, see \Cref{Notation from Neeman's paper}.
        \item As $\HomT{\Sigma^n G}{\mathsf{U}} = 0$ for all $n \geq A$, we get that,
        \[\HomT{\mathcal{G}_n}{\mathcal{R}_{i}} = \HomT{\Sigma^{n} G}{\Sigma^{-i} \mathsf{U}}=0\] for all $n + i \geq A$.
        \item For any $i \in \mathbb{Z}$ and $F \in \mathcal{R}_{i} = \Sigma^{-i}\mathsf{U}$, we get that $\Sigma^{i}F \in \mathsf{U}$. So, by the \Cref{Definition of co-approx}(3), we get a triangle $E \to F \to D \to \Sigma E$ with $E \in \langle G \rangle^{[-A+i,A+i]}$ and $D \in \Sigma^{-i-1} \mathsf{U} = \mathcal{R}_{i+1}$, which is what we needed.
    \end{enumerate}
\end{proof}
\begin{proposition}\label{Proposition co-approximability implies preferred equivalence class}
    Let $\mathsf{T}$ be a weakly co-approximable triangulated category, with a compact generator $G$, co-t-structure $(\mathsf{U}, \mathsf{V} )$ and positive integer $A$ satisfying the conditions (1)-(3) of \Cref{Definition of co-approx}. Then, $(\mathsf{U}, \mathsf{V})$ is in the preferred equivalence class of co-t-structures on $\mathsf{T}$. Conversely, for any co-t-structure $(\mathsf{U}',\mathsf{V}')$ in the preferred equivalence class of co-t-structures (\Cref{Definition preferred equivalence of co-t-structures}), there exists an integer $A'$ such that the conditions of weak co-approximability are satisfied.
\end{proposition}
\begin{proof}
    Follows from \Cref{Theorem Co-Approximability and G-Approximability}, \Cref{Proposition approximable implies preferred equivalence class} and \Cref{Remark preferred equivalence class co-t-structure and metric}. The converse follows from \Cref{Proposition G-approximability from orthogonal metric in preferred equivalence class}.
\end{proof}

\begin{proposition}\label{Proposition Weak co-approx implies aisle in G^{-A-a, infty}}
    Let $\mathsf{T}$ be a weakly co-approximable triangulated category, with compact generator $G$, co-t-structure $(\mathsf{U}, \mathsf{V} )$ and positive integer $A$ satisfying the conditions (1)-(3) of \Cref{Definition of co-approx}. Then, $\mathsf{U} \subseteq \overline{\langle G \rangle}^{[-A-1, \infty)}$, see \Cref{Notation from Neeman's paper}. In particular, $\mathsf{U}_{G} \colonequals {}^{\perp}(G[0,\infty)^{\perp}) \subseteq \overline{\langle G \rangle}^{[-A-1, \infty)}$, see \Cref{Definition of compactly generated co-t-structure} and \Cref{Proposition co-approximability implies preferred equivalence class}.
\end{proposition}
\begin{proof}
    Follows from \Cref{Theorem Co-Approximability and G-Approximability} and \Cref{Corollary approximating by generating sequence}.
\end{proof}

We now see what the closure of the compacts and the orthogonal to $\mathcal{G}$ are in this setting, see \Cref{Definition closure of compacts}.

The following is again just an exercise in unwinding definitions.
\begin{remark}\label{Remark 2 on co-approx and G-approx }
    Let $\mathsf{T}$ be a triangulated category with a single compact generator $G$. We define the finite generating sequence $\mathcal{G}$ by $\mathcal{G}^n \colonequals\{\Sigma^{n}G\}$. Then, $\overline{\mathsf{T}^c} = \mathsf{T}^+_c$ and $\mathcal{G}^{\perp} = \mathsf{T}^-$, see \Cref{Definition T^+_c and T^-}, \Cref{Remark 1 on co-approx and G-approx}, and \Cref{Definition closure of compacts}.
\end{remark}

We note the following results about the closure of the compacts in this setting.

\begin{proposition}
    Let $\mathsf{T}$ a compactly generated triangulated category with a co-t-structure $(\mathsf{U},\mathsf{V})$. Assume there exists a compact generator $G$ and a positive integer $A$ such that $\Hom{\mathsf{T}}{\Sigma^A G}{\mathsf{U}} = 0$. Then $\mathsf{T}_c^+$ (\Cref{Definition T^+_c and T^-}) is triangulated.
\end{proposition}
\begin{proof}
    Follows from \Cref{Remark 1 on co-approx and G-approx} and \Cref{Proposition Closure of compacts is triangulated}
\end{proof}

\begin{proposition}
    Let $\mathsf{T}$ a compactly generated triangulated category with a co-t-structure $(\mathsf{U},\mathsf{V})$. Assume there exists a compact generator $G$ and a positive integer $A$ such that $\Hom{\mathsf{T}}{\Sigma^A G}{\mathsf{U}} = 0$ and $\Sigma^{-A} G \subseteq \mathsf{U}$. Then $\mathsf{T}_c^+$ (\Cref{Definition T^+_c and T^-}) is thick.
\end{proposition}
\begin{proof}
    Follows from \Cref{Remark 1 on co-approx and G-approx} and \Cref{Proposition Closure of compacts is thick}.
\end{proof}

\begin{corollary}
    Let $\mathsf{T}$ be a weakly co-approximable triangulated category, with compact generator $G$, co-t-structure $(\mathsf{U}, \mathsf{V} )$ and positive integer $A$ satisfying \Cref{Definition of co-approx}(1)-(3). Then, for any object $F \in \mathsf{T}^+_c \cap \mathsf{U}$ (\Cref{Definition T^+_c and T^-}), there exists a sequence $\{E_i \to E_{i+1}\}_{i \geq 1}$ mapping to $F$ such that $E_i \in \langle G \rangle^{[-B, B + i - 1]}$ and $\operatorname{Cone}(E_i \to F) \in \Sigma^{- i}\mathsf{U}$ for all $i \geq 1$. If $\mathsf{T}$ is co-approximable, we can further arrange the sequence $E_*$ so that $E_i \in \langle G \rangle^{[-B, B + i - 1]}_{iB}$, see \Cref{Notation from Neeman's paper}.
    
    The analogous result also holds for (weak) quasiapprox. In all of these cases, the non-canonical map $\hocolim{E}_i \to F$ is an isomorphism.
\end{corollary}
\begin{proof}
    Follows from \Cref{Remark 1 on co-approx and G-approx}, \Cref{Theorem Co-Approximability and G-Approximability}, and \Cref{Corollary 4 of section 5}.
\end{proof}

\subsection{Examples coming from algebraic geometry}

Now, we give a class of examples of co-approximable triangulated categories, see \Cref{Definition of co-approx}. We begin with some conventions.
\begin{notation}
    Let $X$ be a noetherian scheme. Then, 
    \begin{enumerate}
        \item $(\mathbf{D}({\operatorname{Qcoh}}(X))^{\leq 0},\mathbf{D}({\operatorname{Qcoh}}(X))^{\geq 0})$ denotes the standard t-structure on $\mathbf{D}({\operatorname{Qcoh}}(X))$, where the truncation triangles are given by the soft or the canonical truncation of complexes.
        \item $(\mathbf{K}(\operatorname{Inj}\text{-}X)^{\geq 0},\mathbf{K}(\operatorname{Inj}\text{-}X)^{\leq 0})$ denotes the standard co-t-structure on $\mathbf{K}(\operatorname{Inj}\text{-}X)$, where the truncation triangles are given by the hard or the brutal truncation of complexes.
        
    \end{enumerate}
\end{notation}

\begin{lemma}\label{Lemma coaisle of derived category and "aisle" of homotopy category}
    Let $X$ be a noetherian scheme and $p : \mathbf{K}(\operatorname{Inj}\text{-}X) \to \mathbf{D}({\operatorname{Qcoh}}(X))$ be the Verdier quotient map, see \cite[Proposition 3.6]{Krause:2005}. Then, this restricts to an equivalence $ p : \mathbf{K}(\operatorname{Inj}\text{-}X)^{\geq 0} \to \mathbf{D}({\operatorname{Qcoh}}(X))^{\geq 0} $.
\end{lemma}
\begin{proof}
    This follows trivially as $\Hom{\mathbf{K}(\operatorname{Inj}\text{-}X)}{E}{F} \cong \Hom{\mathbf{D}({\operatorname{Qcoh}}(X))}{E}{F}$ if $F$ is a bounded below complex of injective quasicoherent sheaves, and $E$ is an arbitrary complex of injective quasicoherent sheaves, see \cite[Tag 05TG]{StacksProject}.
\end{proof}

\begin{theorem}\label{Theorem Kinj co-approx}
    Let $X$ be a noetherian scheme. Assume there exists $G \in \mathbf{D^b_{\operatorname{coh}}}(X)$ such that $\mathbf{D}(\operatorname{Qcoh}(X)) = \overline{\langle G \rangle}_N$ for some positive integer $N$, see \Cref{Notation from Neeman's paper}. Then $\mathbf{K}(\operatorname{Inj}\text{-}X)$ is co-approximable, see \Cref{Definition of co-approx}.    
\end{theorem}
\begin{proof}
    We work with the standard co-t-structure on $\mathbf{K}(\operatorname{Inj}\text{-}X)$ again. The only non-trivial axiom to verify is the existence of approximating triangles (\Cref{Definition of co-approx}(3$'$)). So, let $F \in \mathbf{K}(\operatorname{Inj}\text{-}X)^{\geq 0} \cong \mathbf{D}(\operatorname{Qcoh}(X))^{\geq 0}$, see \Cref{Lemma coaisle of derived category and "aisle" of homotopy category}. We now work in the category $\mathbf{D}(\operatorname{Qcoh}(X))$. Consider the truncation triangle $F^{\leq 0} \to F \to F^{\geq 1} \to \Sigma F^{\leq 0}$ with respect to the standard t-structure on $\mathbf{D}(\operatorname{Qcoh}(X))$. We have that $F^{\leq 0} \in \mathbf{D}(\operatorname{Qcoh}(X))^{\leq 0}\cap \mathbf{D}(\operatorname{Qcoh}(X))^{\geq 0}$. So, by \cite[Lemma 2.4]{Neeman:2018b}, there exists a positive integer $A$, independent of the complex $F \in \mathbf{D}(\operatorname{Qcoh}(X))^{\geq 0}$, such that $F^{\leq 0} \in \overline{\langle G \rangle}^{[-A, A]}_{A}$. As $F^{\geq 1} \in \mathbf{D}(\operatorname{Qcoh}(X))^{\geq 1} \cong \mathbf{K}(\operatorname{Inj}\text{-}X)^{\geq 1}$, the truncation triangle is the one we were after to prove co-approximability. Note that $G$ is a classical generator for $\mathbf{D^b_{\operatorname{coh}}}(X)$ by \cite[Lemma 2.7]{Neeman:2021a}, and hence is a compact generator for $\mathbf{K}(\operatorname{Inj}\text{-}X)$.
\end{proof}

\begin{remark}\label{Remark Kinj standard co-t-structure in preferred equiv class}
    Let $X$ be a noetherian scheme. Assume there exists $G \in \mathbf{D^b_{\operatorname{coh}}}(X)$ such that $\mathbf{D}(\operatorname{Qcoh}(X)) = \overline{\langle G \rangle}_N$ for some positive integer $N$. By the proof of \Cref{Theorem Kinj co-approx}, \Cref{Proposition approximable implies preferred equivalence class}, and \Cref{Theorem Co-Approximability and G-Approximability}, we get that the standard co-t-structures on $\mathbf{K}(\operatorname{Inj}\text{-}X)$ lies in the preferred equivalence class of co-t-structures, see \Cref{Definition preferred equivalence of co-t-structures}. This in turn implies that the metric given by $\{\mathbf{D}^{b}(\operatorname{coh}(X))^{\geq n}\}_{n \in \mathbb{Z}}$ is equivalent to the metric given by $\{\langle G \rangle ^{[n,\infty)}\}_{n \in \mathbb{Z}}$. This also follows immediately from \cite[Lemma 2.3]{Neeman:2018b}
\end{remark}

\begin{remark}\label{Remark Kinj co-approx for stacks and algebras}
    We note here that the arguments in \Cref{Theorem Kinj co-approx} and \Cref{Remark Kinj standard co-t-structure in preferred equiv class} hold in the more general setting of the homotopy category of injective objects of a locally noetherian Grothendieck abelian category using \cite{Krause:2005}. We give two important cases where this is applicable.

    \begin{enumerate}
        \item Let $X$ be a noetherian scheme, and let $\mathcal{A}$ be a coherent algebra over $X$. Assume there exists $G \in \mathbf{D^b_{\operatorname{coh}}}(\mathcal{A})$ such that $\mathbf{D}(\operatorname{Qcoh}(\mathcal{A})) = \overline{\langle G \rangle}_N$ for some positive integer $N$, see \Cref{Notation from Neeman's paper}. Then $\mathbf{K}(\operatorname{Inj}\text{-}\mathcal{A})$ is co-approximable and the standard co-t-structure lies in the preferred equivalence class.
        \item Let $\mathcal{X}$ be a noetherian stack. Again, if there exists an object $G \in \mathbf{D^b_{\operatorname{coh}}}(\mathcal{X})$ such that $\mathbf{D}(\operatorname{Qcoh}(\mathcal{X})) = \overline{\langle G \rangle}_N$ for some positive integer $N$, then $\mathbf{K}(\operatorname{Inj}\text{-}\mathcal{X})$ is co-approximable and the standard co-t-structure lies in the preferred equivalence class. 
    \end{enumerate}

\end{remark}

We now give a class of example of weak co-quasiapproximable triangulated categories. The example in question will be of the mock homotopy category of projectives. Throughout our discussion, we will try to stick to the notation in \cite{Murfet:2008}.

We begin by setting up some conventions. 
\begin{convention}\label{Convention Mock homotopy category of projectives}
    Let $X$ be a noetherian separated scheme. We define a compactly generated co-t-structure $(\mathsf{U},\mathsf{V})$ on the mock homotopy category of projectives $\mathbf{K}_{m}(\operatorname{Proj}\text{-}X)$ defined by the set of compact objects $U_{\lambda}(\operatorname{coh}(X))^{\circ}$, where $U_{\lambda}(-)^{\circ}: \mathbf{D}^b(\operatorname{coh}(X))^{\operatorname{op}} \xrightarrow{\sim} \mathbf{K}_m(\operatorname{Proj}(X))^c $, see \cite[Theorem 7.4]{Murfet:2008}. This immediately gives us an orthogonal metric $\mathcal{R}$ given by $\mathcal{R}_n = \Sigma^{-n} \mathsf{U}$. By \cite[Theorem 2.3.4]{Bondarko:2022}, the metric $\mathcal{R} \cap \mathsf{T}^c$ is $\mathbb{N}$-equivalent to the metric $\mathcal{M}$ on $\mathsf{T}^c$ given by $\mathcal{M}_n = U_\lambda(D^{b}(\operatorname{coh}(X)^{\leq -n})$.
\end{convention}

\begin{lemma}\label{Closure of compacts for the mock homotopy category of projectives}
    Let $X$ be a noetherian, separated scheme. Then, the closure of the compacts $\mathbf{K}_m(\operatorname{Proj}\text{-}X)^+_c = U_\lambda(\mathbf{D}^{-}(\operatorname{coh}(X))^{\circ}$, where we take the closure of the compacts with respect to the metric $\mathcal{R}$ of \Cref{Convention Mock homotopy category of projectives}.
\end{lemma}
\begin{proof}
    One inclusion is easier to show. For any $F \in D^{-}(\operatorname{coh}(X))$ and any positive integer $n$, we have a triangle $F^{\leq -{n+1}} \to F \to F^{\geq n} \to \Sigma F^{\leq -{n+1}}$ given by the canonical truncation. Applying $U_\lambda(-)^{\circ}$ to these triangles gives us that $U_\lambda(F)^{\circ} \in \mathbf{K}_m(\operatorname{Proj}\text{-}X)^+_c$.
    
    For the other inclusion, first note that if we define the generating sequence $\mathcal{G}$ by $\mathcal{G}_n = U_\lambda(\Sigma^{n}\operatorname{coh}(X))$, then we are in the setting of \Cref{Lemma on strong approximating sequences}. Let $F \in \mathbf{K}_m(\operatorname{Proj}\text{-}X)^+_c$. Then, there exists a strong $\mathbf{K}_m(\operatorname{Proj}\text{-}X)^c$-approximating sequence $\{F_n \to F_{n+1}\}$ such that $\hocolim F_n \cong F$. By the description of compact objects, we have that for all $n$, there exists $E_n \in  \mathbf{D}^b(\operatorname{coh}(X))$ with $F_n = U_\lambda(E_n)^{\circ}$. This gives us a triangle,
    \begin{displaymath}
        \bigoplus U_\lambda(E_n)^{\circ} \to \bigoplus U_\lambda(E_n)^{\circ} \to F \to \bigoplus \Sigma U_\lambda(E_n)^{\circ}
    \end{displaymath}
    By working locally, and comparing the triangles we get by applying $\operatorname{\mathcal{H}\textit{om}}(-,\mathcal{O}_X)$ and $U(-^{\circ})$ to the above triangle gives us that $F$ is locally a complex of vector bundles (that is, affine locally it is isomorphic to a complex of finitely generated projective modules in $\mathbf{K}(\operatorname{Proj}\text{-}R)$). From this, it is easy to show that $U(F^{\circ}) = E$ for some $E \in \mathbf{D}^{-}(\operatorname{coh}(X))$, which is what we needed to show.
\end{proof}

\begin{theorem}\label{Theorem Kmproj co-approx}
    Let $X$ be a noetherian scheme. Assume there exists $G \in \mathbf{D^b_{\operatorname{coh}}}(X)$ such that $\mathbf{D}(\operatorname{Qcoh}(X)) = \overline{\langle G \rangle}_N$ for some positive integer $N$, see \Cref{Notation from Neeman's paper}. Then $\mathbf{K}_{m}(\operatorname{Proj}\text{-}X)$ is co-quasiapprox, see \Cref{Definition of co-quasiapprox}.    
\end{theorem}
\begin{proof}
    We first note that $G$ is a classical generator for $\mathbf{D^b_{\operatorname{coh}}}(X)$ by \cite[Lemma 2.7]{Neeman:2021a}, and hence, $U_\lambda(G)^{\circ}$ is a compact generator for $\mathbf{K}_{m}(\operatorname{Proj}\text{-}X)$. The first two conditions of \Cref{Definition of co-quasiapprox} are easy enough to verify with the co-t-structure $(\mathsf{U},\mathsf{V})$ as in \Cref{Convention Mock homotopy category of projectives}. So, it remains to verify the final condition.

    Let $F \in \mathsf{U}$. Then, by \Cref{Closure of compacts for the mock homotopy category of projectives}, $F = U_\lambda(\widetilde{F})^{\circ}$ for some $\widetilde{F} \in \mathbf{D}^{-}(\operatorname{coh}(X))^{\leq 1}$. We have a triangle $\widetilde{F}^{\leq 2} \to \widetilde{F} \to \widetilde{F}^{\geq 1} \to \Sigma \widetilde{F}^{\leq 2}$ via the canonical truncation. Note that,
    \begin{displaymath}
        \widetilde{F}^{\geq -1} \in \mathbf{D}^b(\operatorname{coh}(X))^{\geq -1} \cap \mathbf{D}^b(\operatorname{coh}(X))^{\leq 1} \subseteq \langle G \rangle_A^{[-A,A]}
    \end{displaymath}
    where the inclusion holds by \cite[Lemma 2.4]{Neeman:2018b}. Applying $U(-)^{\circ}$ we get the required triangle.
    
\end{proof}

\begin{remark}\label{Remark Kmproj standard co-t-structure in preferred equiv class}
    Let $X$ be a noetherian scheme. Assume there exists $G \in \mathbf{D^b_{\operatorname{coh}}}(X)$ such that $\mathbf{D}(\operatorname{Qcoh}(X)) = \overline{\langle G \rangle}_N$ for some positive integer $N$. By the proof of \Cref{Theorem Kinj co-approx}, \Cref{Proposition approximable implies preferred equivalence class}, and \Cref{Theorem Co-Approximability and G-Approximability}, we get that the metric given by $\{\mathbf{D}^{b}(\operatorname{coh}(X))^{\leq -n}\}_{n \in \mathbb{Z}}$ is equivalent to the metric given by $\{\langle G \rangle ^{(-\infty,-n]}\}_{n \in \mathbb{Z}}$. This also follows immediately from \cite[Lemma 2.2]{Neeman:2018b}.
\end{remark}

\begin{theorem}\label{Theorem Examples of co-approximability}
    Let $X$ be a separated noetherian scheme. Then there exists a complex $G$ in $\mathbf{D^b_{\operatorname{coh}}}(X)$ and a positive integer $N$ such that $\mathbf{D}(\operatorname{Qcoh}(X)) = \overline{\langle G \rangle}_N$ in the following cases, 
    \begin{enumerate}
        \item\label{reg} $X$ is regular and finite dimensional.
        \item\label{qexc} $X$ is a quasiexcellent and finite dimensional scheme.
    \end{enumerate} 
    Further, in either of these case, for any Azumaya algebra $\mathcal{A}$, there exists a complex $G$ in $\mathbf{D^b_{\operatorname{coh}}}(\mathcal{A})$ and a positive integer $N$ such that $\mathbf{D}(\operatorname{Qcoh}(\mathcal{A})) = \overline{\langle G \rangle}_N$ by \cite[Corollary 6.1 \& Corollary 6.2]{DeDeyn/Lank/ManaliRahul:2024b}.

    Note that the homotopy category of injectives is co-approximable in these cases by \Cref{Theorem Kinj co-approx} and \Cref{Remark Kinj co-approx for stacks and algebras}.
\end{theorem}
\begin{proof}
    \begin{enumerate}
        \item  If $X$ is regular and finite dimensional, this is well known, see for example \cite[Theorem 7.5]{Rouquier:2008} or \cite[Theorem 7.5]{Neeman:2021c}.
        \item This is shown in \cite[Proof of Main Theorem, page 137]{Aoki:2021}.
    \end{enumerate}
\end{proof}

The following is the analogous result for stacks.
\begin{theorem}\label{Theorem Examples of co-approximability for stacks}
    Let $\mathcal{X}$ be a noetherian, separated, finite-dimensional, quasiexcellent tame algebraic stack. Then, there exists a complex $G$ in $\mathbf{D^b_{\operatorname{coh}}}(\mathcal{X})$ and a positive integer $N$ such that $\mathbf{D}(\operatorname{Qcoh}(\mathcal{X})) = \overline{\langle G \rangle}_N$.

    Note that the homotopy category of injectives is co-approximable by \Cref{Theorem Kinj co-approx} and \Cref{Remark Kinj co-approx for stacks and algebras}.
\end{theorem}
\begin{proof}
    This follows from \cite[Corollary 4.5 \& Theorem 5.1]{Aoki:2021} and \cite[Lemma B.2]{Hall/Priver:2024} by the same proof idea as \cite[Theorem B.1]{Hall/Priver:2024}
\end{proof}

\subsection{Miscellaneous examples}

We now give two examples where the generating sequence does not arise from a single compact generator. Hence, in particular these are not examples of approximability or co-approximability.

The first example is of the derived category of the classifying stack of $\mathbb{G}_m$ over the field $\mathbb{C}$. The derived category $\mathbf{D}_{\operatorname{Qcoh}}(B\mathbb{G}_m)$ is exceptionally simple here. We start with the structure of quasicoherent sheaves on $B\mathbb{G}_m$, which is well known in the literature, see for example \cite[Theorem 2.1 and Theorem 2.15]{Schwarz:2018}.
\begin{lemma}\label{Lemma on quasicoherent sheaves on BG_m}
    For all $i \in \mathbb{Z}$, let $G_i$ be the quasicoherent sheaf corresponding to the 1-dimensional weight $i$ representation. Then,
    \begin{enumerate}
        \item $\Hom{}{G_i}{G_j} = 0$ for all $i \neq j$.
        \item $\Hom{}{G_i}{G_i} \cong \mathbb{C}$ for all $i \in \mathbb{Z}$.
        \item Any quasicoherent sheaf $F$ is isomorphic to a direct sum of copies of $G_i$, that is, $F \cong \bigoplus_{i \in \mathbb{Z}} G_i^{\oplus I_i}$ for sets $I_i$ for all $i \in \mathbb{Z}$.
    \end{enumerate}
    This gives us that in the derived category, $\mathbf{D}_{\operatorname{Qcoh}}(B\mathbb{G}_m)$, we have that, 
    \begin{enumerate}
        \item $\Hom{\mathbf{D}_{\operatorname{Qcoh}}(B\mathbb{G}_m)}{\Sigma^n G_i}{G_j} = 0$ for all $i \neq j$ and for all $n \in \mathbb{Z}$.
        \item $\Hom{\mathbf{D}_{\operatorname{Qcoh}}(B\mathbb{G}_m)}{\Sigma^n G_i}{G_i} = 0 $ for all $n \neq 0$ and for all $i \in \mathbb{Z}$.
        \item $\Hom{\mathbf{D}_{\operatorname{Qcoh}}(B\mathbb{G}_m)}{G_i}{G_i} \cong \mathbb{C}$ for all $i \in \mathbb{Z}$.
    \end{enumerate}
\end{lemma}
\begin{remark}\label{Remark on Derived category of BG_m}
    By \Cref{Lemma on quasicoherent sheaves on BG_m}, $\mathbf{D}_{\operatorname{Qcoh}}(B\mathbb{G}_m) = \operatorname{Add}(\{\Sigma^{j}G_i\}_{i,j \in \mathbb{Z}})$. We have a finite generating sequence (\Cref{Definition of generating sequence}) $\mathcal{G}$ on $\mathbf{D}_{\operatorname{Qcoh}}(B\mathbb{G}_m)$, given by 
       \[\mathcal{G}^n \colonequals \{ \Sigma^{n+|i|} G_i : |i| \leq -n\} \cup \{\Sigma^{-n-|i|} G_i : |i| \leq -n\}  \] 
    Note that $\mathcal{G}_n = \varnothing$ for all positive integers $n$.
  
    We can explicitly compute the extended good metric $\mathcal{M}^{\mathcal{G}}$ (see \Cref{Definition of generating sequence}(2)) and $\overline{\langle G \rangle}{[-n,n]}$ (see \Cref{Notation Subcategories for (pre)-generating sequences}) for this generating sequence,
    \begin{enumerate}
        \item $\mathcal{G}(-\infty, -n] = \{\Sigma^j G_i : |j|+|i| \geq |n|\}$ for all $n \leq 0$, and $\mathcal{G}(-\infty, n] = \mathcal{G}(-\infty, 0]$ for all $n \geq 0$.
        \item $\mathcal{M}^{\mathcal{G}}_n = \operatorname{add}(\mathcal{G}(-\infty, -n]), \mathcal{R}^{\mathcal{G}}_{n} = \operatorname{Add}(\mathcal{G}(-\infty, -n])$ for all $n \in \mathbb{Z}$.
        \item For any $n \geq 0$, $\mathcal{G}[-n,n] = \{\Sigma^j G_i : |j|+|i| \leq n\}$.
        \item For all $n \geq 0$, $\overline{\langle G \rangle}{[-n,n]} = \operatorname{Add}(\mathcal{G}[-n,n])$.
    \end{enumerate}
\end{remark}

\begin{theorem}\label{Theorem G-approximability for BG_m}
    The derived category of $B\mathbb{G}_m$ is $\mathcal{G}$-approximable (\Cref{Definition of G-approximability}), with $\mathcal{G}$ being defined as in \Cref{Remark on Derived category of BG_m}.
\end{theorem}
\begin{proof}
    We will show $\mathcal{G}$-approximability using the orthogonal metric $\mathcal{R}^{\mathcal{G}}$ (see \Cref{Definition of generating sequence} for the definition of $\mathcal{R}^{\mathcal{G}}$) and the integer $A=1$. We verify below the conditions of \Cref{Remark conditions of G-approximability},
    \begin{enumerate}
     \setcounter{enumi}{-1}
        \item $\mathcal{M}^{\mathcal{G}}_i = \bigcup_{|n| \leq 1} \Sigma^n \mathcal{M}^{\mathcal{G}}_{i+1}$ for all $i \in \mathbb{Z}$ by \Cref{Remark on Derived category of BG_m}(1)-(2).
        \item $\mathcal{M}^{\mathcal{G}}_i = \mathcal{R}^{\mathcal{G}}_i \cap \mathsf{T}^c$ for all $i \in \mathbb{Z}$, by \Cref{Remark on Derived category of BG_m}(2) and \cite[Proposition 1.9]{Neeman:2021a}.
        \item $\Hom{\mathbf{D}_{\operatorname{Qcoh}}(B\mathbb{G}_m)}{\mathcal{G}^{n}}{\mathcal{R}^{\mathcal{G}}_{i}} = 0$ for all $n + i \geq 1$ by \Cref{Lemma on quasicoherent sheaves on BG_m}.
        \item[(3)$'$] By \Cref{Remark on Derived category of BG_m}, any $F \in \mathcal{R}^{\mathcal{G}}_i$ for $i \geq 0$ can be written as a direct sum $F = E \oplus D$ with $E \in \overline{\mathcal{G}}^{[-i,0]}_1 = \overline{\mathcal{G}}^{[-i,i]}$ and $D \in \mathcal{R}^{\mathcal{G}}_{i+1}$. This gives us the triangle $E \to F \to D \to \Sigma E$ we needed. For $i < 0 $, it is trivial, as $\mathcal{R}^{\mathcal{G}}_{i} = \mathcal{R}^{\mathcal{G}}_{i+1}$.
    \end{enumerate}
\end{proof}

\begin{remark}
    Let $\operatorname{Spec}(R)$ be an affine scheme with a $\mathbb{G}_m$-action such that the corresponding grading on $R$ is bounded, that is, only finitely many of the graded pieces are non-zero. Then, the the derived category of the quotient stack $[\operatorname{Spec}(R)/\mathbb{G}_m]$ can be shown to be $\mathcal{G}$-approximable analogous to the proof of \Cref{Theorem G-approximability for BG_m}. 
\end{remark}

We now give an example of weak $\mathcal{G}$-approximability for a special class of rings. We begin by recalling the definition of a pure-projective module. Pure-projectivity can be defined in a more general setting, but for modules, the following is an equivalent definition.

\begin{definition}
     Let $R$ be a ring. An $R$-module is \emph{pure-projective} if it is the direct summand of a direct sum of finitely presented modules. Given an $R$-module $M$, we define a pure-projective resolution to be a complex $P$ of pure-projective modules such that $P^i = 0$ for all $i > 0$, which is a \emph{pure exact} resolution of $M$, that is, the complex $\Hom{}{N}{P}$ is a resolution of the module $\Hom{}{N}{M}$ for all pure projective modules $N$.
     
     The \emph{pure-projective dimension} of a module is the minimal length of a pure-projective resolution.
     The \emph{pure global dimension} of a ring $R$ is the supremum of the pure-projective dimension of all $R$-modules.
\end{definition}
The following result by Kie{\l}pi\'{n}ski and Simson gives a large class of rings with finite pure global dimension.

\begin{theorem}
    \cite[Theorem 2.2]{Kielpinski/Simson:1975} For a ring of cardinality less than equal to $\aleph_n$, the pure global dimension is less than or equal to $n+1$.
\end{theorem}

\begin{hypothesis}\label{Condition ring finite pure global dimension}
    A noetherian ring $R$ satisfies \Cref{Condition ring finite pure global dimension} if there exists a positive integer $A$ such that every $R$-module $M$ has a resolution by pure-projective modules of length less than or equal to $A$.
\end{hypothesis}

\begin{remark}
    Note that in \Cref{Condition ring finite pure global dimension}, we do not assume that the resolution by pure-projective modules is a pure-exact resolution. Further, note that the \Cref{Condition ring finite pure global dimension} is satisfied by a ring if it has finite pure global dimension.
\end{remark}

\begin{theorem}
    Let $R$ be commutative noetherian ring satisfying \Cref{Condition ring finite pure global dimension}. Then, $\mathbf{K}(\operatorname{Inj}\text{-}R)$ is weakly $\mathcal{G}$-approximable (\Cref{Definition of G-approximability}) for the generating sequence (\Cref{Definition of generating sequence}) given by $\mathcal{G}^i \colonequals \{\Sigma^{i}M : M \in \operatorname{mod}(R)\}$.
\end{theorem}
\begin{proof}
    We again use the standard co-t-structure $(\mathbf{K}(\operatorname{Inj}\text{-}R)^{\geq 0},\mathbf{K}(\operatorname{Inj}\text{-}R)^{\leq 0})$ on $\mathbf{K}(\operatorname{Inj}\text{-}R)$ to define an orthogonal metric (\Cref{Good Metric}). We define the orthogonal metric $\mathcal{R}$ by setting $\mathcal{R}_i = \mathbf{K}(\operatorname{Inj}\text{-}R)^{\geq n} \colonequals \Sigma^{-n} \mathbf{K}(\operatorname{Inj}\text{-}R)^{\geq 0}$ for all $n \in \mathbb{Z}$. We prove weak $\mathcal{G}$-approximability, see \Cref{Remark conditions of G-approximability}, with this orthogonal metric $\mathcal{R}$, and the integer $A$ given by \Cref{Condition ring finite pure global dimension}.
    \begin{enumerate}
    \setcounter{enumi}{-1}
        \item Note that by definition of $\mathcal{M}^{\mathcal{G}}$ (see \Cref{Definition of generating sequence}(2)), $\mathcal{M}^{\mathcal{G}}_i = \operatorname{smd}(\operatorname{coprod}(\mathcal{G}(-\infty,i]))$, see \Cref{Notation from Neeman's paper} and \Cref{Notation Subcategories for (pre)-generating sequences}. Further, from the definition of $\mathcal{G}$, it is clear that that $\mathcal{G}(-\infty,-i] = \{\Sigma^{-n}M : n \geq i \text{ and } M \in \operatorname{mod}(R)\}$. And so, we get that $\mathcal{M}^{\mathcal{G}}_{i} = \Sigma \mathcal{M}^{\mathcal{G}}_{i+1}$.
        \item Note that $\mathcal{G}(-\infty,-i] = \{\Sigma^{-n}M : n \geq i \text{ and } M \in \operatorname{mod}(R)\} \subseteq \Sigma^{-i} \mathbf{K}(\operatorname{Inj}\text{-}X)^{\geq 0} = \mathcal{R}_i$ for all $i \in \mathbb{Z}$ by definition. And so $\mathcal{M}^{\mathcal{G}}_i = \operatorname{smd}(\operatorname{coprod}(\mathcal{G}(-\infty,i])) \subseteq \mathcal{R}_i$ for all $i \in \mathbb{Z}$. 
        \item $\Hom{}{\mathcal{G}^i}{\mathcal{R}_{n}} = 0$ for all $i+n \geq 1$ using \cite[Lemma 2.1]{Krause:2005}.
        \item Let $F \in \mathcal{R}_i = \mathbf{K}(\operatorname{Inj}\text{-}R)^{\geq i} \cong \mathbf{D}(\operatorname{Qcoh}(X))^{\geq i}$. We take its canonical truncation to get the triangle $\tau^{\leq i}F \to F \to \tau^{\geq i+1}F \to \Sigma \tau^{\leq i}F$. Note that this is the required triangle, as $\tau^{\geq 1}F \in \mathbf{D}(\operatorname{Qcoh}(X))^{\geq i+1} \cong \mathbf{K}(\operatorname{Inj}\text{-}R)^{\geq i+1}$ (\Cref{Lemma coaisle of derived category and "aisle" of homotopy category}) and $\tau^{\leq i}F \cong \Sigma^{-i}M$ for a $R$-module $M$. As $R$ satisfies \Cref{Condition ring finite pure global dimension}, $M \in \langle G \rangle{[0,A]} \subseteq \langle G \rangle{[-A,A]}$, which gives us that $\tau^{\leq i}F \in \langle G \rangle{[-A-i,A-i]}$, which is what we needed.
    \end{enumerate}
\end{proof}

\section{Completion and closure of compacts}
In this section, we will see what completion \`{a} la Neeman gives us for the setting we are interested in, see \Cref{Section Metrics and Completions}, and in particular \Cref{Definition completions}. Further, using these computations, we give some applications of the representability theorems proved in \Cref{Section Representability results}.

We start with a triangulated category $\mathsf{T}$ with a generating sequence $\mathcal{G}$, and compute what the completion machinery gives us when applied to the category of compacts $\mathsf{T}^c$. We recall the relevant definitions again for the reader's convenience in this setup.

\begin{definition}\label{Definition completions for Chapter on completions}\cite[Definition 1.10]{Neeman:2018}
    Let $\mathsf{T}$ be a compactly generated triangulated category with a generating sequence $\mathcal{G}$ (\Cref{Definition of generating sequence}). Further, let $\{\mathcal{L}_i\}_{i \geq 0}$ be a good metric (or equivalently an extended good metric, see \Cref{Good Metric}) on $\mathsf{T}^c$.
    We denote the restricted Yoneda functor by $\mathcal{Y}:\mathsf{T} \to \operatorname{Mod}(\mathsf{T}^c)$, where $\operatorname{Mod}(\mathsf{T}^c)$ denotes $ \Hom{}{[\mathsf{T}^c]^{\operatorname{op}}}{\operatorname{Mod}(\mathbb{Z})}$. Then, we can define the following full subcategories of $\operatorname{Mod}(\mathsf{T}^c)$,
    \begin{enumerate}
        \item $\mathfrak{L}(\mathsf{T}^c) \colonequals \{\colim \mathcal{Y}(E_i) : E_1 \to E_2 \to \cdots \text{ is a Cauchy sequence in } \mathsf{T}^c\}$, see \Cref{Definition Cauchy sequence} for the definition of a Cauchy sequence for a triangulated category with a metric.
        \item $\mathfrak{C}(\mathsf{T}^c) \colonequals \bigcup_{n \geq 1}\mathcal{Y}(\mathcal{M}_n)^{\perp}$.
        \item $\mathfrak{S}(\mathsf{T}^c) \colonequals \mathfrak{L}(\mathsf{T}^c) \cap \mathfrak{C}(\mathsf{T}^c)$. This is a triangulated category by \cite[Theorem 2.11]{Neeman:2018}
    \end{enumerate}
    Note that these three categories remain the same if we replace the metric $\mathcal{L}$ with an equivalent metric, see \Cref{Definition equivalence relation on extended good metrics}.
\end{definition}
Recall that \Cref{Theorem Tc to T is a good extension} gives us a way of computing the completion of $\mathsf{T}^c$. We now compute the categories $\mathfrak{L}'(\mathsf{T}^c)$, $\mathcal{Y}^{-1}(\mathfrak{C}(\mathsf{T}^c))$, and $\mathfrak{S}'(\mathsf{T}^c)$ of \Cref{Theorem Tc to T is a good extension} in the setup we are interested in. Recall from \Cref{Theorem Tc to T is a good extension} that, 
    \[\mathfrak{L}'(\mathsf{T}^c) \colonequals \{F \in \mathsf{T} : F \cong \hocolim E_n \text{ for a Cauchy sequence } E_1 \to E_2 \to \cdots \text{ in } \mathsf{T}^c \}\] 
\begin{proposition}\label{Proposition 1 on L'(Tc)}
    Let $\mathsf{T}$ a triangulated category with a generating sequence $\mathcal{G}$ (\Cref{Definition of generating sequence}), and an orthogonal metric $\mathcal{R}$ (\Cref{Good Metric}) such that for all $n\in \mathbb{Z} $, $\HomT{\mathcal{G}^n}{\mathcal{R}_i} = 0$ for all $i >> 0$. If we consider the compacts with respect to the extended good metric $\mathcal{R} \cap \mathsf{T}^c$, and the closure of the compacts $\overline{\mathsf{T}^c}$ (\Cref{Definition closure of compacts}) with the metric $\mathcal{R}$, then $\mathfrak{L}'(\mathsf{T}^c) = \overline{\mathsf{T}^c}$.
\end{proposition}
\begin{proof}
    From \Cref{Lemma on strong approximating sequences}(3), we have that,
    \[\mathfrak{L}'(\mathsf{T}^c) = \{F \in \mathsf{T} : F\text{ has a strong }\mathsf{T}^c \text{-approximating system}\}\]
    noting that any $\mathsf{T}^c$-approximating system (\Cref{Definition of an Approximating Sequence}) is a Cauchy sequence (\Cref{Definition Cauchy sequence}) with respect to the metric $\mathcal{R} \cap \mathsf{T}^c$ and conversely, any Cauchy sequence with respect to $\mathcal{R} \cap \mathsf{T}^c$ has a subsequence which is a strong $\mathsf{T}^c$-approximating system.
    Then, the result immediately follows from \Cref{Lemma on strong approximating sequences}(1) and \Cref{Lemma on strong approximating sequences}(3).
\end{proof}

\begin{proposition}\label{Proposition computation of orthogonal to metric}
    Let $\mathsf{T}$ a triangulated category with a generating sequence $\mathcal{G}$, see \Cref{Definition of generating sequence}. Let $\mathcal{Y} : \mathsf{T} \to \Hom{}{[\mathsf{T}^c]^{\operatorname{op}}}{\operatorname{Mod}(\mathbb{Z})}$ be the restricted Yoneda functor. Then,
    \begin{enumerate}
        \item If we equip $\mathsf{T}^c$ with the metric $\mathcal{M}^{\mathcal{G}}$ (see \Cref{Definition of generating sequence}), then, $\mathcal{Y}^{-1}(\mathfrak{C}(\mathsf{T}^c)) = \big(\mathcal{M}^{\mathcal{G}}\big)^{\perp}$, see \Cref{Definition closure of compacts} for the definition of $\big(\mathcal{M}^{\mathcal{G}}\big)^{\perp}$ and see \Cref{Definition completions for Chapter on completions}(2) for the definition of $\mathfrak{C}(\mathsf{T}^c)$.
        \item If $\mathsf{T}$ is weakly $\mathcal{G}$-quasiapproximable (\Cref{Definition of G-quasiapproximability}), and we equip $\mathsf{T}^c$ with the metric $\mathcal{R}^{\mathcal{G}} \cap \mathsf{T}^c$ (see \Cref{Definition of generating sequence}), then also $\mathcal{Y}^{-1}(\mathfrak{C}(\mathsf{T}^c)) = \big(\mathcal{M}^{\mathcal{G}}\big)^{\perp}$.
    \end{enumerate}
\end{proposition}
\begin{proof}
    \begin{enumerate}
        \item Follows directly from \Cref{Theorem Tc to T is a good extension}(2).
        \item As $\mathsf{T}$ is weakly $\mathcal{G}$-quasiapproximable, the metrics $\mathcal{M}^{\mathcal{G}}$ and $\mathcal{R}^{\mathcal{G}} \cap \mathsf{T}^c$ are equivalent by \Cref{Proposition approximable implies preferred equivalence class}, and hence we are done by (1).
    \end{enumerate}
\end{proof}

\begin{corollary}\label{Corollary Closure of compact for weak G-approximability}
    Let $\mathsf{T}$ be a weakly $\mathcal{G}$-quasiapproximable triangulated category (\Cref{Definition of G-quasiapproximability}). Consider $\mathsf{T}^c$ with the extended good metric $\mathcal{M}^{\mathcal{G}}$, with $\mathcal{M}^{\mathcal{G}}_n = \mathcal{G}^{(-\infty,-n]}$, see \Cref{Definition of generating sequence} and \Cref{Notation Subcategories for (pre)-generating sequences}. Then, $\mathfrak{L}'(\mathsf{T}^c) = \overline{\mathsf{T}^c}$ (\Cref{Definition completions for Chapter on completions}(1)), where the closure of the compacts $\overline{\mathsf{T}^c}$ (see \Cref{Definition closure of compacts}) is with respect to the metric $\mathcal{R}^{\mathcal{G}}$, see \Cref{Definition closure of compacts}. Further, the category $\mathfrak{S}'(\mathsf{T}^c) = \mathsf{T}^b_c$, see \Cref{Definition completions for Chapter on completions}(3) and \Cref{Definition closure of compacts}.
\end{corollary}
\begin{proof}
    For a weakly $\mathcal{G}$-quasiapproximable triangulated category, the metrics $\mathcal{R}^{\mathcal{G}} \cap \mathsf{T}^c$ and $\mathcal{M}^{\mathcal{G}}$ are equivalent by \Cref{Proposition approximable implies preferred equivalence class}, see \Cref{Definition of generating sequence} for the definition of $\mathcal{R}^{\mathcal{G}}$. So, the first part follows from \Cref{Proposition 1 on L'(Tc)}. Further, by \Cref{Proposition computation of orthogonal to metric}, we get that $\mathcal{Y}^{-1}(\mathfrak{C}(\mathsf{T}^c)) = \left(\mathcal{M}^{\mathcal{G}}\right)^{\perp}$, see \Cref{Definition closure of compacts}. So, 
    \[\mathsf{T}^b_c \colonequals \overline{\mathsf{T}^c} \cap \left(\mathcal{M}^{\mathcal{G}}\right)^{\perp} = \mathfrak{L}'(\mathsf{T}^c) \cap \mathcal{Y}^{-1}(\mathfrak{C}(\mathsf{T}^c)) = \mathfrak{S}'(\mathsf{T}^c)\]
    which is what we needed to show.
\end{proof}

Now, we do a few concrete computations. We start with the following lemma.

\begin{lemma}\label{Lemma Standard co-t-structure is compactly generated}
    Let $X$ be a noetherian scheme. Then the standard co-t-structure on $\mathbf{K}(\operatorname{Inj}\text{-}X)$, which we denote by $(\mathbf{K}(\operatorname{Inj}\text{-}X)^{\geq 0},\mathbf{K}(\operatorname{Inj}\text{-}X)^{\leq 0})$, is compactly generated (see \Cref{Definition of compactly generated co-t-structure}). In fact, $\mathbf{K}(\operatorname{Inj}\text{-}X)^{\leq 0} = \left(\bigcup_{n \geq 1}\Sigma^{-n} \operatorname{Coh}(X)\right)^{\perp}$. 
\end{lemma}
\begin{proof}
    Consider the full subcategory $\operatorname{Coh}(X) \subseteq \mathbf{D}^b_{\operatorname{coh}}(X) \subseteq \mathbf{K}(\operatorname{Inj}\text{-}X)$. We claim that the co-t-structure generated by $\operatorname{Coh}(X)$ is exactly the standard co-t-structure. Let the co-t-structure generated by $\operatorname{Coh}(X)$ be $(\mathsf{U},\mathsf{V})$. That is, $\mathsf{V} = \left(\bigcup_{n \geq 1}\Sigma^{-n} \operatorname{Coh}(X)\right)^{\perp}$ and $\mathsf{U} = {}^{\perp}\left(\Sigma \mathsf{V} \right)$ It is enough to show that $\mathsf{V} = \mathbf{K}(\operatorname{Inj}\text{-}X)^{\leq 0}$.

    Note that $\mathbf{K}(\operatorname{Inj}\text{-}X)^{\leq 0} \subseteq \left(\bigcup_{n \geq 1}\Sigma^{-n} \operatorname{Coh}(X)\right)^{\perp} = \mathsf{V}$, by \cite[Lemma 2.1]{Krause:2005}. So, we just need to show the other inclusion. Let $F \in \mathsf{V}$. Then, first of all note that if any of the cohomology sheaves $\mathcal{H}^n(X) \neq 0$ for some $n \geq 1$, then there exists a map $\mathcal{F} \to \mathcal{Z}^n(F)$ for a coherent sheaf $\mathcal{F}$,  where $ \mathcal{Z}^n(F)$ is the kernel of the n$^{\operatorname{th}}$ differential, such that the induced map $\mathcal{F} \to \mathcal{H}^n(\mathcal{F})$ is not zero. This gives us a non-zero map from $\Sigma^{-n}\mathcal{F} \to F$, which is a contradiction. So, $\mathcal{H}^n(X) = 0$ for all $n \geq 1$. So, the brutal truncation $\sigma^{\geq 0}F$ gives a injective resolution of $\mathcal{Z}^0(F)$, the kernel of the 0$^{\operatorname{th}}$ differential. If $\mathcal{Z}^0(F)$ is not injective, by Baer's criterion (see for example \cite[Lemma 11.2.10]{Krause:2022}), $\operatorname{Ext}^1(\mathcal{F},\mathcal{Z}^0(F)) \neq 0$ for some coherent sheaf $\mathcal{F}$. 
    Note that a non-zero element of $\operatorname{Ext}^1(\mathcal{F},\mathcal{Z}^0(F))$ corresponds to a non-zero map $\Sigma^{-1} \mathcal{F} \to F$. And so, there exists a non-zero map $\Sigma^{-1} \mathcal{F} \to F$ for a coherent sheaf $\mathcal{F}$, which is a contradiction. Therefore, $\mathcal{Z}^0(F)$ is injective. Now, we have a short exact sequence of complexes as follows,
    \[\begin{tikzcd}
	   \cdots & {F^{-2}} & {F^{-1}} & {\mathcal{Z}^0(F)} & 0 & 0 & \cdots \\
	   \cdots & {F^{-2}} & {F^{-1}} & {F^0} & {F^1} & {F^2} & \cdots \\
	   \cdots & 0 & 0 & {F^0/\mathcal{Z}^0(F)} & {F^1} & {F^2} & \cdots
	   \arrow[from=1-1, to=1-2]
	   \arrow[from=1-2, to=1-3]
	   \arrow[from=1-2, to=2-2]
	   \arrow[from=1-3, to=1-4]
	   \arrow[from=1-3, to=2-3]
	   \arrow[from=1-4, to=1-5]
	   \arrow[from=1-4, to=2-4]
	   \arrow[from=1-5, to=1-6]
	   \arrow[from=1-5, to=2-5]
	   \arrow[from=1-6, to=1-7]
	   \arrow[from=1-6, to=2-6]
	   \arrow[from=2-1, to=2-2]
	   \arrow[from=2-2, to=2-3]
	   \arrow[from=2-2, to=3-2]
	   \arrow[from=2-3, to=2-4]
	   \arrow[from=2-3, to=3-3]
	   \arrow[from=2-4, to=2-5]
	   \arrow[from=2-4, to=3-4]
	   \arrow[from=2-5, to=2-6]
	   \arrow[from=2-5, to=3-5]
	   \arrow[from=2-6, to=2-7]
	   \arrow[from=2-6, to=3-6]
	   \arrow[from=3-1, to=3-2]
	   \arrow[from=3-2, to=3-3]
	   \arrow[from=3-3, to=3-4]
	   \arrow[from=3-4, to=3-5]
	   \arrow[from=3-5, to=3-6]
	   \arrow[from=3-6, to=3-7]
    \end{tikzcd}\]
    Note that this gives us a triangle in $\mathbf{K}(\operatorname{Inj}\text{-}X)$ as all of the columns are split short exact sequences. The columns in degree not equal to 0 split trivially, and the column in degree 0 splits as $\mathcal{Z}^0(F)$ is injective. Further, this implies that $F^0/\mathcal{Z}^0(F)$ is also injective, as it is a summand of the injective object $F^0$. As $\mathcal{H}^n(F) = 0$ for all $n \geq 1$, the complex given by the bottom row is in fact acyclic. As it is a bounded below acyclic complex of injectives, it is in fact further null-homotopic. So, we get that $F$ is isomorphic to the complex given by the top row, and hence $F \in \mathbf{K}(\operatorname{Inj}\text{-}X)^{\leq 0}$, which is what we needed to show.
\end{proof}

\begin{theorem}\label{Theorem completion of compacts for for Kinj}
    Let $X$ be a noetherian scheme. Consider the homotopy category of injectives $\mathbf{K}(\operatorname{Inj}\text{-}X)$.
    We endow the category of compacts $\mathbf{D}^b_{\operatorname{coh}}(X)$ with the metric which is given by $\{\mathbf{D}^b_{\operatorname{coh}}(X)\cap \mathbf{K}(\operatorname{Inj}\text{-}X)^{\geq n}\}_{n \in \mathbb{Z}}$, where $(\mathbf{K}(\operatorname{Inj}\text{-}X)^{\geq 0}, \mathbf{K}(\operatorname{Inj}\text{-}X)^{\leq 0})$ is the standard co-t-structure on $\mathbf{K}(\operatorname{Inj}\text{-}X)$ defined via brutal truncations. Then, 
    \begin{enumerate}
        \item $\mathfrak{L}'(\mathbf{D}^b_{\operatorname{coh}}(X)) \cong \mathbf{D}^{+}_{\operatorname{coh}}(X)$, see \Cref{Theorem Tc to T is a good extension} for the definition of $\mathfrak{L}'(\mathbf{D}^b_{\operatorname{coh}}(X))$. Note that this agrees with $\mathbf{K}(\operatorname{Inj}\text{-}X)^+_c$ if $\mathbf{D}(\operatorname{Qcoh}(X)) = \overline{\langle G \rangle}_N$ for an object $G \in \mathbf{D}^b_{\operatorname{coh}}(X)$ and $N \geq 1$ by \Cref{Corollary Closure of compact for weak G-approximability}, see \Cref{Definition T^+_c and T^-}, and \Cref{Remark 2 on co-approx and G-approx }, \Cref{Remark Kinj standard co-t-structure in preferred equiv class}.
        \item $\mathcal{Y}^{-1}(\mathfrak{C}(\mathbf{D}^b_{\operatorname{coh}}(X))) = \bigcup_{n \in \mathbb{Z}} \mathbf{K}(\operatorname{Inj}\text{-}X)^{\leq n}$, see \Cref{Definition completions for Chapter on completions}(2).
        \item  $\mathbf{D}^{b}(\operatorname{Inj}\text{-}X) \colonequals \mathbf{K}^{b}(\operatorname{Inj}\text{-}X) \cap \mathbf{D}^b_{\operatorname{coh}}(X) = \mathfrak{S}'(\mathbf{D}^b_{\operatorname{coh}}(X)) \cong \mathfrak{S}({\mathbf{D}^b_{\operatorname{coh}}(X)})$, see \Cref{Definition completions for Chapter on completions}(3) and \Cref{Theorem Tc to T is a good extension}(3).  Further, this agrees with $\mathbf{K}(\operatorname{Inj}\text{-}X)^b_c$ if $\mathbf{D}(\operatorname{Qcoh}(X)) = \overline{\langle G \rangle}_N$ for an object $G \in \mathbf{D}^b_{\operatorname{coh}}(X)$ and $N \geq 1$ by \Cref{Corollary Closure of compact for weak G-approximability}, see \Cref{Definition T^+_c and T^-}, and \Cref{Remark 2 on co-approx and G-approx }, \Cref{Remark Kinj standard co-t-structure in preferred equiv class}
    \end{enumerate}
\end{theorem}
\begin{proof}
    \begin{enumerate}
        \item Let $F \in \mathfrak{L}'(\mathbf{D}^b_{\operatorname{coh}}(X)$. Then, by definition, we can choose a sequence  $E_1 \to E_2 \to E_3 \to E_4 \to \cdots$ in $\mathbf{D}^b_{\operatorname{coh}}(X)$ mapping to $F$ such that $\operatorname{Cone}(E_n \to E_{n + i}) \in \mathbf{K}(\operatorname{Inj}\text{-}X)^{\geq n + 1} $ for all $n,i \geq 1$, and $\hocolim E_n \cong F$. And so, $\operatorname{Cone}(E_n \to F) \in \mathbf{K}(\operatorname{Inj}\text{-}X)^{\geq n}$ for all $n \geq 1$. This shows that $F \in \mathbf{K}^{+}(\operatorname{Inj}\text{-}X) = \bigcup_{n \in \mathbb{Z}}\mathbf{K}(\operatorname{Inj}\text{-}X)^{\geq n}$. Now, we use the equivalence of categories $p : \mathbf{D}^{+}(\operatorname{Qcoh}(X)) \to \mathbf{K}^{+}(\operatorname{Inj}\text{-}X)$, see \Cref{Lemma coaisle of derived category and "aisle" of homotopy category}. As we have $\operatorname{Cone}(E_n \to F) \in \mathbf{K}(\operatorname{Inj}\text{-}X)^{\geq n} = \mathbf{D}(\operatorname{Qcoh}(X))^{\geq n}$ for all $n \geq 1$, we get that $H^i(E_n) \cong H^i(F)$ for all $i < n$. So, $F \in \mathbf{D}^{+}_{\operatorname{coh}}(X)$. 
        
        Conversely, for any $F \in \mathbf{D}^{+}_{\operatorname{coh}}(X)$, the canonical truncations give us give us a Cauchy sequence in $\mathbf{D}^b_{\operatorname{coh}}(X)$,
        \[\tau^{\leq 0}F \to \tau^{\leq 1}F \to \tau^{\leq 2}F \to \tau^{\leq 3}F \to \cdots\]
        Note that $F \cong \colim \tau^{\leq n}F$ as complexes. So, we have that by \cite[Remark 2.2]{Bkstedt/Neeman:1993}, $F \cong \hocolim \tau^{\leq n}F$ in the derived category.

        \item Note that by \Cref{Theorem Tc to T is a good extension}(2),
        \[\mathcal{Y}^{-1}(\mathfrak{C}(\mathbf{D}^b_{\operatorname{coh}}(X))) =  \bigcup_{n \geq 1} \left(\mathbf{D}^b_{\operatorname{coh}}(X)\cap \mathbf{K}(\operatorname{Inj}\text{-}X)^{\geq n}\right)^{\perp}\]
        But, note that $\left(\mathbf{D}^b_{\operatorname{coh}}(X)\cap \mathbf{K}(\operatorname{Inj}\text{-}X)^{\geq n}\right)^{\perp} = \left(\bigcup_{i \geq n}\Sigma^{-i}\operatorname{Coh}(X)\right)^{\perp}$, which is the same as $\mathbf{K}(\operatorname{Inj}\text{-}X)^{\leq n-1}$ by \Cref{Lemma Standard co-t-structure is compactly generated}.
        \item Note that $\mathcal{L}'(\mathbf{D}^b_{\operatorname{coh}}(X)) \cap \mathcal{Y}^{-1}(\mathfrak{C}(\mathbf{D}^b_{\operatorname{coh}}(X))) \cong \mathfrak{S}'(\mathbf{D}^b_{\operatorname{coh}}(X)) \cong \mathfrak{S}(\mathbf{D}^b_{\operatorname{coh}}(X))$, see \Cref{Theorem Tc to T is a good extension}(3). So, (3) follows from parts (1) and (2).
    \end{enumerate}
\end{proof}

We note that the metric used in \Cref{Theorem completion of compacts for for Kinj} can be intrinsically defined up to equivalence. 
\begin{theorem}
    Let $X$ be a noetherian scheme, then the metric $\{\mathbf{D}^b_{\operatorname{coh}}(X)\cap \mathbf{K}(\operatorname{Inj}\text{-}X)^{\geq n}\}$ is equivalent to the intrinsically defined metric $\{\Sigma^{-n}\mathcal{Q}(\mathbf{D}^b_{\operatorname{coh}}(X))\}$, where $\mathcal{Q}(\mathbf{D}^b_{\operatorname{coh}}(X))$ is a full subcategory of $\mathbf{D}^b_{\operatorname{coh}}(X)$ such that $\{\Sigma^{-n}\mathcal{Q}(\mathbf{D}^b_{\operatorname{coh}}(X))\}_{n \in \mathbb{Z}}$ is the finest metric (up to equivalence) in the following collection,
    \[\left\{\{\Sigma^{-n}\mathcal{P}\}_{n \in \mathbb{Z}} : \mathcal{P} = G(-\infty,0]^{\perp} \text{ for some } G \in \mathbf{D}^b_{\operatorname{coh}}(X)\right\}\]
    see \Cref{Definition equivalence relation on extended good metrics} for the order relation on metrics.
\end{theorem}
\begin{proof}
    Observe that the metric $\{\mathbf{D}^b_{\operatorname{coh}}(X)\cap \mathbf{K}(\operatorname{Inj}\text{-}X)^{\geq n}\}$ is the same as the metric $\{\mathbf{D}^b_{\operatorname{coh}}(X)\cap \mathbf{D}_{\operatorname{Qcoh}}(X)^{\geq n}\}$. Then, this follows by \cite[Theorem 6.3]{Neeman:2018}, as $\mathbf{D}^b_{\operatorname{coh}}(X) = \mathsf{T}^b_c$ for the noetherian triangulated category $\mathbf{D}_{\operatorname{Qcoh}}(X)$, see \Cref{Theorem Tc- Tbc for schemes Neeman}.
\end{proof}
From this, we get the following corollary
\begin{corollary}
    For a noetherian scheme $X$, there is an intrinsic way to construct $\mathbf{D}^{b}(\operatorname{Inj}\text{-}X)$ from $\mathbf{D}^b_{\operatorname{coh}}(X)$.
\end{corollary}
Under the assumption of the existence of a dualizing complex, $\mathbf{D}^{b}(\operatorname{Inj}\text{-}X)$ is equivalent to the category of perfect complexes.

\begin{theorem}\label{Theorem Dbinj and Dperf equivalent with dualizing Complex}
    Let $X$ be a noetherian scheme with a dualizing complex $\omega \in \mathbf{D}^b_{\operatorname{coh}}(X)$. Then, we have an equivalence $\mathbf{D}^{b}(\operatorname{Inj}\text{-}X) \cong \mathbf{D}^{\operatorname{perf}}(X)$
\end{theorem}
\begin{proof}
    The map $-\otimes^{\mathbb{L}}\omega : \mathbf{D}^{\operatorname{perf}}(R) \to \mathbf{D}^{b}(\operatorname{Inj}\text{-}X)$ gives us the required equivalence. We can check this locally. But the local result is immediate from \cite[Proposition 4.7]{Iyengar/Krause:2006} using the fact that the functors in the statement preserve coherent cohomology.
\end{proof}

We have analogous result for the mock homotopy category of projectives.

\begin{theorem}\label{Theorem completion of compacts for for Kmproj}
    Let $X$ be a noetherian, separated scheme. Consider the mock homotopy category of projectives $\mathbf{K}_m(\operatorname{Proj}\text{-}X)$.
    We endow the category of compacts $ \mathbf{K}_m(\operatorname{Proj}\text{-}X)^{c} \cong  \mathbf{D}^b_{\operatorname{coh}}(X)^{\operatorname{op}}$ (see \cite[Theorem 7.4]{Murfet:2008}) with the metric which is given by $\{\mathbf{D}^b_{\operatorname{coh}}(X)^{\leq -n}\}_{n \in \mathbb{Z}}$. Then, with notation as in \cite{Murfet:2008}, 
    \begin{enumerate}
        \item $\mathfrak{L}'(\mathbf{D}^b_{\operatorname{coh}}(X)) = U_\lambda(D^{-}(\operatorname{coh}(X))^{\circ}$, see \Cref{Theorem Tc to T is a good extension} for the definition of $\mathfrak{L}'(\mathbf{D}^b_{\operatorname{coh}}(X))$. Note that this agrees with $\mathbf{K}_{m}(\operatorname{Proj}\text{-}X)^+_c$ if $\mathbf{D}(\operatorname{Qcoh}(X)) = \overline{\langle G \rangle}_N$ for an object $G \in \mathbf{D}^b_{\operatorname{coh}}(X)$ and $N \geq 1$ by \Cref{Corollary Closure of compact for weak G-approximability}, see \Cref{Definition T^+_c and T^-}, and \Cref{Remark 2 on co-approx and G-approx }, \Cref{Remark Kmproj standard co-t-structure in preferred equiv class}.
        \item $\mathcal{Y}^{-1}(\mathfrak{C}(\mathbf{D}^b_{\operatorname{coh}}(X))) = \bigcup_{n \in \mathbb{Z}} \Sigma^{n} \mathsf{V}$ where $(\mathsf{U},\mathsf{V})$ is the co-t-structure of \Cref{Convention Mock homotopy category of projectives}.
        \item  $\mathbf{D}^{\operatorname{perf}}(X) = \mathfrak{S}'(\mathbf{D}^b_{\operatorname{coh}}(X)) \cong \mathfrak{S}({\mathbf{D}^b_{\operatorname{coh}}(X)})$, see \Cref{Definition completions for Chapter on completions}(3) and \Cref{Theorem Tc to T is a good extension}(3).  Note that this agrees with $\mathbf{K}_{m}(\operatorname{Proj}\text{-}X)^b_c$ if we have that $\mathbf{D}(\operatorname{Qcoh}(X)) = \overline{\langle G \rangle}_N$ for an object $G \in \mathbf{D}^b_{\operatorname{coh}}(X)$ and $N \geq 1$ by \Cref{Corollary Closure of compact for weak G-approximability}, see \Cref{Definition T^+_c and T^-}, \Cref{Remark 2 on co-approx and G-approx }, and \Cref{Remark Kmproj standard co-t-structure in preferred equiv class}.
    \end{enumerate}
\end{theorem}
\begin{proof}
    \begin{enumerate}
        \item The proof is analogous to \Cref{Theorem completion of compacts for for Kinj}(1) by using \cite[Lemma 7.3]{Murfet:2008}
    
        \item Obvious from definition.
        
        \item Follows by \cite[Proposition 5.6]{Neeman:2018} using \cite[Lemma 7.3]{Murfet:2008}.
    \end{enumerate}
\end{proof}

We have the following result for the derived category of $B\mathbb{G}_m$.
\begin{theorem}\label{Theorem completion for BGm}
    For all $i \in \mathbb{Z}$, let $G_i$ be the quasicoherent sheaf corresponding to the 1-dimensional weight $i$ representation. Let $\mathsf{T} = \mathbf{D}_{\operatorname{Qcoh}}(B\mathbb{G}_m)$. We equip $\mathsf{T}$ with the generating sequence $\mathcal{G}$ of \Cref{Remark on Derived category of BG_m}, and $\mathsf{T}^c$ with the metric $\mathcal{M}^{\mathcal{G}}$, see \Cref{Definition of generating sequence}(2). Then,
    \begin{enumerate}
        \item  $ \overline{\mathsf{T}^c} =\mathfrak{L}'(\mathsf{T}^c) =  \left\{\bigoplus_{i \in \mathbb{Z}} \bigoplus_{j \in \mathbb{Z}} \Sigma^i G_j^{\oplus n_{ij}} : n_{ij} \in \mathbb{N} \text{ for all i,j}\right\}$
        \item $\mathsf{T}^b_c = \mathfrak{S}({\mathsf{T}^c}) \cong \mathfrak{S}'(\mathsf{T}^c) = \mathsf{T}^c$
    \end{enumerate}
\end{theorem}
\begin{proof}
    The computation of $\mathfrak{L}'(\mathsf{T}{\mathsf{T}^c})$ and $\mathfrak{S}'(\mathsf{T}^c)$ is straightforward from the definitions, and the structure of the derived category, see \Cref{Lemma on quasicoherent sheaves on BG_m} and \Cref{Remark on Derived category of BG_m}. Further, as $\mathbf{D}_{\operatorname{Qcoh}}(B\mathbb{G}_m)$ is $\mathcal{G}$-approximable, see \Cref{Theorem G-approximability for BG_m}, we have that $\overline{\mathsf{T}^c} = \mathfrak{L}'(\mathsf{T}^c) $ and $\mathsf{T}^b_c = \mathfrak{S}'(\mathsf{T}^c) $ by \Cref{Corollary Closure of compact for weak G-approximability}.
\end{proof}

\subsection{Applications of the representability theorems}
We now state some applications of the abstract representability results of the previous chapter, using the computations of the closure of the compacts done above. We don't mention the examples involving approximable triangulated categories here, as those are already covered in \cite[Theorem 0.3]{Neeman:2021b}. 
\begin{theorem}\label{Theorem on representability on DbcohX}
    Let $X$ be a proper scheme over a noetherian ring $R$ such that $\mathbf{K}(\operatorname{Inj}\text{-}X)$ is co-approximable, see \Cref{Theorem Examples of co-approximability} for examples. Let $\mathcal{Y} : \mathbf{D}^{+}_{\operatorname{coh}}(X) \to \Hom{}{\mathbf{D}^b_{\operatorname{coh}}(X)^{\operatorname{op}}}{\operatorname{Mod}(R)}$ be the restricted Yoneda functor. Then,
    \begin{enumerate}[label=(\roman*)]
        \item $\mathcal{Y}$ is full, and the essential image is the cohomological functors $H$ such that $H(\Sigma^{i} F) = 0$ for $i > > 0$ for any $F \in \mathbf{D}^b_{\operatorname{coh}}(X)$. 
        \item $\mathcal{Y}$ gives an equivalence of categories between $\mathbf{D}^{b}(\operatorname{Inj}\text{-}X)$ and the finite cohomological functors, that is, cohomological functors $H$ such that $H(\Sigma^{i} F) = 0$ for $|i| > > 0$ for any $F \in \mathbf{D}^b_{\operatorname{coh}}(X)$. Recall that $\mathbf{D}^{b}(\operatorname{Inj}\text{-}X) = \mathbf{D}^b_{\operatorname{coh}}(X) \cap \mathbf{K}^{b}(\operatorname{Inj}\text{-}X)$.
    \end{enumerate}
\end{theorem}
\begin{proof}
    Recall that a co-approximable triangulated category is an example of a $\mathcal{G}$-approximable triangulated category, see \Cref{Theorem Co-Approximability and G-Approximability}. We have computed the closure of the compacts $\mathsf{T}^+_c$ (see \Cref{Remark 2 on co-approx and G-approx }) and the bounded objects in the closure of the compacts $\mathsf{T}^b_c$ for $\mathsf{T} = \mathbf{K}(\operatorname{Inj}\text{-}X)$  in \Cref{Theorem completion of compacts for for Kinj}. We have that,
    \[\mathbf{K}(\operatorname{Inj}\text{-}X)^+_c \cong \mathbf{D}^{+}_{\operatorname{coh}}(X) \text{ and } \mathbf{K}(\operatorname{Inj}\text{-}X)^b_c \cong \mathbf{D}^{b}(\operatorname{Inj}\text{-}X)\]
    see \Cref{Theorem completion of compacts for for Kinj}(1) and (3).
    
    Then (i) and (ii) follow directly from \Cref{Main Theorem on Tc-} and \Cref{Main Theorem on Tbc} respectively. Note that the properness of the scheme ensures that $\Hom{}{D}{F}$ is a finitely generated $R$-module for any $D,F \in \mathbf{D}^b_{\operatorname{coh}}(X)$. Further, this description of the $\mathcal{G}$-semifinite and $\mathcal{G}$-finite $\mathsf{T}^c$-cohomological functors comes from \Cref{Example finite bounded cohomological functors}(2). 
\end{proof}
\begin{remark}
    If in \Cref{Theorem on representability on DbcohX}, $X$ has a dualizing complex, then there is an equivalence between the category of finite cohomological functors on $\mathbf{D}^b_{\operatorname{coh}}(X)$ and the category $\mathbf{D}^{\operatorname{perf}}(X)$ by \Cref{Theorem Dbinj and Dperf equivalent with dualizing Complex}.
\end{remark}

We have results analogous to \Cref{Theorem on representability on DbcohX} for coherent algebras and stacks with the same proofs using \Cref{Remark Kinj co-approx for stacks and algebras}, which we write down in the remarks below.

\begin{remark}\label{Remark on representability on DbcohA for algebra}
     Let $X$ be a proper scheme over a noetherian ring $R$, and $\mathcal{A}$ an Azumaya algebra such that $\mathbf{K}(\operatorname{Inj}\text{-}\mathcal{A})$ is co-approximable, see \Cref{Theorem Examples of co-approximability} for examples. Let $\mathcal{Y} : \mathbf{D}^{+}_{\operatorname{coh}}(\mathcal{A}) \to \Hom{}{\mathbf{D}^b_{\operatorname{coh}}(\mathcal{A})^{\operatorname{op}}}{\operatorname{Mod}(R)}$ be the restricted Yoneda functor. Then,
    \begin{enumerate}[label=(\roman*)]
        \item $\mathcal{Y}$ is full, and the essential image is the cohomological functors $H$ such that $H(\Sigma^{i} F) = 0$ for $i > > 0$ for any $F \in \mathbf{D}^b_{\operatorname{coh}}(\mathcal{A})$. 
        \item $\mathcal{Y}$ gives an equivalence of categories between $\mathbf{D}^{b}(\operatorname{Inj}\text{-}\mathcal{A})$ and the finite cohomological functors, that is, cohomological functors $H$ such that $H(\Sigma^{i} F) = 0$ for $|i| > > 0$ for any $F \in \mathbf{D}^b_{\operatorname{coh}}(\mathcal{A})$. Recall that $\mathbf{D}^{b}(\operatorname{Inj}\text{-}\mathcal{A}) = \mathbf{D}^b_{\operatorname{coh}}(\mathcal{A}) \cap \mathbf{K}^{b}(\operatorname{Inj}\text{-}\mathcal{A})$.
    \end{enumerate}
\end{remark}

\begin{remark}\label{Remark on representability on DbcohX for stacks}
     Let $\mathcal{X}$ be a proper stack over a noetherian ring $R$ such that $\mathbf{K}(\operatorname{Inj}\text{-}\mathcal{X})$ is co-approximable, see \Cref{Theorem Examples of co-approximability for stacks} for examples. Let $\mathcal{Y} : \mathbf{D}^{+}_{\operatorname{coh}}(\mathcal{X}) \to \Hom{}{\mathbf{D}^b_{\operatorname{coh}}(\mathcal{X})^{\operatorname{op}}}{\operatorname{Mod}(R)}$ be the restricted Yoneda functor. Then,
    \begin{enumerate}[label=(\roman*)]
        \item $\mathcal{Y}$ is full, and the essential image is the cohomological functors $H$ such that $H(\Sigma^{i} F) = 0$ for $i > > 0$ for any $F \in \mathbf{D}^b_{\operatorname{coh}}(\mathcal{X})$. 
        \item $\mathcal{Y}$ gives an equivalence of categories between $\mathbf{D}^{b}(\operatorname{Inj}\text{-}\mathcal{X})$ and the finite cohomological functors, that is, cohomological functors $H$ such that $H(\Sigma^{i} F) = 0$ for $|i| > > 0$ for any $F \in \mathbf{D}^b_{\operatorname{coh}}(\mathcal{X})$. Recall that $\mathbf{D}^{b}(\operatorname{Inj}\text{-}\mathcal{X}) = \mathbf{D}^b_{\operatorname{coh}}(\mathcal{X}) \cap \mathbf{K}^{b}(\operatorname{Inj}\text{-}\mathcal{X})$.
    \end{enumerate}
\end{remark}

Now we state a representability theorem which follows from the co-quasiapproximability of the mock homotopy category of projectives.

\begin{theorem}\label{Theorem on representability on DbcohXop}
    Let $X$ be a proper scheme over a noetherian ring $R$ such that $\mathbf{K}_{m}(\operatorname{Proj}\text{-}X)$ is co-quasiapproximable, see \Cref{Theorem Examples of co-approximability} for examples. We have the restricted Yoneda functor $\mathcal{Y} : \mathbf{D}^{-}(\operatorname{coh}\text{-}X)^{\operatorname{op}} \to \Hom{}{\mathbf{D}^b_{\operatorname{coh}}(X)}{\operatorname{Mod}(R)}$. Then,
    \begin{enumerate}[label=(\roman*)]
        \item $\mathcal{Y}$ is full, and the essential image is the homological functors $H$ such that $H(\Sigma^{i} F) = 0$ for $i > > 0$ for any $F \in \mathbf{D}^b_{\operatorname{coh}}(X)$. 
        \item $\mathcal{Y}$ gives an equivalence of categories between $\mathbf{D}^{\operatorname{perf}}(X)$ and the finite homological functors, that is, homological functors $H$ such that $H(\Sigma^{i} F) = 0$ for $|i| > > 0$ for any $F \in \mathbf{D}^b_{\operatorname{coh}}(X)$. 
    \end{enumerate}
\end{theorem}
\begin{proof}
    Recall that a co-quasiapproximable triangulated category is an example of a $\mathcal{G}$-quasiapproximable triangulated category, see \Cref{Theorem Co-Approximability and G-Approximability}. We have computed the closure of the compacts $\mathsf{T}^+_c$ (see \Cref{Remark 2 on co-approx and G-approx }) and the bounded objects in the closure of the compacts $\mathsf{T}^b_c$ for $\mathsf{T} = \mathbf{K}_{m}(\operatorname{Proj}\text{-}X)$  in \Cref{Theorem completion of compacts for for Kmproj}. We have that,
    \[\mathbf{K}_{m}(\operatorname{Proj}\text{-}X)^+_c \cong \mathbf{D}^{-}_{\operatorname{coh}}(X)^{\operatorname{op}} \text{ and } \mathbf{K}_{m}(\operatorname{Proj}\text{-}X)^b_c \cong \mathbf{D}^{\operatorname{perf}}(X)\]
    see \Cref{Theorem completion of compacts for for Kmproj}(1) and (3).
    
    Then, (i) and (ii) follow directly from \Cref{Main Theorem on Tc-} and \Cref{Main Theorem on Tbc} respectively. Note that the properness of the scheme ensures that $\Hom{}{D}{F}$ is a finitely generated $R$-module for any $D,F \in \mathbf{D}^b_{\operatorname{coh}}(X)$. 
\end{proof}

\begin{remark}
    Note that at the moment, the examples of co-approximable triangulated categories all come via strong generation. The reader may wonder on the relation between \Cref{Theorem on representability on DbcohX} and the representability results in the literature involving strong generation. 
    \begin{enumerate}
        \item One of the earliest and most famous of this form of results is \cite[Theorem 1.3]{Bondal/VanDenBergh:2003}. But, this does not apply here as this category is not Ext-finite in general.
        \item For a noetherian, separated scheme $X$ such that $\mathbf{D}(\operatorname{Qcoh}(X)) = \overline{\langle G \rangle}_N$ for some object $G \in \mathbf{D}^b_{\operatorname{coh}}(X)$ and a positive integer $N$, the conclusion of \Cref{Theorem on representability on DbcohX}(ii) follows from \cite[Theorem 0.3]{Neeman:2018b}.
        \item For a noetherian, separated scheme $X$ with a dualizing complex satisfying the condition $\mathbf{D}(\operatorname{Qcoh}(X)) = \overline{\langle G \rangle}_N$ for some object $G \in \mathbf{D}^b_{\operatorname{coh}}(X)$ and a positive integer $N$, the conclusion of \Cref{Theorem on representability on DbcohX}(ii) follows from \cite[Theorem 0.3]{Neeman:2018b}.
    \end{enumerate}
    
\end{remark}
We now apply the representability theorems to the derived category of the classifying stack of $\mathbb{G}_m$. Note that due to the simple nature of the derived category in this particular case, these result can be shown quite trivially using other methods. The main point of this result is to show we can apply the representability theorems to triangulated categories which do not have a single compact generator. That being said, it is important to note that the definition of $\mathcal{G}$-semifinite and $\mathcal{G}$-finite in this case depends on the generating sequence $\mathcal{G}$. Compare this with \Cref{Example finite bounded cohomological functors}, where these have an intrinsic description.
\begin{theorem}
    Let $\mathsf{T} = \mathbf{D}_{\operatorname{Qcoh}}(B\mathbb{G}_m)$ and denote the restricted Yoneda functor $\mathcal{Y} : \mathsf{T} \to \Hom{}{\mathsf{T}^c}{\operatorname{Mod}(\mathbb{C})}$. We endow $\mathsf{T}$ with the generating sequence $\mathcal{G}$ of \Cref{Remark on Derived category of BG_m}. Then, 
    \begin{enumerate}
        \item $\mathcal{Y}$ is full on the subcategory $\overline{\mathsf{T}^c}$, where, \[\overline{\mathsf{T}^c} = \left\{\bigoplus_{i \in \mathbb{Z}} \bigoplus_{j \in \mathbb{Z}} \Sigma^i G_j^{\oplus n_{ij}} : n_{ij} \in \mathbb{N} \text{ for all i,j}\right\}\]
        from \Cref{Theorem completion for BGm}(1), 
        and the essential image is the $\mathcal{G}$-semifinite cohomological functors. In this case, these are exactly the cohomological functors $H : [\mathsf{T}^c]^{\operatorname{op}} \to \operatorname{Mod}(\mathbb{C})$ such that for any $F \in \mathsf{T}^c$, $H(\Sigma^i F)$ is a finite vector space for all $i \in \mathbb{Z}$.
        \item  $\mathcal{Y}$ gives an equivalence of categories between $\mathsf{T}^c$ and the $\mathcal{G}$-finite cohomological functors on $\mathsf{T}^c$.
    \end{enumerate}
\end{theorem}
\begin{proof}
    We know that $\mathbf{D}_{\operatorname{Qcoh}}(B\mathbb{G}_m)$ is $\mathcal{G}$-approximable from \Cref{Theorem G-approximability for BG_m}. Further, we have computed $\overline{\mathsf{T}^c}$ and $\mathsf{T}^b_c$ for $\mathsf{T} = \mathbf{D}_{\operatorname{Qcoh}}(B\mathbb{G}_m)$ in \Cref{Theorem completion for BGm}. So, (1) and (2) follow from from \Cref{Main Theorem on Tc-} and \Cref{Main Theorem on Tbc} respectively.
\end{proof}

\bibliographystyle{alpha}
\bibliography{mainbib}

\end{document}